\documentclass[11pt]{article}

\usepackage[utf8]{inputenc} 
\usepackage[T1]{fontenc}    
\usepackage{url}
\usepackage{ifthen}
\usepackage{cite}
\usepackage{graphicx,colordvi,psfrag}
\usepackage{amsmath,amssymb,amsthm}
\usepackage{xcolor}
\usepackage{bbding}
\usepackage{bbm}
\usepackage{dsfont}
\usepackage{centernot}
\usepackage{subcaption} 
\usepackage{bm}
\usepackage{hyperref}
\usepackage{authblk}
\usepackage[margin=2.5cm]{geometry}
\usepackage{comment}
\usepackage{booktabs}       
\usepackage{amsfonts}       
\usepackage{nicefrac}       
\usepackage{microtype}      
\usepackage{wrapfig}
\usepackage{enumitem}
\usepackage[linesnumbered,ruled]{algorithm2e}
\usepackage{algorithmic, float}
\usepackage{caption}
\usepackage{tikz,pgfplots}
\hypersetup{
	colorlinks,
	linkcolor={red!40!gray},
	citecolor={blue!40!gray},
	urlcolor={blue!70!gray}
}
\usepackage{fancyhdr}

\newcommand{\boldeta}{\boldsymbol{\eta}}

\newcommand{\0}{\mathbf{0}}

\newcommand{\y}{\mathbf{y}}
\newcommand{\w}{\mathbf{w}}

\renewcommand{\u}{\mathbf{u}}
\renewcommand{\v}{\mathbf{v}}
\newcommand{\X}{\mathbf{X}}

\newcommand{\Y}{\mathbf{Y}}
\newcommand{\N}{\mathbf{N}}

\newcommand{\Id}{\mathbf{I}}

\renewcommand{\epsilon}{\eps}

\newcommand{\Fr}{\mathrm{F}}
\newcommand{\wh}{w}
\newcommand{\tr}{\mathrm{tr}}

\newcommand{\AMSE}{\mathrm{AMSE}}

\newcommand{\R}{\mathbb{R}}

\newcommand{\F}{\mathcal{F}}

\newcommand{\Prob}{\mathbb{P}}
\newcommand{\what}{\widehat}
\newcommand{\wtilde}{\widetilde}

\newcommand{\UpsOne}{\Upsilon_1}
\newcommand{\UpsTwo}{\Upsilon_2}

\newcommand{\Efun}{E}
\newcommand{\res}{\zeta}

\newcommand{\sbar}{\underline{s}}
\newcommand{\stiel}{\overline{s}}
\newcommand{\rhat}{\widehat{r}}
\newcommand{\Rout}{\mathbf{R}'}
\newcommand{\Rin}{\mathbf{R}}
\newcommand{\MPStiel}{\mathsf{m}}

\newcommand{\Zout}{\mathbf{Z}'}
\newcommand{\Zin}{\mathbf{Z}}

\newcommand{\LESDMomInv}{\tau}

\newcommand{\Ind}{\mathds{1}}

\newcommand{\by}{\mathbf{y}}
\newcommand{\bY}{\mathbf{Y}}
\newcommand{\bx}{\mathbf{x}}
\newcommand{\bX}{\mathbf{X}}

\newcommand{\bZ}{\mathbf{Z}}

\newcommand{\bA}{\mathbf{A}}
\newcommand{\bN}{\mathbf{N}}
\newcommand{\bM}{\mathbf{M}}
\newcommand{\bO}{\mathbf{O}}
\newcommand{\ba}{\mathbf{a}}

\newcommand{\bC}{\mathbf{C}}

\newcommand{\bD}{\mathbf{D}}

\newcommand{\bb}{\mathbf{b}}
\newcommand{\bU}{\mathbf{U}}
\newcommand{\bu}{\mathbf{u}}
\newcommand{\bV}{\mathbf{V}}
\newcommand{\bv}{\mathbf{v}}

\newcommand{\bw}{\mathbf{w}}

\newcommand{\bq}{\mathbf{q}}

\newcommand{\be}{\mathbf{e}}
\newcommand{\bE}{\mathbf{E}}
\newcommand{\bR}{\mathbf{R}}
\newcommand{\bI}{\mathbf{I}}
\newcommand{\bP}{\mathbf{P}}

\newcommand{\bW}{\mathbf{W}}
\newcommand{\bg}{\mathbf{g}}

\newcommand{\bS}{\mathbf{S}}
\newcommand{\bTau}{\bm{\m{T}}}
\newcommand{\bF}{\mathbf{F}}
\newcommand{\bLambda}{\bm{\Lambda}}
\newcommand{\balpha}{\boldsymbol{\alpha}}
\newcommand{\bbeta}{\boldsymbol{\beta}}

\newcommand{\RR}{\mathbb{R}}
\newcommand{\CC}{\mathbb{C}}

\newcommand{\diag}{\mathop{\mathrm{diag}}}

\newcommand{\rank}{\mathop{\mathrm{rank}}}

\newcommand{\trace}{\mathop{\mathrm{tr}}}

\newcommand{\Expt}{\mathbb{E}}
\newcommand{\eps}{\varepsilon}

\newcommand{\m}{\mathcal}
\newcommand{\T}{\top}

\DeclareMathOperator*{\argmin}{\arg\!\min}

\newcommand{\trlim}{\overline{\mathrm{tr}}}
\newcommand{\asympeq}{\simeq}

\newcommand{\sigmaBulk}{{\theta}_{\mathrm{max}}}
\newcommand{\sigmaMin}{{\theta}_{\mathrm{min}}}
\newcommand{\sigmaThresh}{\sigma_{\mathrm{thresh}}}
\newcommand{\bSigma}{{\bm{\Sigma}}}
\newcommand{\hbSigma}{\widehat{\bSigma}}
\newcommand{\sigmaBBP}{{\sigma}_{\mathrm{BBP}}}
\newcommand{\thetaBBP}{{\theta}_{\mathrm{BBP}}}

\newcommand{\yval}{\widehat{\theta}}

\newcommand{\hbu}{\widehat{\bu}}
\newcommand{\hbv}{\widehat{\bv}}

\newcommand{\Revision}[1]{{#1}}

\theoremstyle{thmstyletwo}%
\newtheorem{theorem}{Theorem}
\newtheorem{proposition}[theorem]{Proposition}%

\newtheorem{lemma}[theorem]{Lemma}

\theoremstyle{definition}
\newtheorem{definition}{Definition}

\theoremstyle{remark}
\newtheorem{remark}{Remark}%

\numberwithin{equation}{section}
\newcommand{\longtitle}{Matrix Denoising with 
Partial Noise Statistics:\\ Optimal Singular Value Shrinkage of Spiked F-Matrices}

\newcommand{\figurepath}{{Fig/}}

\title{\longtitle}

\graphicspath{\figurepath}

\author[1]{Matan Gavish\thanks{\texttt{gavish@cs.huji.ac.il}}}
\author[2]{William Leeb\thanks{\texttt{wleeb@umn.edu}}}
\author[3]{Elad Romanov\thanks{\texttt{elad.romanov@gmail.com}}}

\affil[1]{School of Computer Science and Engineering, Hebrew University of Jerusalem}
\affil[2]{School of Mathematics, University of Minnesota}
\affil[3]{Department of Statistics, Stanford University}
\date{}

\begin{document}

    \maketitle

    \begin{abstract}
        We study the problem of estimating a large, low-rank matrix corrupted by additive noise of unknown covariance, assuming one has access to additional side information in the form of noise-only measurements. We study the Whiten-Shrink-reColor (WSC) workflow, where a “noise covariance whitening” transformation is applied to the observations, followed by appropriate singular value shrinkage and a “noise covariance re-coloring” transformation. We show that under the mean square error loss, a unique, asymptotically optimal shrinkage nonlinearity exists for the WSC denoising workflow, and calculate it in closed form. To this end, we calculate the asymptotic eigenvector rotation of the random spiked F-matrix ensemble, a result which may be of independent interest. With sufficiently many pure-noise measurements, our optimally-tuned WSC denoising workflow outperforms, in mean square error, matrix denoising algorithms based on optimal singular value shrinkage which do not make similar use of noise-only side information; numerical experiments show that our procedure's relative performance is  particularly strong in challenging statistical settings with high dimensionality and large degree of heteroscedasticity.
    \end{abstract}
    
    \section{Introduction}


Low-rank matrix reconstruction from partial or corrupted measurements is a well-studied problem in machine learning and statistics, with applications ranging from  computer vision \cite{moore2014improved} and
structural biology \cite{anden2018structural, bhamre2016denoising} 
to 
medical imaging \cite{cordero2019complex} and medical signal processing \cite{liu2021denoising}. This paper considers 
recovery of a $p$-by-$n$ matrix $\bX$ of rank $r \ll n,p$ from an additively corrupted measured matrix $\bY=\bX+\Rin$. Here, $\Rin$ is a noise matrix (independent of $\bX$) whose columns are assumed independent and identically-distributed (i.i.d.), with an arbitrary between-row correlation structure $\bSigma$. 

Such matrix denoising problems occur, for example,  
in principal component analysis (PCA) under a low-rank factor model \cite{tipping1999probabilistic}. Consider  $n$ 
data points in $p$-dimensional Euclidean space, denoted   $\by_1,\ldots,\by_n \in \RR^p$ of the form $\by_i = \bx_i + \bm{\eps}_i$. The i.i.d. ``signal'' vectors $\bx_1,\ldots,\bx_n$ are 
assumed to lie on some unknown $r$-dimensional subspace, and the i.i.d. ``noise'' vectors $\bm{\eps}_1,\ldots,\bm{\eps}_n$  have a full rank covariance  matrix $\bSigma$, and so spread over the entire ambient space. One would like to reconstruct the low dimensional signal vectors from the noisy observations $\by_1,\ldots,\by_n$. Let $\bY \in \RR^{p\times n}$ be the data matrix, formed by stacking the observations $\by_1,\ldots,\by_n$ as its columns, so that $\bY=\bX+\Rin$, where $\bX\in \RR^{p\times n}$ is a matrix whose columns are $\bx_1,\ldots,\bx_n$ and $\Rin$ is a noise matrix. This signal estimation problem becomes a matrix denoising problem: estimate $\bX$ from $\bY$.

This paper approaches the matrix denoising problem in the well-known \emph{spiked model} \cite{johnstone2001distribution}, wherein the matrix dimensions $p,n\to\infty$ with a {limiting aspect ratio $\gamma:= \lim_{p,n \to \infty} p/n > 0$} while the signal rank $r=\rank(\bX)$ is fixed. The spiked model captures the key features of the matrix denoising problem in a regime where both  underlying signal rank and  signal-to-noise ratio (SNR) are small; in particular, the ground truth $\bX$ is not consistently estimable from $\bY$ 
(see Section~\ref{sec:setup} below).

From the perspective of random matrix theory, the spiked model has relatively simple and well-understood asymptotic behavior  \cite{baik2005phase,baik2006eigenvalues,paul2007asymptotics,benaych2012singular}. Specifically, the behavior of $\bY$ is described by the following phenomena:
\begin{enumerate}
    \item \emph{Singular value displacement:} The singular values of $\bY$ are divided into a ``bulk'', which has a deterministic limiting shape and corresponds to pure noise, and at most $r$ ``outliers'' which exceed the bulk and correspond to ``signal''. The locations of the outliers are {asymptotically} deterministic, {and} the $i$-th largest singular value of $\bY$ depends only on the $i$-th largest singular value of $\bX$. The presence (or lack thereof) of an outlier is a threshold phenomenon: there is some SNR level $\sigma^*$, a detection threshold, such that the $i$-th singular value $\sigma_i$ creates an observable outlier if and only if $\sigma_i > \sigma^*$.
    \item \emph{Principal component angles:} The angles between the leading observed PCs and the signal PCs are essentially deterministic. The $i$-th signal PC is essentially orthogonal to the the $j$-th ($j\ne i$) observed PC. The angle between the $i$-th signal and $i$-th observed PCs concentrates around a deterministic number $\in [0,1)$, which may be consistently estimated from the observed $i$-th singular value of $\bY$.
\end{enumerate}

A popular and practical approach to the matrix denoising problem is \emph{singular value shrinkage}, where $\bX$ is estimated by taking the singular value decomposition (SVD) of $\bY$, retaining its singular vectors while systematically deflating the singular values to correct for the noise \cite{perry2009cross,shabalin2013reconstruction,donoho2014minimax,donoho2018optimal,gavish2017optimal,nadakuditi2014optshrink}. 
Leveraging the above spiked model asymptotic phenomena, which are entirely quantifiable, several authors have derived optimal singular value shrinkers under various settings, cf. \cite{shabalin2013reconstruction,nadakuditi2014optshrink,gavish2014optimal,gavish2017optimal,donoho2018optimal,hong2018asymptotic,hong2018optimally,leeb2021optimal,leeb2021matrix,leeb2022operator,donoho2020screenot,su2022optimal,dobriban2020optimal}.

{For isotropic noise ($\bSigma=\bI$)}, matrix denoising in the spiked model is well-understood.  However, the general case of {heteroscedastic} noise offers interesting problems of practical interest. Several recent works have studied PCA, singular value shrinkage and related spectral methods in the presence of heteroscedastic noise, either across rows, columns or both; see for example \cite{nadakuditi2014optshrink,donoho2020screenot,leeb2021optimal,leeb2021matrix,ding2021spiked,hong2021heppcat,liu2018pca,landa2021biwhitening,zhang2022heteroskedastic,hong2018optimally,hong2018asymptotic,behne2022fundamental,agterberg2022entrywise,su2022optimal}. Under our present setting of independent columns and inter-row {covariance} matrix $\bSigma$,  formulas for the optimal singular value shrinker are available: when the noise covariance $\bSigma$ is known, the optimal shrinker can be computed exactly, and when it is not, one can consistently estimate the optimal shrinkage rule from the observed spectrum of $\bY$; the resulting procedure is known as \emph{OptShrink} \cite{nadakuditi2014optshrink} (see also \cite{donoho2020screenot}).

When denoising a signal sampled with additive heteroscedastic noise, one sometimes has the opportunity to sample the noisy channel in the absence of any signal. A natural approach -- indeed a classical idea in signal estimation -- is  to use noise-only measurements to ``whiten'' the measurements. According to this approach, one should (i) ``whiten'' the data, that is multiply by $\bSigma^{-1/2}$ (assuming $\bSigma$ is somehow available); (ii) apply an estimation procedure calibrated for uncorrelated noise, yielding an estimator $\widehat{\bX}^w$; and (iii)  apply a ``recoloring'' transformation $\widehat{\bX}=\bSigma^{1/2}\widehat{\bX}^w$. The popularity of this approach in signal processing is due both to its conceptual simplicity, and to the ubiquity of linear filtering, under which it is often optimal. The influential textbook of van Trees \cite{van2004detection} advocates, for example:
    {
    \it
    ``Many of our models assumed the received signal, either a waveform or a vector, was
    observed [in] the presence of ``white noise'' [...]
    We demonstrated, first with
    vectors and then with waveforms, that we could always find a ``whitening transformation'' that
    mapped the original process into a signal plus white noise problem.
    The reader should remember to consider this approach when dealing with more general
    problems.''
    }
(\cite[Epilogue]{van2004detection}). 

 In our present denoising problem, under the spiked model, it is natural to consider a
{\it \textbf{W}hiten-\textbf{S}hrink-re\textbf{C}olor} (WSC) workflow, where a ``signal covariance whitening'' transformation is applied to the observations, followed by appropriate singular value shrinkage and a ``signal covariance re-coloring'' transformation.
The simplest scenario where WSC may be considered is when
 the noise (population)
 covariance $\bSigma$ is available as side information. 
 This case was studied recently by two of the authors, who derived the optimal shrinkage rule of the WSC procedure \cite{leeb2021optimal}. As one can expect, denoising performance were shown to improve by incorporating the side information $\bSigma$ into the WSC workflow, compared with optimal singular value shrinkage \cite{nadakuditi2014optshrink}; interestingly, the optimal shrinker for WSC was found to be different from the singular value shrinker optimally tuned for white noise. 

Clearly, the assumption that the noise covariance $\bSigma$ is known or consistently estimable is typically unrealistic  in high dimensions \cite{cai2016estimating}. As \cite{leeb2021optimal} demonstrated,
complete knowledge of the noise covariance offers significant improvement for matrix denoising  under heteroscedastic noise. One then naturally  wonders whether {\em partial} information of the noise $\Rin$ could be similarly leveraged. Specifically, assume access to side information in the form of {$m > p$} \emph{pure-noise} samples. 
Following \cite{leeb2021optimal},  the present paper considers matrix denoising in the spiked model under a WSC workflow, where instead of the noise population covariance, ``whitening'' and ``recoloring'' are done using a sample covariance matrix of pure-noise samples $\hbSigma$.  
This problem is fascinating in part owing to the fact that, as $\hbSigma$ is an inconsistent estimate of $\bSigma$, whitening by $\hbSigma$ injects additional ``noise'' into the estimation procedure,  which should be systematically corrected for. 

\paragraph*{Contributions.} 
 \begin{enumerate}
\item {\bf Optimal WSC denoiser.} Our main contribution is the derivation of the optimal WSC denoiser in mean square error. Specifically, we show that an asymptotically optimal shrinker exists for this WSC workflow, and derive it in closed form.
\item {\bf Asymptotic singular vector rotations for the spiked F-matrix.} Curiously,  the whitened data matrix 
\[
\bY^{w}=\hbSigma^{-1/2}\bY = \hbSigma^{-1/2}\bX + \hbSigma^{-1/2}\Rin \,.
\]
is an object of independent interest known in the random matrix theory literature as a  \emph{spiked F-matrix} \cite{bai2009book,nadakuditi2010fundamental}.
Prior works in signal processing and statistics have studied the problem of signal detection under spiked F-matrix ensemble, focusing on the behavior of the largest eigenvalues in the presence or absence of a signal  \cite{zhao1986detection,zhu1991estimating,stoica1997detection,nadakuditi2010fundamental,johnstone2017roy,dharmawansa2014local,johnstone2015testing}. In particular, it was shown \cite{nadakuditi2010fundamental} that the singular values of $\bY^w$ exhibit the same basic phenomenology similar to the ``classical'' spiked model:\footnote{Note that $\bY^w$ does \emph{not} follow a generalized spiked model in the sense of e.g. \cite{benaych2012singular}, since the low-rank and noise parts are statistically dependent on one another through their mutual dependence on $\hbSigma$. } its singular values are arranged in the form of a ``bulk'' plus at most $r$ outliers exceeding the bulk. (The limiting distribution of the bulk was calculated already in the classical work of Wachter \cite{wachter1980limiting}.) Formulas for the spike detection threshold, as well as the spike-forward map (singular value displacements) have also been computed.  As discussed above, 
derivation of optimal shrinkers for WSC worflow in our scenario requires calculation of the limiting angles between the population (signal) principal components and their empirical counterparts. 
A secondary contribution of the present paper is closed-form formulas for the limiting angles of the spiked F-matrix ensemble.
\end{enumerate}

\paragraph{Paper outline.}
This paper is organized as follows. In Section \ref{sec:setup}, we state the precise  observation model and estimation problem; provide key definitions of functions; and describe the proposed denoising algorithm in detail. In Section \ref{sec:main_results}, we state the theoretical results on spiked F-matrices, and show how these may be used to derive the optimal denoisers.
In Section~\ref{sec:background-F} we briefly survey several known results on the properties of the F-matrix ensemble, which shall be crucial in the derivation to follow.
Section \ref{sec:proofs} is devoted to the proofs of our technical results, 
with some details deferred to the Appendix. Lastly, in Section \ref{sec:numerical} we report on numerical experiments illustrating the behavior of the proposed denoising method. In particular, we numerically compare the performance of our method to that of optimal singular value shrinkage (OptShrink \cite{nadakuditi2014optshrink}) under different model configurations.
Section~\ref{sec:conclusion} is devoted to conclusion and some additional discussion.

\section{Notation and problem setup}
\label{sec:setup}

\subsection{Observation model}
Let $\sigma_1,\ldots,\sigma_r>0$, be positive numbers, and $\bu_1,\ldots,\bu_r\in \RR^{p}$ and $\bv_1,\ldots,\bv_r \in \RR^n$ be vectors. Denote by $\bU \in \RR^{p\times r}$ the matrix whose columns are $\bu_1,\ldots,\bu_r$, and similarly for $\bV\in\RR^{n\times r}$. Also, let $\bLambda \in \RR^{r\times r}$ be a diagonal matrix with diagonal given by $(\sigma_1,\ldots,\sigma_r)$. The signal matrix, the object to be estimated, is
\begin{align}
\label{eq:signal}
\bX = \sum_{i=1}^r \sigma_i \bu_i \bv_i^\T = \bU\bLambda\bV^\T \,. 
\end{align}
Clearly, $\rank(\bX)\le r$. 
Let $\Zin \in \RR^{p\times n}$ be a random matrix, independent of $\bX$, with i.i.d. Gaussian entries $Z_{i,j}\sim \m{N}(0,1)$, and let $\Rin = \bSigma^{1/2}\Zin$, where $\bSigma\in \RR^{p\times p}$ is positive definite. One observes $\bY\in \RR^{p\times n}$, 
\begin{align}
\label{eq:Y}
\bY = \bX + \frac{1}{\sqrt{n}}\Rin = \bU\bLambda\bV^\T + \frac{1}{\sqrt{n}}\bSigma^{1/2}\Zin \,.
\end{align}
That is, each column of $\bX$ is corrupted by additive noise of mean $\bm{0}$ and covariance $\bSigma/n$. 
We remark that this normalization (dividing the noise by $\sqrt{n}$) is such that the singular values of the signal $\bX$ and the noise $n^{-1/2}\Rin$ are of the same order, cf. \cite{shabalin2013reconstruction,nadakuditi2014optshrink,gavish2017optimal,donoho2018optimal,leeb2021optimal,donoho2020screenot}.

The noise covariance $\bSigma$ is assumed to be unknown. Instead, one is given side information in the form of $m$ pure-noise samples, which are independent of the measurement matrix $\bY$. That is, 
one observes a noise-only matrix
%
\begin{align}
\label{eq:out-of-sample-noise}
\Rout = \bSigma^{1/2} \Zout \;\in\RR^{p\times m}\,,\qquad \Zout_{i,j}\overset{i.i.d.}{\sim} \m{N}(0,1) \,.
\end{align}
We will assume that $m\ge p$ so that $\rank(\Rout)=\rank(\bSigma^{1/2})$ with probability (w.p.) $1$. 
Having observed the signal-plus-noise measurement matrix $\bY$ and the side information $\Rout$, our goal is to estimate the signal matrix $\bX$. For an estimator $\widehat{\bX}=\widehat{\bX}(\bY,\Rout)$, we measure its error using the Frobenius loss (MSE): $\Expt \|\bX-\widehat{\bX}\|_F^2$.  
The matrices in the observation model are summarized in Table \ref{table:matrices}.

\begin{table}
\centering
\begin{tabular}{| c | c  | c | c  |}
\hline  
 Symbol(s) &  Description & Size(s)  & Observed?  \\
\hline
$\bU$, $\bV$, $\bLambda$ & Signal SVD & $p \times r$, $p \times n$, $r \times r$ &  Not observed \\
\hline
$\bX$ & Signal $\bU \bLambda \bV^\top$ & $p \times n$   & Not observed \\
\hline
$\Zin$, $\Zout$ & White noise & $p \times n$, $p \times m$  & Not observed \\
\hline
$\bSigma$ & Noise covariance & $p \times p$&  Not observed \\
\hline
$\Rin$ & Noise $\bSigma^{1/2} \Zin$ & $p \times n$ & Not observed \\
\hline
$\Rout$ & Out-of-sample noise $\bSigma^{1/2} \Zout$ & $p \times m$ &  Observed \\
\hline
$\Y$ & Signal-plus-noise $\bX + \Rin$ & $p \times n$&  Observed \\
\hline
$\what\bU$, $\what \bV$, $\what \bLambda$ & $r$-SVD of $\Y$ & $p \times r$, $p \times n$, $r \times r$ 
    &  Observed \\
\hline
$\what \bSigma$ & Sample covariance $\Rout \Rout^\top/m$ & $p \times p$ &  Observed \\
\hline
$\bN$ & Pseudo-whitened noise $\what \bSigma^{-1/2} \Rin / \sqrt{n}$ & $p \times m$ &  Not observed \\
\hline
$\bE$ & Wishart matrix $\Zin \Zin^\T / n$ & $p \times p$ &  Not observed \\
\hline
$\bS$ & Wishart matrix $\Zout \Zout^\T / m$ & $p \times p$ &  Not observed \\
\hline
$\bD_1,\dots,\bD_r$ & Signal PC weights & $p \times p$ &  Not observed \\
\hline
\end{tabular}
\caption{Matrices used in this paper.}
\label{table:matrices}
\end{table}  

We study this denoising problem under the spiked model \cite{johnstone2001distribution,benaych2012singular}. Formally, we consider
a sequence of denoising problems, $n,m,p\to\infty$, 
with the following specifications:
\begin{enumerate}
    \item \textbf{``High-dimensional'' asymptotics:} 
    For constants $\gamma\in (0,\infty)$ and $\beta\in (0,1)$,
\begin{align}
        \frac{p}{n} \to \gamma,\quad \frac{p}{m} \to \beta, \quad\textrm{ as }\quad n,m,p\to\infty \,.
\end{align}

    \item \label{assum:cov} \textbf{Regularity of noise covariance sequence:} 
    \begin{enumerate}
        \item The empirical spectral distribution (ESD) of $\bSigma$ converges weakly almost surely to some deterministic, \emph{compactly supported} law $dH$. To wit, for every bounded continuous\footnote{\Revision{By way of notation, for univariate $f(\cdot)$ and symmetric $\bSigma$ we define $f(\bSigma):=\bm{E}\diag(f(\lambda_1),\ldots,f(\lambda_p))\bm{E}^\T$, where $\bSigma:=\bm{E}\diag(\lambda_1,\ldots,\lambda_p)\bm{E}^\T$ is the eigendecomposition. (This is the ``standard'' functional calculus for symmetric matrices.) }}
        $f(\cdot)$, $p^{-1}\tr(f(\bSigma)):=p^{-1}\sum_{i=1}^p f(\lambda_i(\bSigma)) \longrightarrow \int f(\lambda)dH(\lambda)$ a.s. as $p\to\infty$.
        \item The extremal eigenvalues of $\bSigma$ converge to the edges of the support of $dH$. To wit, if $\mathrm{supp}(dH)=[a,b]$ then $\lambda_{\min}(\bSigma)\to a$, $\lambda_{\max}(\bSigma)\to b$. We further require that $a>0$.
    \end{enumerate}
    We denote the first moment of the limiting empirical spectral distribution (LESD) by
    \begin{equation}
        \label{eq:first-moment-H}
            \mu := \lim_{p\to\infty} p^{-1} \trace(\bSigma) = \int_a^b \lambda dH(\lambda) \,.
    \end{equation}
    
    \item \textbf{Generative assumptions on signal:}
    The number of spikes $r$ and the intensities $\sigma_1,\ldots,\sigma_r>0$ are fixed as $n,m,p\to\infty$. The signal principal directions satisfy the following:
    \begin{enumerate}
        \item {\bf Right principal directions:} The matrix $\bV\in \RR^{n\times r}$ is such that $\bV^\T \bV \to \Id_{r\times r}$ almost surely. In other words, $\{\bv_1,\ldots,\bv_r\}$ are (asymptotically) orthonormal vectors.
        \item {\bf Left principal directions:} 
        The vectors $\bu_1,\ldots,\bu_r$ have the form\footnote{The Gaussianity of the vectors $\bw_i$ is not strictly necessary; one could replace it by any other isotropic distribution of sufficiently light tail.  }
        \begin{equation}\label{eq:ui-def}
            \bu_i = p^{-1/2}\bD_i \bw_i\,,\quad\textrm{for}\quad \bw_1,\ldots,\bw_r \sim \m{N}(\0,\bI_{p\times p}) \,.
        \end{equation}
        Furthermore, the (sequences of) matrices
        $\bD_i$ satisfy: 
        \begin{enumerate}
            \item \Revision{{\it Boundedness in operator norm:}} For some $C>0$, $\max_{1\le i \le r}\|\bD_i\|<C$ almost surely.
            \item {\it Unit energy:} $\lim_{p\to\infty}p^{-1}\|\bD_i\|_F^2=1$.
            \item {\it Limiting joint law:}  The algebra generated by $\{\bD_i\bD_i^\T,\bSigma,\bSigma^{-1}\}$ has a limiting joint law in the sense of free probability theory (see Section~\ref{sec:free-probability}).
        \end{enumerate}
        The following quantities will play an important role in our formulas:
        \begin{equation}\label{eq:tau-def}
            \LESDMomInv_i := \lim_{p\to\infty} p^{-1}\tr\left( \bD_i^\T \bSigma^{-1}\bD_i \right) {\,, \quad  1 \le i \le r.}
        \end{equation}
We assume that each $\tau_i$ is finite and strictly positive.
        \item {\bf Simple signal spectrum:} The following ``effective'' spike intensities contain no multiplicities. By way of notation, they are ordered as $\sqrt{\LESDMomInv_1} \sigma_1 > \ldots > \sqrt{\LESDMomInv_r} \sigma_r$. We remark that this assumption is standard throughout the literature on singular value shrinkage, see for example \cite{shabalin2013reconstruction,nadakuditi2014optshrink,gavish2017optimal,leeb2021optimal,donoho2020screenot}. 
    \end{enumerate}
\end{enumerate}
    
    \begin{remark}
        Much of the existing literature on singular value shrinkage either assumes isotropic noise ($\bSigma=\bI_{p\times p})$ or an isotropic prior ($\bD_i=\bI_{p\times p}$) on the left signal directions $\bu_1,\ldots,\bu_r$ (cf. \cite{shabalin2013reconstruction,nadakuditi2014optshrink,gavish2017optimal,donoho2020screenot}). Similar to \cite{leeb2021optimal}, we relax this assumption by allowing in (\ref{eq:ui-def}) for an alignment between the signal direction and the noise, whose ``strength'' is captured by the parameter $\LESDMomInv_i$ from (\ref{eq:tau-def}). Note that (\ref{eq:ui-def}) implies that $\bu_i^\T \bu_j\to \Ind\{i=j\}$ a.s., thus the $\bu_i$-s may be interpreted as the principal components of $\bX$ (though this statement is only precise asymptotically). In the language of factor analysis, the $\bv_i$-s may be interpreted as the factor values \cite{anderson1984, anderson2003introduction, schervish1987review, dobriban2017factor}.
    \end{remark}

    \begin{remark}
        Throughout the paper we assume that $\Zin,\Zout$ have Gaussian entries; this assumption is used explicitly in the proofs (specifically, the bi-orthogonal invariance of these matrices). In Section~\ref{sec:numerical} we give numerical evidence indicating that our results should be universal with respect to the noise distribution: they continue to hold when $\Zin,\Zout$ are i.i.d. with sufficiently light-tailed isotropic entries. 
    \end{remark}

\subsection{Whiten-Shrink-reColor (WSC) denoisers}

Our task is to estimate the signal matrix $\bX$ from the observed signal-plus-noise matrix $\bY = \bX + \Rin$ and the noise-only samples $\Rout$. We next describe the class of procedures we consider for this problem. 

Let $\hbSigma$ be an estimate of the covariance matrix $\bSigma$. Later, we will take $\hbSigma = \Rout\Rout^\T / m$, but for now any estimate would suffice. We use $\what \bSigma$ to \emph{pseudo-whiten} the noise on the observation matrix, constructing a new matrix $\bY^w$:
\begin{align}\label{eq:Yw-def}
\bY^w = \what{\bSigma}^{-1/2} \bY = \what{\bSigma}^{-1/2}\bU\bLambda\bV^\T + \frac{1}{\sqrt{n}}(\what{\bSigma}^{-1/2}\bSigma^{1/2})\Zin \,.
\end{align}
We consider the following family of estimators $\F$, computed from the SVD of $\bY^w$:
\begin{align}\label{eq:family-F}
\bY^w \overset{\mathrm{SVD}}{=} \sum_{k=1}^{\min\{p,n\}} \yval_k \what \bu_k^w \what \bv_k^w,\quad
\F = \left\{\sum_{k=1}^{\rhat} \eta_k \what \bSigma^{1/2} \what \bu_k^w \what \bv_k^w 
    : \eta_1,\dots,\eta_{\rhat} \in \R \right\},
\end{align}
where $\rhat$ denotes a data-driven estimator of $\rank(\bX)$, to be described in Section \ref{sec.estimator}.
Let $\what \bX^{\boldsymbol\eta} = \sum_{k=1}^{\rhat} \eta_k \what \bSigma^{1/2} \what \bu_k^w \what \bv_k^w $, where $\boldsymbol \eta = (\eta_1,\dots,\eta_{\hat r})$. It will be shown later (in Section \ref{sec:derivation_shrinker}) that for any deterministic $\boldeta$, the asymptotic loss
\begin{align}
\AMSE(\boldsymbol \eta) = \lim_{p \to \infty} \| \bX - \what \bX^{\boldeta} \|_{\Fr}^2
\end{align}
almost surely exists, and is finite. The goal, then, is to find $\eta_1,\dots,\eta_r$ so to minimize the asymptotic loss:
\begin{align}
\boldeta = \argmin_{\wtilde \boldeta} \AMSE(\wtilde \boldeta).
\end{align}
%
We will derive the optimal choice of $\eta_1,\dots,\eta_{\rhat}$ in Section \ref{sec:derivation_shrinker}.


\begin{remark}
The values $\eta_1,\dots,\eta_{\what r}$ are known as the generalized singular values of $\what \bX$ with respect to the matrix $\hbSigma^{-1}$; correspondingly, the vectors $\hbSigma^{1/2}\hbu_1^w,\dots,\hbSigma^{1/2}\hbu_r^w$, $\hbv_1^w,\dots,\hbv_r^w$ are the generalized singular vectors \cite{vanloan1976generalizing}.
That is, 
$\hbSigma^{1/2}\hbu_1^w,\dots,\hbSigma^{1/2}\hbu_r^w$ are orthonormal with respect to the weighted inner product $\langle \bx, \tilde{\bx} \rangle_{\hbSigma^{-1}} = \bx^\T \hbSigma^{-1} \tilde{\bx}$.
\end{remark}

An equivalent formulation of the estimator, which may be more natural, is given as follows. Define
\begin{align}
\widehat{\bu}_k
    =\frac{\what{\bSigma}^{1/2}\what{\bu}^{w}_k}{\|\what{\bSigma}^{1/2}\what{\bu}^{w}_k\|},\quad 
    \widehat{\bv}_k = \widehat{\bv}_k^w
\quad 1 \le k \le r,
\end{align}
so that 
we may write the estimator in the form
\begin{align}
\label{eq:approximate_SVD}
\what \bX = \sum_{k=1}^{r} t_k \what \bu_k \what \bv_k,
\quad\textrm{where} \quad
t_k = \|\what{\bSigma}^{1/2}\what{\bu}^{w}_k\| \cdot \eta_k .
\end{align}
As we will show in Theorem \ref{thm:products}, the vectors $\what \bu_1,\dots,\what \bu_r$ are asymptotically orthonormal as $p,n,m \to \infty$; consequently, equation \eqref{eq:approximate_SVD} is an approximate SVD of $\what \bX$, with approximate singular values $t_1,\dots,t_r$.
The vectors $\hbu_1,\dots,\hbu_r$ may be interpreted as estimates of the  population principal components, $\bu_1,\dots,\bu_r$. Accordingly, we will refer to $\hbu_1,\dots,\hbu_r$ as the \emph{empirical principal components}.

\textbf{Rationale for the estimation procedure.}
The estimator family $\F$ in (\ref{eq:family-F}) was studied in the authors' previous work \cite{leeb2021optimal} when the estimator $\what \bSigma$ of $\bSigma$ is asymptotically consistent in operator norm as $p,n,m \to \infty$; in the setting of the present paper, this occurs when, {for example, $\hbSigma$ is a sample covariance and} $p/m \to 0$. In this setting, it is shown that under a uniform prior on the singular vectors of $\bX$, the estimator $\what \X$ outperforms the optimal singular value shrinkage estimator (OptShrink) described in \cite{nadakuditi2014optshrink}. It follows that so long as $\what \bSigma$ is sufficiently close to $\bSigma$, the optimal $\what\bX$ will outperform OptShrink as well.

\subsection{Key definitions}

\label{sec:key_definitions}

We introduce several parameters and functions that will be used to evaluate the optimal $\eta_1,\dots,\eta_r$. \Revision{Their significance will be explained in Section \ref{sec:background-F}; for now, we simply present 
the relevant mathematical formulae.
} The key parameters and functions, along with others that are introduced later in the paper, are summarized in Tables \ref{table:functions} and \ref{table:parameters}, respectively.

Define the following values:
\begin{align}
\label{eq:sigma-thresh}
    \sigmaThresh = \sqrt{\frac{\beta+\sqrt{\beta+\gamma-\beta\gamma}}{1-\beta}} \,,
\end{align}
and
\begin{align}
\label{eq:sigma-bulk-min}
    \sigmaBulk = \frac{1+\sqrt{\beta+\gamma-\beta\gamma}}{1-\beta} \,,\quad 
    \sigmaMin = \frac{1-\sqrt{\beta+\gamma-\beta\gamma}}{1-\beta} \,. 
\end{align}
The quantities $\sigmaBulk,\sigmaMin$ are the almost-sure limits of, respectively the largest and smallest non-zero singular values of $\widehat{\bSigma}^{-1/2}\bY$ in the pure-noise case, that is, when there are no spikes. The value $\sigmaThresh$ is the smallest  singular value of $\bX$ that can be reliably detected as $p \to \infty$  in the case where $\bSigma = \bI_p$ . We discuss this in more detail in Section~\ref{sec:main_results-theory} below.


On the ray $(\sigmaBulk^2,\infty) \subset \RR$, we define the  \emph{Stieltjes transform of Wachter's distribution} as follows:
\begin{align}\label{eq:stieltjes-def}
\stiel(z) = \frac{1}{\gamma z} - \frac{1}{z} - \frac{\gamma\left(z(1-\beta)+(1-\gamma)\right)+2\beta z-\gamma\sqrt{\left(z(1-\beta) + (1-\gamma)\right)^2-4z}}{2\gamma z(\gamma+\beta z)}.
\end{align}
We also define the \emph{associated Stieltjes transform of Wachter's distribution} as follows:
\begin{align}\label{eq:stieltjes-assoc-def}
\sbar(z)= -\frac{\gamma\left(z(1-\beta)+(1-\gamma)\right)+2\beta z-\gamma\sqrt{\left(z(1-\beta) + (1-\gamma)\right)^2-4z}}{2 z(\gamma+\beta z)}.
\end{align}

We also define a related function on the same domain $(\sigmaBulk^2,\infty)$:
\begin{align}
\label{eq:res-def}
\res(z) = - \frac{z(1-\beta)-(1-\gamma)-\sqrt{\left((1-\beta)z+(1-\gamma)\right)^2-4z}}
    {2(\gamma+\beta-\gamma\beta)z} \,.
\end{align}
%

It is also convenient to define
\begin{align}
\label{eq:psi}
\psi(z)=z \cdot \sbar(z) \cdot \res(z)
=\frac{(1-\beta)z-1-\gamma-\sqrt{\left(z(1-\beta) + (1-\gamma)\right)^2-4z}}{2(\beta z + \gamma    )} \,,
\end{align}
and the function
\begin{align}
\varphi(z) = z\cdot |\sbar(z)| \cdot \left( \stiel(z)\right)^2 \,.
\end{align}


We also define the \emph{spike-forward} map $\Xi:[\sigmaThresh,\infty) \to [\sigmaBulk,\infty)$ by
\begin{align}
\label{eq:spike-forward-in-thm}
    \Xi(\sigma) = \sqrt{ \frac{(1+\sigma^2)(\gamma + \sigma^2)}{(1-\beta)\sigma^2-\beta} } \,.
\end{align}
Note that $\Xi(\cdot)$ is strictly increasing and smooth. Its inverse 
$\Xi^{-1}\,:\,[\sigmaBulk,\infty) \to [\sigmaThresh,\infty)$ is
\begin{align}
\Xi^{-1}(\theta) = \frac{1}{\sqrt{\psi(\theta^2)}} \,,
\end{align}
where $\psi(z)$ is defined in \eqref{eq:psi}.

Let $\MPStiel_\beta(\cdot)$ be the Stieltjes transform of the Marchenko-Pastur law with shape parameter $\beta$. {When $0<z<(1-\sqrt{\beta})^2$, namely it is smaller than the left edge edge of the Marchenko-Pastur bulk, it is given by the following formula}, cf. \cite[Lemma 3.11]{bai2009book}:
\begin{equation}
    \label{eq:MPStiel-def}
    \MPStiel_\beta(z) = \frac{1-\beta-z - \sqrt{(z-1-\beta)^2-4\beta }}{2\beta z} \,.
\end{equation}
Denote the functions $\UpsOne(\cdot),\UpsTwo(\cdot)$,  $z\in (\sigmaBulk^2,\infty)$,
\begin{align}
\label{eq:New:UpsOneFormula}
    &\UpsOne(z) = \frac{1}{z}\left[ 1 - \sbar(z) + (\sbar(z))^2\MPStiel_\beta(-\sbar(z))  \right]  \\
    \label{eq:New:UpsTwoFormula}
    &\UpsTwo(z) = \frac{1 + \gamma(\res(z) + z\res'(z))}{z^2} \left [ 1 - 2\sbar(z)\MPStiel_\beta(-\sbar(z)) + (\sbar(z))^2\MPStiel_\beta'(-\sbar(z)) \right]\,,
\end{align}
where $\MPStiel_\beta'(\cdot)$ is the derivative of $\MPStiel_\beta(\cdot)$. {One can verify that whenever $z\in (\sigmaBulk^2,\infty)$, one has $0<-\sbar(z)<(1-\sqrt{\beta})^2$ and so (\ref{eq:MPStiel-def}) is indeed applicable. We show this explicitly in Lemma \ref{lem:range_sbar}, Appendix~\ref{sec:New:Ups}. }

Finally, define the function
\begin{align}
\Efun(z) = -\gamma|\res(z)|\cdot  \left(  \UpsOne(z) + z\UpsOne'(z) \right) 
    + z\cdot |\sbar(z)|\cdot \left( \UpsTwo(z)-\stiel(z)^2 \right)\,.
\end{align}

\subsection{Detailed description of the estimator}
\label{sec.estimator}

We now give the exact details of our proposed estimator, including estimates of the formulas for the optimal choice of $\eta_1,\dots,\eta_r$. This process will be derived formally in Section \ref{sec:derivation_shrinker}; we present it now for the reader's convenience. 

\begin{table}
\centering
\begin{tabular}{| c | c  |  }
\hline  
 Symbol &  Description 
\\
\hline
$\stiel$ & Stieltjes transform 
\\
\hline
$\sbar$ & Associated Stieltjes transform 
\\
\hline
$\res$ & Resolvent-like function 
\\
\hline
$\psi$ & $\psi(z)=z \cdot \sbar(z) \cdot \res(z)$  
\\
\hline
$\varphi$ & $\varphi(z) = z\cdot |\sbar(z)| \cdot \left( \stiel(z)\right)^2 $ 
\\
\hline
$\Xi$ & Spike-forward map 
\\
\hline
$\Xi^{-1}$ & Spike-backward map 
\\
\hline
$\UpsOne$, $\UpsTwo$, $\Efun$ & Auxiliary functions  
\\
\hline
\end{tabular}
\caption{Functions used in this paper. 
}
\label{table:functions}
\end{table}

\begin{table}
\centering
\begin{tabular}{| c | c  |  }
\hline
 Symbol(s) &  Description    \\
\hline
$\sigmaThresh$ & Detection threshold      \\
\hline
$\sigmaMin$ & Left bulk edge    \\
\hline
$\sigmaBulk$ & Right bulk edge    \\
\hline
$\yval_1,\dots,\yval_r$ & Singular values of $\bY^w$     \\
\hline
$\sigma_1,\dots,\sigma_r$ & Singular values of $\bX$    \\
\hline
$c_1,\dots,c_r$ & Left weighted inner products     \\
\hline
$\underline{c}_1,\dots,\underline{c}_r$ & Right inner products     \\
\hline
$\tau_1,\dots,\tau_r$ & Signal/noise alignments    \\
\hline
$\mu$ & Noise covariance trace     \\
\hline
$p$ & Dimensionality     \\
\hline
$n$ & Number of signal-plus-noise samples     \\
\hline
$m$ & Number of noise-only samples     \\
\hline
$\gamma$ & Signal-plus-noise aspect ratio $p/n$    \\
\hline
$\beta$ & Noise-only aspect ratio $p/m$     \\
\hline
$\eta_1,\dots,\eta_r$ & Optimal generalized singular values     \\
\hline
$t_1,\dots,t_r$ & Optimal approximate singular values     \\
\hline
\end{tabular}
\caption{Parameters used in this paper.}
\label{table:parameters}
\end{table}

\paragraph{Inputs:} 
\begin{itemize}

\item The data matrix $\bY \in \RR^{p\times n}$ to be denoised.

\item $m>p$ samples of pure noise $\Rout \in \RR^{p\times m}$.

\item A small parameter $\epsilon>0$.

\end{itemize}

\paragraph{Algorithm:}
\begin{enumerate}
    
\item Form the sample covariance of pure noise samples $\widehat{\bSigma} = \Rout \Rout^\T / m$. 

\item Estimate the normalized trace of the noise covariance:
\begin{align}
\widehat{\mu} = \frac{1}{p}\trace\widehat{\bSigma}.
\end{align}

\item Form the matrix $\bY^w = \hbSigma^{-1/2}\bY$ and take its SVD,
\begin{align}
    \bY^w = \sum_{i=1}^{\min(n,p)} \yval_i \widehat{\bu}^w_i (\widehat{\bv}_i^w)^\T \,.
\end{align}

\item Estimate the "effective" signal rank\footnote{It will shown later that for any fixed $\epsilon>0$, we have $\rhat(\epsilon)\le r$ asymptotically almost surely; furthermore, for $\epsilon$ sufficiently small, $\rhat(\epsilon)$ will be exactly the number of spikes $r^*$ whose intensity is strong enough so as to not be ``swallowed'' completely by the noise. }
\begin{align}\label{eq:rank-estimation}
\rhat = \max \left\{ 1\le i \le n \,:\,\yval_i > \sigmaBulk+\epsilon \right\}\,,
\end{align}
(if the set is empty, set $\rhat=0$), where $\sigmaBulk$ appears in (\ref{eq:sigma-bulk-min}).

\item Estimate the parameters $\tau_1,\dots,\tau_{\rhat}$ via
\begin{align}
\label{eq.estimate.tau}
\widehat \tau_j
    = \frac{\varphi(\yval_j^2)}{|\psi'(\yval_j^2)|\|\hbSigma^{1/2} \hbu_j\|^2 - \widehat \mu  \Efun(\yval_j^2) },
\end{align}
\Revision{where $\psi'(\cdot)$ denotes the derivative of $\psi(\cdot)$.}

\item 
For $1 \le k \le \rhat$, estimate the following parameters:
\begin{align}
\what \sigma_k = \frac{1}{\sqrt{\widehat{\tau_k}}}\Xi^{-1}(\what \theta_k),
\quad
\what c_k =
\sqrt{\frac{1}{\what \tau_k} \cdot 
\frac{\varphi(\yval_k^2)}{\psi'(\yval_k^2)} },
\quad
\underline{\what c}_{k} =
\sqrt{\sbar(\yval_k^2)\cdot \frac{\psi(\yval_k^2)}{\psi'(\yval_k^2)}}.
\end{align}

\item Define
\begin{align}
\what \eta_k = \frac{\what \sigma_k \what c_k \underline{\what c}_k}{\|\what \bSigma^{1/2} \what \u_k^\wh\|^2}.
\end{align}
Return the estimator
\begin{align}
\what \X = \sum_{k=1}^{\rhat} \what \eta_k  \what \bSigma^{1/2} \what \u_k^\wh (\what \v_k^w)^\T,
\end{align}
and the estimated AMSE:
\begin{align}
\what \AMSE = \sum_{k=1}^{r} \what \sigma_k^2 \left(1-\what c_k^2\underline{\what c}_k^2 
        \frac{1}{\|\hbSigma^{1/2} \what \u_k^\wh\|^2} \right) .
\end{align}

\end{enumerate}

\Revision{
\begin{remark}
    In lieu of the ``effective rank'' estimator \eqref{eq:rank-estimation}, one might assume access to a known upper bound $r\le \overline{r}$ and set $\eta_k=0$ ($1\le k \le \overline{r}$) for every observed singular value with $\yval_k \le \sigmaBulk$. When $\overline{r}$ is constant as $n,p,m\to\infty$, the resulting estimator has asymptotically optimal MSE, in the same sense as described in Section~\ref{sec:derivation_shrinker}.
\end{remark}
}

\section{Main results}
\label{sec:main_results}

In this section, we provide the mathematical results that justify the algorithm in Section \ref{sec.estimator}. Specifically, we provide formulas for the quantities necessary to estimate the asymptotically  optimal coefficients  $\eta_1,\dots,\eta_{\rhat}$.

\subsection{The limiting spectrum of a spiked F-matrix}
\label{sec:main_results-theory}

As mentioned before, the random matrix $\bY^w$ is an instance of a spiked F-matrix. The results of this section quantify the relevant phenomenology surrounding these matrices, namely we: 1) compute the spike detection threshold, and a formula for the singular value displacement (spike-forward) map; 2) compute the limiting cosines between the population and empirical PCs. These are the two components necessary to derive the optimal shrinkage rule, as presented above; we do this in Section~\ref{sec:derivation_shrinker}.


We start with a limiting formula for the singular values of $\bY^w$:
\begin{theorem}
\label{thm:detectability-pt}
Set $r^* = \max \left\{ 1 \le k \le r \,: \,\sqrt{\LESDMomInv_k}\sigma_k > \sigmaThresh \right\}$.
%
Then, for any fixed $k$, almost surely,
\begin{align}
\lim_{p\to\infty} \yval_{k} = y_k :=  \begin{cases}
\Xi(\sqrt{\tau_k}\sigma_k ) \quad&\textrm{ if } 1\le k \le r^* \\
\sigmaBulk \quad&\textrm{ if } k > r^* 
\end{cases} \,.
\end{align}
\end{theorem}


Theorem~\ref{thm:detectability-pt} identifies a detectability phase-transition: a spike can be identified consistently by looking at the leading empirical singular value (as an outlier, separated from the ``main bulk'' of other singular values) whenever its ``effective intensity''
$\sqrt{\tau_i}\sigma_i$ exceeds the threshold $\sigma_{\mathrm{thresh}}$.\footnote{Note that Theorem~\ref{thm:detectability-pt} does not imply that detection is impossible below the threshold. Detection in this so-called ``sub-critical regime'' was studied in \cite{johnstone2015testing}, and  their results imply that {consistent} detection is indeed impossible (assuming $\bSigma=\bI$ and uniformly random signal spikes). } When this is the case, this effective intensity can in fact be estimated consistently from the data as $\Xi^{-1}(\yval_i)$. 
Figure \ref{figure.thresholds} plots the detection threshold $\sigmaThresh$ as a function of $\beta$, for different values of $\gamma$.

\begin{figure}
\center
\includegraphics[scale=.3]{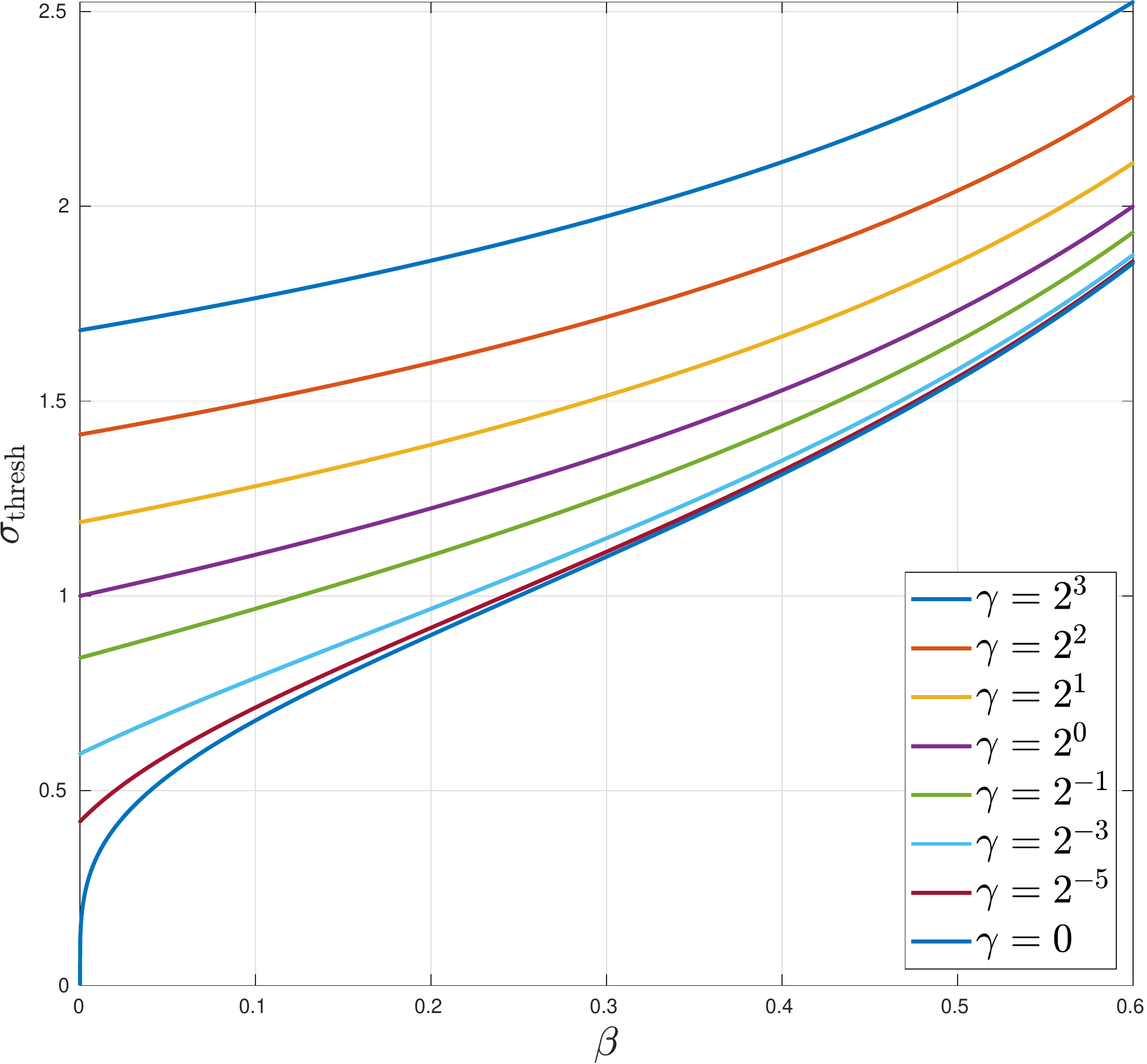}
\caption{The value of $\sigmaThresh$ as a function of $\beta$, for different values of $\gamma$.}
\label{figure.thresholds}
\end{figure}

The result in Theorem~\ref{thm:detectability-pt} is not new. To our knowledge, it first appeared in \cite{nadakuditi2010fundamental}, where it is stated for $\bSigma=\bI$; an extension for the non-white case is straightforward. 
More sophisticated results have since appeared in the literature. For a spike above the detectability threshold, assuming $\bSigma=\bI$ and Gaussian $\Zin,\Zout$, \cite{dharmawansa2014local} proved a CLT for the empirical singular values: $\{ \sqrt{p}(\yval_k-y_k) \}_{i\le r^*}$ converge jointly to a centered multivariate Gaussian (with an explicitly given covariance matrix). This result was extended by \cite{wang2017extreme}, which only assumed $\Zin,\Zout$ with independent entries and finite 4-th moment. Another related work is \cite{johnstone2017roy} which studied, assuming Gaussian noise, the distribution of the leading empirical singular values in the non-asymptotic (finite $n$), high SNR regime ($\sigma\to\infty$) by perturbation theory methods. A follow-up work \cite{dharmawansa2019roy} extended the analysis for complex Gaussian matrices.
Lastly, \cite{xie2021limiting} studied the leading eigenvalues in a model where the numbers of spikes $r$ is divergent (but not too large compared to $n,p,m$). 
Although Theorem~\ref{thm:detectability-pt} not new, in Section \ref{sec:proof_detectability} we will nonetheless provide a self-contained proof, which shall also function as an important preparatory step towards proving Theorem~\ref{thm:products} below. 

Lastly, Theorem~\ref{thm:detectability-pt} readily justifies the rank estimation procedure described in (\ref{eq:rank-estimation}). It implies that for any fixed small enough $\eps$, specifically, $0<\eps < \Xi(\sqrt{\tau_{r^*}}\sigma_{r^*} )-\sigmaBulk$, one has $\lim_{p\to\infty}\rhat(\eps)= r^*$ almost surely. We remark in passing that one may in fact choose $\eps=o(1)$, so to always ensure (asymptotically) consistent rank estimation. Indeed, when there is no signal ($\bX=\0$) the stochastic fluctuations $n^{2/3}(\lambda_1(n^{-1}\Y \Y^\T) - \sigmaBulk)$  converge in distribution to a Tracy-Widom law \cite{johnstone2008multivariate,han2016tracy}. In particular, in the presence of spikes, one may verify (e.g. by singular value interlacing) that $\lambda_{r+1}(n^{-1}\Y^\T \bY) = \sigmaBulk+O_{\Prob}(n^{-2/3})$. Consequently, for any vanishing $\eps=\omega(n^{-2/3})$, a.s. $r^*\le \liminf_{p\to\infty} r^*(\eps)\le \limsup_{p\to\infty}r^*(\eps)\le r$.
\newline

%

The next theorem calculates the limiting cosines between the population signal spikes and their empirical counterparts:


\begin{theorem}
\label{thm:products}
Let $1 \le k \le r$, and set $y_k=\Xi(\sqrt{\tau_k}\sigma_k)$. Suppose $\sqrt{\tau_k}\sigma_k > \sigmaThresh$. Then
\begin{align}
\label{eq:cos}
c_k^2 \equiv\lim_{p\to\infty}\left( \bu_k^\T\hbSigma^{1/2}\hbu_k^w \right)^2
=
\frac{1}{\tau_k} \cdot \frac{y_k^2\left(\stiel(y_k^2)\right)^2\sbar(y_k^2)}{\psi'(y_k^2)} \,,
\end{align}
\begin{align}
\label{eq:cos2}
\underline{c}_{k}^2 \equiv \lim_{p\to\infty} \left(\bv_k^\T\hbv_k\right)^2
=
\sbar(y_k^2)\cdot \frac{\psi(y_k^2)}{\psi'(y_k^2)},
\end{align}
\begin{align}
\label{eq:lem:post-corr-2}
\lim_{p\to\infty} \left(\bu_k^\T\hbSigma^{1/2}\hbu_k^w\right) \cdot \left(\bv_k^\T\hbv_k\right) 
=
-\frac{1}{\sqrt{\tau_k}} \cdot \frac{y_k\stiel(y_k^2)\sbar(y_k^2)\cdot \sqrt{\psi(y_k^2)}}{\psi'(y_k^2)}
\ge 0,
\end{align}
\begin{align}
\label{eq:thm-norm}
\lim_{p\to\infty}\|\hbSigma^{1/2}\hbu_k^w\|^2 
= \frac{1}{|\psi'(y_k^2)|}\left[ \mu\cdot \Efun(y_k^2) + \frac{1}{\tau_k} \varphi(y_k^2) \right] \,,
\end{align}
and
\begin{align}
\lim_{p\to\infty}\left(\bu_k^\T \hbu_k\right)^2 = \frac{\varphi(y_k^2)}{\tau_k\mu\cdot \Efun(y_k^2) + \varphi(y_k^2)} \,,
\end{align}
where the limits hold almost surely. Moreover, if $1 \le k,l \le r$ and $k \ne l$, then
\begin{align}\label{eq:cos-off}
\lim_{p\to\infty}\left( \bu_k^\T\hbSigma^{1/2}\hbu_l^w \right)^2
=
\lim_{p\to\infty} \left(\bv_k^\T\hbv_l\right)^2
=
0 \,.
\end{align}
%

\end{theorem}

The proof of Theorem \ref{thm:products} is found in Section \ref{sect:proof-precosines}.

\subsection{Derivation of the optimal shrinker}
\label{sec:derivation_shrinker}

Recall the form of the estimators (\ref{eq:family-F}) and (\ref{eq:approximate_SVD}). Equipped with Theorems \ref{thm:detectability-pt} and \ref{thm:products}, it is straightforward to evaluate the optimal weights $\eta_1,\dots,\eta_r$ in terms of the population parameters.

\begin{theorem}
\label{thm:shrinker}
Let $y_k=\Xi(\sqrt{\tau_k} \sigma_k)$, and suppose $\sqrt{\tau_k} \sigma_k > \sigmaThresh$, $1 \le k \le r$. Then the optimal $\boldeta = (\eta_1,\dots,\eta_r)$ that minimizes $\AMSE(\boldeta)$ is given by the following:
\begin{align}
\eta_k = \sigma_k \cdot c_k \cdot \underline{c}_k 
    \cdot |\psi'(y_k^2)
    |\left[ \mu\cdot \Efun(y_k) + \frac{1}{\tau_k} \varphi(y_k) \right]^{-1},
\end{align}
where $c_k$ and $\underline{c}_k$ are defined in \eqref{eq:cos} and \eqref{eq:cos2}, respectively. The optimal $t_1,\dots,t_k$ are given by
\begin{align}
t_k = \sigma_k \cdot c_k \cdot \underline{c}_k 
    \cdot |\psi'(y_k^2)
    |\left[ \mu\cdot \Efun(y_k) + \frac{1}{\tau_k} \varphi(y_k) \right]^{-1/2}.
\end{align}
The AMSE at these optimal values is equal to
\begin{align}
\label{eq:AMSE_formula}
\AMSE(\boldeta) = \sum_{k=1}^{r} \sigma_k^2 \left(1-c_k^2\underline{c}_k^2 
        |\psi'(y_k^2)| \left[ \mu\cdot \Efun(y_k) + \frac{1}{\tau_k} \varphi(y_k) \right]^{-1} \right) .
\end{align}
\end{theorem}

\begin{proof}
Using the asymptotics in Theorem \ref{thm:products}, the mean squared error may be written as:
\begin{align}
\label{eq:5411}
\|\what \X - \X\|_F^2
&= \sum_{k=1}^{r} \left\{\eta_k^2 \| \what \bSigma^{1/2}\what \u_k^w \|^2
        - 2\eta_k \sigma_k (\bu_k^\T\hbSigma^{1/2}\hbu_k^w)  (\bv_k^\T\hbv_k) \right\}
+ \sum_{k=1}^{r} \sigma_k^2 
- \sum_{j \ne k} 2\eta_j \sigma_k (\bu_k^\T\hbSigma^{1/2}\hbu_j^w)  (\bv_k^\T\hbv_j)
\nonumber \\
&\asympeq \sum_{k=1}^{r} \left\{\eta_k^2 \frac{1}{|\psi'(y_k^2)|}\left[ \mu\cdot \Efun(y_k^2) + \frac{1}{\tau_k} \varphi(y_k^2) \right]
        - 2\eta_k \sigma_k c_k \underline{c}_k \right\} + \sum_{k=1}^{r} \sigma_k^2,
\end{align}
and minimizing this gives the desired expression for $\eta_k$. The optimal $t_k$ follow from $t_k = \eta_k \cdot \| \what \bSigma^{1/2}\what \u_k^w \|$, and \eqref{eq:thm-norm}. The formula for the AMSE is obtained substituting $\eta_k$ into the expression for the AMSE in \eqref{eq:5411}.
\end{proof}

The optimal $\eta_k$ may be consistently estimated using only the observed matrices $\bY$ and $\Rout$. Indeed, suppose that $\yval > \sigmaBulk$. 
Then from Theorems~\ref{thm:detectability-pt} and \ref{thm:products}, the values $\tau_k$ may be estimated via
\begin{align}
\widehat \tau_k
    = \frac{\varphi(\yval_k^2)}{|\psi'(\yval_k^2)|\|\hbSigma^{1/2} \hbu_k\|^2 - \widehat \mu  \Efun(\yval_k^2) }, \quad 1 \le k \le r.
\end{align}
while the population intensities $\sigma_k$ may be consistently estimated as
\begin{align}
\what \sigma_k = \frac{1}{\sqrt{\widehat{\tau}_k}} \Xi^{-1}(\yval_k), \quad 1 \le k \le r.
\end{align}
%
%
With these estimates, using Theorem \ref{thm:products} $c_k$ and $\underline{c}_k$ may be estimated as follows:
\begin{align}
\what c_k =
\sqrt{\frac{1}{\what \tau_k} \cdot \frac{\yval_k^2\left(\stiel(\yval_k^2)\right)^2\sbar(\yval_k^2)}{\psi'(\yval_k^2)} },\qquad
\underline{\what c}_{k} =
\sqrt{\sbar(\yval_k^2)\cdot \frac{\psi(\yval_k^2)}{\psi'(\yval_k^2)}}.
\end{align}
Finally, the factor
\begin{align}
|\psi'(y_k^2)| \left[ \mu\cdot \Efun(y_k) + \frac{1}{\tau_k} \varphi(y_k) \right]^{-1}
\end{align}
may be estimated as $\|\hbSigma^{1/2} \what \u_k^w\|^{-2}$. Putting these together, we estimate the optimal $\eta_k$ as
\begin{align}
\what \eta_k = \what \sigma_k \cdot \what c_k \cdot \underline{\what c}_k \,.
\end{align}
This completes the derivation of the optimal $\what \X $.

\section{Background on F-Matrices}
\label{sec:background-F}

Before moving on to the proofs of our main results, we briefly review several known result about the spectrum of the matrix $\Y^{w}$ in the absence of a signal. In other words, we consider the pseudo-whitened noise matrix $\N\in\R^{ p \times n}$
\begin{equation}\label{eq:N-def}
    \N = \frac{1}{\sqrt{n}}\hbSigma^{-1/2}\Rin = \frac{1}{\sqrt{n}}\hbSigma^{-1/2}\bSigma^{1/2}\Zin \,.
\end{equation}

Define the following Wishart matrices:
\begin{equation}
\label{eq:E_and_S}
    \bE = \Zin\Zin^\T/n,\quad \bS=\Zout\Zout^\T/m,
\end{equation}

Importantly, observe that the singular values of $\N$ do not depend on the population covariance $\bSigma$. To see this, write $\hbSigma=\bSigma^{1/2}\bS \bSigma^{1/2}$, and recall that the squared singular values of $\N$ are the (non-zero) eigenvalues of $\N^\T \N = \Zin^\T \bSigma^{1/2}\hbSigma^{-1}\bSigma^{1/2}\Zin = \Zin^\T \bS^{-1} \Zin$, which does not depend on $\bSigma$. Thus, when studying the singular values of $\N$ we may assume w.l.o.g. that $\bSigma=\Id$, as we shall do throughout this section. The non-zero eigenvalues of $\N^\T \N\in \R^{n\times n}$ and $\N\N^\T \in \R^{p\times p}$ are the same. With $\bE$ defined as in \eqref{eq:E_and_S}, we write $\N\N^\T = \bS^{-1/2}\bE \bS^{-1/2}$. Applying a similarity transformation, its eigenvalues are identical to those of the matrix
\begin{equation}\label{eq:fisher-def}
     \bF = \bS^{-1/2}\bE \bS^{1/2} = \bS^{-1}\bE \,.
\end{equation}
The random matrix ensemble (\ref{eq:fisher-def}) is known as an F-matrix/Fisher matrix; see, for example, \cite{bai2009book}. It is closely related to the 
multivariate Beta ensemble, which features in classical multivariate statistics, in particular Multivariate Analysis of Variance (MANOVA) \cite{muirhead2009aspects}.\footnote{
One can show that the eigenvalues of $\bF$ and the MANOVA matrix $(\bS+\bE)^{-1/2}\bE(\bS+\bE)^{-1/2}$ are bijectively mapped to one another via the mapping $z\mapsto z/(1+z)$.
}

We briefly recall some definitions. 
For a diagonalizable matrix $\bA\in \R^{p\times p}$, its empirical spectral distribution (ESD) is the counting measure of the eigenvalues: $\mu_{\bA}:= n^{-1}\sum_{i=1}^p \delta_{\lambda_i(\bA)}$. For a probability measure $\mu$ on $\RR$, its Stieltjes transform $s_{\mu}(\cdot)$ is the function 
\begin{equation}
    s_{\mu}(z) = \int_{-\infty}^\infty \frac{1}{\lambda -z }d\mu(z)\,,\quad \textrm{where}\quad z\in \CC\setminus \RR \,.
\end{equation}
The limiting empirical spectral distribution (LESD) for F-matrices was first derived by Wachter \cite{wachter1980limiting} with subsequent generalizations in \cite{yin1983limiting,silverstein1985limiting,silverstein1995empirical}; for a textbook reference, see \cite[Theorem 4.10]{bai2009book}.

\begin{theorem}
    [Wachter's distribution]
    \label{thm:wachter-density}
    Let $\sigmaMin,\sigmaBulk$ be as in (\ref{eq:sigma-bulk-min}). The ESD of the F-matrix (\ref{eq:fisher-def}) converges weakly almost surely to a deterministic distribution. When $\gamma\le 1$, the LESD is continuous, supported on $[\sigmaMin,\sigmaBulk]$ and has the density, \begin{equation}\label{eq:Wachter-density}
        F_{\gamma,\beta}(\lambda) = 
    \frac{(1-\beta)\sqrt{(\sigmaBulk^2-\lambda)(\lambda-\sigmaMin^2)}}{2\pi \lambda(\gamma+\beta \lambda)} \Ind\{\sigmaMin^2 \le \lambda \le \sigmaBulk^2\} \,.
    \end{equation}
    When $\gamma>1$ the LESD has a continuous density (\ref{eq:Wachter-density}), but also an atom at $\lambda=0$ with weight $1-\gamma^{-1}$.  
\end{theorem}
Note that when $\beta=0$, (\ref{eq:Wachter-density}) collapses to the well-known Marchenko-Pastur law with shape $\gamma$.

Theorem~\ref{eq:Wachter-density} implies, in particular, that all but a vanishing fraction of the non-zero eigenvalues of $\bF$ are contained in $[\sigmaMin^2,\sigmaBulk^2]$. The next result, which follows from \cite[Theorem 1.1]{bai1998no}, implies that in fact, asymptotically almost surely there are no eigenvalues outside the support:
\begin{theorem}
    [Extreme eigenvalues of F-matrix]
    \label{thm:wachter-edge}
    Almost surely,
    \[
    \lim_{p\to\infty} \lambda_{\min}(\bF)=\sigmaMin^2,\quad \lim_{p\to\infty} \lambda_{\max}(\bF)=\sigmaBulk^2 \,.
    \]
\end{theorem}
In our proofs, we shall use a well-known formula for the Stieltjes transform of Wachter's law. The following appears, for example, in \cite[Theorem 4.10]{bai2009book}. 
\begin{proposition}
[Stieltjes transform of Wachter's law]
\label{prop:stieltjes}
The Stieltjes transform of Wachter's law is given by (\ref{eq:stieltjes-def}). Furthermore, we have the following convergence for the empirical Stieltjes transform:
\begin{align*}
    &p^{-1}\tr(\bN\bN^\T - z\Id)^{-1} \longrightarrow \stiel(z), \\
     &\frac{d}{dz} p^{-1}\tr(\bN\bN^\T - z\Id)^{-1} =  p^{-1}\tr(\bN\bN^\T - z\Id)^{-2} \longrightarrow \stiel'(z) \,,
\end{align*}
for all 
$z\in (\sigmaBulk^2,\infty)$.
Above, convergence is a.s. and uniform on compact subintervals. 

Furthermore,
\begin{align*}
    &n^{-1}\tr(\N^\T \N -z\Id)^{-1} = n^{-1}\left[\tr(\N \N^\T -z\Id)^{-1} + (n-p)\frac{1}{z}\right]  \longrightarrow \sbar(z), \\
    &\frac{d}{dz}n^{-1}\tr(\N^\T \N -z\Id)^{-1} = n^{-1}\tr(\N^\T\N-z\Id)^{-2} \longrightarrow \sbar'(z) \,.
\end{align*}
\end{proposition}
We also need a limiting expression for the following resolvent-like expression from \cite[Lemma 1]{dharmawansa2014local}: 
\begin{proposition}\label{prop:res}
Let $\res(z)$ be as in (\ref{eq:res-def}). Then 
\begin{align*}
    &p^{-1}\trace \left(\bE-z\bS\right)^{-1} \longrightarrow \res(z),\\
    &\frac{d}{dz} p^{-1}\trace \left(\bE-z\bS\right)^{-1} = p^{-1}\trace \bS\left(\bE-z\bS\right)^{-2} \longrightarrow \res'(z) \,,
\end{align*}
for all $z\in (\sigmaBulk^2,\infty)$. 
Above, convergence is a.s. and uniform on compact subintervals.
\end{proposition}

\paragraph{Notation.} For (possibly random) sequences $a_p,b_p\in \RR$, we write $a_p\asympeq b_p$ to indicate $a_p-b_p\to 0$ a.s. 
For a sequence of matrices $\bM_p \in \RR^{p\times p}$, we denote $\trlim(\bM_p):=\lim_{p\to\infty}p^{-1}\tr(\bM_p)$, where it is understood that the limit exists a.s. 

\section{Proofs of main results}
\label{sec:proofs}

\subsection{Proof of Theorem~\ref{thm:detectability-pt}}
\label{sec:proof_detectability}

\begin{proof}[\unskip\nopunct]

We study the leading eigenvalues of the pseudo-whitened data matrix (\ref{eq:Yw-def}),
\begin{equation}
    \Y^w = \bP + \bN\,,
\end{equation}
where $\bP = \hbSigma^{-1/2}\bU\bLambda\bV $ is the pseudo-whitened signal matrix, and $\bN$ is the pseudo-whitened noise matrix (\ref{eq:N-def}). 
Since $\rank(\bP)=r$, by Weyl's interlacing inequalities, for every $i$, 
\begin{align*}
    \sigma_{i+r+1}(\bY^w) \le \sigma_{i+1}(\bN) + \sigma_{r+1}(\bP) = \sigma_{i+1}(\bN)\,, \quad \sigma_{(r+i)+r+1}(\bN) \le \sigma_{r+i+1}(\bY^w) + \sigma_{r+1}(-\bP) = \sigma_{r+i+1}(\bY^w) \,.
\end{align*}
Thus, for any $i$, $\yval_{i+r+1}:=\sigma_{i+r+1}(\bY)$ satisfies $\sigma_{2r+i+1}(\bN) \le \sigma_{i+r+1}(\bY) \le \sigma_{i+1}(\bN)$. By Theorems~\ref{thm:wachter-density} and \ref{thm:wachter-edge}, $\sigma_{2r+i+1}(\bN),\sigma_{i+1}(\bN)\to \sigmaBulk$ a.s. Consequently, for fixed $k\ge r+1$, $\yval_k \to \sigmaBulk$.  

It remains to check whether $\bY^w$ has singular values which are asymptotically larger than $\sigmaBulk$. Note that, appealing to Theorem~\ref{thm:wachter-edge}, such singular values necessarily cannot be singular values of the noise matrix $\bN$. \Revision{By an argument essentially identical to \cite[Lemma 4.1]{benaych2012singular},} the singular values of $\bY^w$ which are not singular values of $\bN$ are exactly (with multiplicities) the solutions to $\det(\widehat{\bM}(y))=0$, where $\widehat{\bM}(y)$ is the $2r$-by-$2r$ symmetric matrix
\begin{align}
    \label{eq:M_n(z)}
    \widehat{\bM}(y) = \begin{bmatrix}
    y \cdot \bU^\T\widehat{\bSigma}^{-1/2} \left( y^2 \bI_p - \bN \bN^\T \right)^{-1} \widehat{\bSigma}^{-1/2}\bU &\quad \bU^\T\widehat{\bSigma}^{-1/2}\left(y^2 \bI_p-\bN\bN^\T\right)^{-1} \bN \bV \\
    \bV^\T \bN^\T \left(y^2 \bI_p-\bN\bN^\T\right)^{-1} \widehat{\bSigma}^{-1/2}\bU &\quad
    y \cdot \bV^\T \left( y^2 \bI_n -\bN^\T \bN \right)^{-1} \bV
    \end{bmatrix} -
    \begin{bmatrix}
    0 &\quad \bLambda^{-1}\\
    \bLambda^{-1} &\quad 0
    \end{bmatrix} \,.
\end{align}

Define the matrix $\bTau=\diag(\tau_1,\ldots,\tau_r)\in \R^{r\times r}$. 
The next lemma asserts that the random matrix $\widehat{\bM}(y)$ converges a.s. to a deterministic limit as $p\to\infty$:
\begin{lemma}\label{lem:M-conv}
    Suppose that $y\notin \CC\setminus[\sigmaMin,\sigmaBulk]$. Then a.s., $\widehat{\bM}(y)\longrightarrow \bM(y)$, where
    \begin{align}
    \label{eq:M(z)-lim}
    \bM(y) = \begin{bmatrix}
    -y \res(y^2)\cdot \bTau  &\quad \mathbf{0} \\
    \mathbf{0} &\quad -y\sbar(y^2)\cdot \bI 
    \end{bmatrix} -
    \begin{bmatrix}
    0 &\quad \bLambda^{-1}\\
    \bLambda^{-1} &\quad 0
    \end{bmatrix} \,.
    \end{align}
    Moreover, convergence (of each entry) is uniform on compact subsets in $y\notin \CC\setminus[\sigmaMin,\sigmaBulk]$.
\end{lemma}
\begin{proof}
We start with the off-diagonal elements of (\ref{eq:M_n(z)}). First, note that the matrix $\bN$ is invariant to multiplication by $O(n)$ from the right. Consequently, if $\mathcal{O}\sim \mathrm{Haar}(O(n))$, then $\bU^\T\widehat{\bSigma}^{-1/2}\left(y^2 \bI_p-\bN\bN^\T\right)^{-1} \bN \bV \overset{d}{=} \bU^\T\widehat{\bSigma}^{-1/2}\left(y^2 \bI_p-\bN\bN^\T\right)^{-1} \bN \mathcal{O}\bV \asympeq \bm{0} \in \R^{r \times r}$. Now, considering the top left block and using (\ref{eq:ui-def}), 
\begin{align*}
    \left( \bU^\T\widehat{\bSigma}^{-1/2} \left( y^2 \bI_p - \bN \bN^\T \right)^{-1} \widehat{\bSigma}^{-1/2}\bU \right)_{j,k} 
    &= p^{-1}\w_j^\T \bD_j^\T \widehat{\bSigma}^{-1/2} \left( y^2 \bI_p - \bN \bN^\T \right)^{-1} \widehat{\bSigma}^{-1/2} \bD_k \w_k \\ 
    &\overset{(\star)}{\asympeq} p^{-1}\tr\left( \bD_k^\T \widehat{\bSigma}^{-1/2} \left( y^2 \bI_p - \bN \bN^\T \right)^{-1} \widehat{\bSigma}^{-1/2} \bD_k  \right) \Ind\{j=k\}\,,
\end{align*}
where $(\star)$ follows by the independence of $\w_j,\w_k$ and, e.g., the Hanson-Wright inequality. The term under the trace is  $\widehat{\bSigma}^{-1/2} \left( y^2 \bI_p - \bN \bN^\T \right)^{-1} \widehat{\bSigma}^{-1/2} = \left( y^2\hbSigma - \hbSigma^{1/2}\bN\bN^\T \hbSigma^{1/2} \right)^{-1}=\bSigma^{-1/2}\left( y^2\bS-\bE\right)^{-1}\bSigma^{-1/2}$. Crucially, the matrix $(y^2\bS-\bE)^{-1}$ is orthogonally invariant, and a.s. has bounded operator norm, since $y>\sigmaThresh$. Consequently, $(y^2\bS-\bE)^{-1}$ and $\bSigma^{-1/2}\bD_k\bD_k^\T \bSigma^{-1/2}$ are asymptotically free (see Section~\ref{sec:free-probability}), and so 
\begin{align}
    \trlim\left( \bD_k^\T \widehat{\bSigma}^{-1/2} \left( y^2 \bI_p - \bN \bN^\T \right)^{-1} \widehat{\bSigma}^{-1/2} \bD_k \right) 
    &= \trlim\left( (y^2\bS-\bE)^{-1} \bSigma^{-1/2}\bD_k\bD_k^\T \bSigma^{-1/2} \right) \nonumber\\
    &= \trlim \left((y^2\bS-\bE)^{-1} \right) \trlim\left( \bSigma^{-1/2}\bD_k\bD_k^\T \bSigma^{-1/2} \right) \nonumber\\
    &= -\res(y^2)\cdot \tau_k \,, \label{eq:aux0}
\end{align}
where we used Proposition~\ref{prop:res} and (\ref{eq:tau-def}). This establishes pointwise convergence; uniform convergence follows from the Arzela-Ascoli theorem, where both equicontinuity and uniform boundedness of the entries clearly hold, since $y$ is bounded away from the support of Wachter's distribution, $[\sigmaMin,\sigmaBulk]$. Finally, the bottom right block of (\ref{eq:M_n(z)}) may be analyzed similarly, using Proposition~\ref{prop:stieltjes}.
\end{proof}

\Revision{Since the operator norm of $\Y^w$ is a.s. bounded by a constant, hence all its singular values lie a.s. within a fixed compact interval, 
one can deduce using elementary complex analysis (see e.g. \cite[Lemma 6.1]{benaych2011eigenvalues} for a formal argument), that for each $1\le k\le r$, either:} 1) $\sigma_k(\Y^w)\to \sigmaBulk$ a.s.; or 2) $\sigma_k(\Y^w)$ tends a.s. to a root of the deterministic polynomial equation $\det(\bM(y))=0$. Moreover, the roots of this equation match exactly (including their multiplicities) the limiting locations and counts of the outlying singular values of $\Y^w$.  
Now, one may verify that $\det(\bM(y))=0$ if and only if for some $1\le k \le r$, 
\begin{equation}\label{eq:outlier-eq}
    \psi(y^2):= y^2\cdot \res(y^2) \cdot \sbar(y^2) = 1/(\tau_k \sigma_k^2) \,,
\end{equation}
where an explicit formula for $\psi(z)$ is given in (\ref{eq:psi}). One can show that $\psi(\cdot)$ is strictly decreasing, and maps $(\sigmaBulk^2,\infty)$ bijectively to $(1/\sigmaThresh^2,0)$. Consequently, a solution to (\ref{eq:outlier-eq}) exists if and only if $\sqrt{\tau_k}\sigma_k > \sigmaThresh$, in other words, $k\le r^*$. In this case, one may also compute explicitly the functional inverse, so that (\ref{eq:outlier-eq}) is equivalent to $y = \Xi(\sqrt{\tau_k}\sigma_k)$, where the spike-forward map $\Xi(\cdot)$ is given in (\ref{eq:spike-forward-in-thm}). 
This concludes the proof of Theorem~\ref{thm:detectability-pt}.
\end{proof}

\subsection{Proof of Theorem~\ref{thm:products}}

\label{sect:proof-precosines}

\begin{proof}[\unskip\nopunct]

The first step of the proof consists of computing the limiting inner products $\langle \bu_i,\hbSigma^{-1/2}\hbu^w_k\rangle $ and $\langle \bv_i,\hbv_k\rangle$. This is an important intermediate step towards calculating $\langle \bu_i,\hbSigma^{1/2}\hbu^w_k\rangle $ and $\|\hbSigma^{1/2}\hbu^w_k\|$.
To this end,
define the vectors $\balpha_k,\bbeta_k\in \RR^r$, $1\le k \le r$,
\begin{equation*}
    \balpha_k = \bU^\T \hbSigma^{-1/2}\bu_k^w\,,\quad \bbeta_k = \bV^\T \hbv_k \,.
\end{equation*}

\begin{lemma}
\label{lem:vecs-off-diagonal}

Let $1\le k \le r$ be such that $\sqrt{\tau_k} \sigma_k > \sigmaThresh$. For every $i\in [r]\setminus\{k\}$, 
\begin{align}
        \lim_{p\to\infty} (\balpha_k)_i = 0,\quad \lim_{p\to\infty} (\bbeta_k)_i = 0 \,.
\end{align}
\end{lemma}
\begin{proof}
\Revision{By an argument identical to \cite[Lemma 5.1]{benaych2012singular}, one can show that the vector $\bm{x}_k=(\bLambda\bbeta_k,\bLambda \balpha_k)\in \R^{2r}$ lies in the kernel of $\widehat{\bM}(\yval_k)$, where $\widehat{\bM}(\cdot)$ is defined in (\ref{eq:M_n(z)}). (We refer to \cite{benaych2012singular} for the details.)}  Since $\bm{x}_k$ has bounded norm, and $\lim_{p\to\infty}\|\widehat{\bM}(\yval_k)-\bM(\y_k)\| = 0$ (by Lemma~\ref{lem:M-conv}), we have $\lim_{p\to\infty}\bM(y_k)\bm{x}_k=\bm{0}$ a.s. Considering only coordinates $i$ and $i+r$, where $i\ne k$, yields the equation
\begin{align}
\label{eq:aux4400}
\bm{0} \asympeq
\begin{bmatrix}
y_k \res(y_k^2) \tau_i & 1/\sigma_i \\
1/\sigma_i              & y_k \sbar(y_k^2)
\end{bmatrix}
\begin{bmatrix}
(\bm{x}_k)_i \\
(\bm{x}_k)_{i+r}
\end{bmatrix} \,.
\end{align}
Observe that the matrix in (\ref{eq:aux4400}) is invertible: its determinant is $\tau_i y_k^2\res(y_k^2)\sbar(y_k^2)-1/\sigma_i^2=\tau_i\psi(y_k^2)-1/\sigma_i^2 = \tau_i/(\tau_k \sigma_k^2)-1/\sigma_i^2$, where we used (\ref{eq:outlier-eq}); now recall that by assumption, $\tau_k\sigma_k^2 \ne \tau_i\sigma_i^2$, so the determinant is not zero. Thus, $(\bm{x}_k)_i,(\bm{x}_k)_{i+r}\asympeq 0$, and the claim follows.  
\end{proof}

By the definition of $\hbu_k^w,\hbv_k$ as singular vectors of $\Y^w$, we have $\bY^w \hbv_k = \yval_k\hbu_k^w$ and ${\bY^w}^\T \hbu_k^w = \yval_k\hbv_k$. Decomposing $\Y^w=\hbSigma^{-1/2}\bU\bLambda\bV^\T + \bN$ yields $\yval_k \hbu_k = (\hbSigma^{-1/2}\bU\bLambda\bV^\T + \bN)\hbv_k = \hbSigma^{-1/2}\bU\bLambda\bbeta_k + \bN \hbv_k$.
Rearranging, $\bN\hbv_k = \yval_k\hbu_k^w - \hbSigma^{-1/2}\bU\bLambda\bbeta_k $. Similarly, $\bN^\T \hbu_k^w = \yval_k\hbv_k - \bV\bLambda\balpha_k$. Now,
\begin{align*}
    \yval_k^2\hbv_k 
    &= {\Y^w}^\T \Y^w \hbv_k 
    = (\hbSigma^{-1/2}\bU\bLambda\bV^\T+\bN)^\T(\hbSigma^{-1/2}\bU\bLambda\bV^\T+\bN) \hbv_k\\
    &= (\hbSigma^{-1/2}\bU\bLambda\bV^\T+\bN)^\T (\hbSigma^{-1/2}\bU\bLambda\bbeta_k+\bN\bv_k) \\
    &= \bV\bLambda\bU^\T \hbSigma^{-1}\bU\bLambda \bbeta_k + \bN^\T \hbSigma^{-1/2} \bU\bLambda\bbeta_k + \bV\bLambda\bU^\T \hbSigma^{-1/2}(\bN\hbv_k) + \bN^\T \bN\hbv_k \\
    &= \bV\bLambda\bU^\T \hbSigma^{-1}\bU\bLambda \bbeta_k + \bN^\T \hbSigma^{-1/2} \bU\bLambda\bbeta_k + \bV\bLambda\bU^\T \hbSigma^{-1/2}(\yval_k\hbu_k^w - \hbSigma^{-1/2}\bU\bLambda\bbeta_k) + \bN^\T \bN\hbv_k \\
    &= \bN^\T \hbSigma^{-1/2} \bU\bLambda\bbeta_k + \yval_k \bV\bLambda\balpha_k + \bN^\T \bN\hbv_k \,.
\end{align*}
When $\yval_k$ is an outlier, namely when, $\sqrt{\tau_k}\sigma_k
>\sigmaThresh$, we have $\yval_k\asympeq y_k := \Xi(\sqrt{\tau_k}\sigma_k)$ and moreover that $\yval_k^2\Id - \bN^\T \bN$ is invertible. Thus, $\hbv_k = (\yval_k^2\Id - \bN^\T \bN)^{-1} \left(\bN^\T \hbSigma^{-1/2} \bU\bLambda\bbeta_k + \yval_k \bV\bLambda\balpha_k \right) $. Finally, recall that by Lemma~\ref{lem:vecs-off-diagonal}, $\balpha_k\asympeq (\balpha_k)_k \bm{e}_k$, $\bbeta_k\asympeq (\bbeta_k)_k \bm{e}_k$, where $\bm{e}_k$ is the $k$-th standard basis element. Note that one could repeat the same calculation above, starting with $\yval_k^2\hbu_k^w = \Y^w {\Y^w}^\T \hbu_k$. We deduce the following    formulas for the outlying singular vectors:
\begin{align}\label{eq:outlier-vector-rep}
\begin{split}
        \hbv_k &\asympeq \sigma_k \cdot  (\bbeta_k)_k \cdot (y_k^2\Id-\N^\T \N)^{-1}\bN^\T \hbSigma^{-1/2} \bu_k + 
    \sigma_k y_k \cdot (\balpha_k)_k\cdot (y_k^2\Id-\N^\T \N)^{-1}
    \bv_k \,, \\
    \hbu_k^w 
    &\asympeq \sigma_k\cdot (\balpha_k)_k\cdot \left(y_k^2\bI - \bN\bN^\T \right)^{-1} \bN \bv_k +
    \sigma_k y_k\cdot (\bbeta_k)_k\cdot  \left(y_k^2\bI - \bN\bN^\T \right)^{-1} \hbSigma^{-1/2}\bu_k \,.
\end{split}
\end{align}

\begin{lemma}\label{lem:pre-cosines-post}
    Suppose that $\sqrt{\tau_k}\sigma_k > \sigmaThresh$, and set $y_k=\Xi(\sqrt{\tau_k}\sigma_k)$. 
    Let
    \begin{align}
        \underline{c}_{k}^2 =
\sbar(y_k^2)\cdot \frac{\psi(y_k^2)}{\psi'(y_k^2)},\quad
{\tilde{c}_k}^2 = \tau_k \cdot \res(y_k^2)\cdot \frac{\psi(y_k^2)}{\psi'(y_k^2)}
\,,
    \end{align}
    where $\underline{c}_k$ also appears in Theorem~\ref{thm:products}. Then a.s.,
    \begin{equation}
        \lim_{p\to\infty}(\balpha_k)_k^2 = \tilde{c}_k^2, \quad
        \lim_{p\to\infty}(\bbeta_k)_k^2 = \underline{c}_k^2, \quad
        \lim_{p\to \infty}  (\balpha_k)_k(\bbeta_k)_k = \tilde{c}_k \underline{c}_k \,.
    \end{equation}
    
\end{lemma}
\begin{proof}
By definition, $\|\hbv_k\|^2=1$. Consequently, by (\ref{eq:outlier-vector-rep}),
\begin{align*}
    1 
&\asympeq 
\sigma_k^2 (\bbeta_k)_k^2 \cdot \bu_k^\T \hbSigma^{-1/2} \bN (y_k^2\Id-\N^\T \N)^{-2}\bN^\T \hbSigma^{-1/2} \bu_k
+
(\sigma_k y_k)^2 (\balpha_k)_k^2 \cdot \bv_k^\T (y_k^2\Id-\N^\T \N)^{-2}
    \bv_k \\
&+ \sigma_k^2 y_k (\bbeta_k)_k(\balpha_k)_k \cdot 
\bv_k^\T (y_k^2\Id-\N^\T \N)^{-2}\bN^\T \hbSigma^{-1/2} \bu_k \,.
\end{align*}
Note the third term above vanishes asymptotically, by the independence of $\bu_k=\bD_k \bw_k$. Similarly, using $\|\bu_k^w\|^2=1$ in (\ref{eq:outlier-vector-rep}), 
and replacing the quadratic forms by the corresponding traces (as in the proof of Theorem~\ref{thm:detectability-pt} above)
yields the system of equations,
\begin{align}\label{eq:lem3-mat}
\begin{bmatrix}
1/\sigma_k^2 \\ 1/\sigma_k^2
\end{bmatrix} \asympeq
\begin{bmatrix} & p^{-1}\tr\left(\bD_k^\T  \hbSigma^{-1/2} \bN (y_k^2\Id-\N^\T \N)^{-2}\bN^\T \hbSigma^{-1/2} \bD_k \right) 
& y_k^2 \cdot n^{-1}\tr (y_k^2\Id-\N^\T \N)^{-2} \\
&y_k^2 \cdot p^{-1}\tr \left( \bD_k^\T  \hbSigma^{-1/2}(y_k^2\Id-\bN\bN^\T)^{-2}\hbSigma^{-1/2} \bD_k \right)
& n^{-1}\tr \left( \bN^\T (y_k^2\Id-\bN\bN^\T)^{-2}\bN \right)
\end{bmatrix}
\begin{bmatrix}
(\bbeta_k)_k^2\\
(\balpha_k)_k^2
\end{bmatrix} \,.
\end{align}
Next, we calculate limiting expressions for the coefficients. Using Proposition~\ref{prop:stieltjes}, the right column of (\ref{eq:lem3-mat}) is
\begin{align}\label{eq:aux1}
    &y_k^2 \cdot n^{-1}\tr(y_k^2 \Id - \N^\T \bN)^{-2} \asympeq y_k^2 \sbar'(y_k^2), \\
    n^{-1}\tr \left( \bN^\T (y_k^2\Id-\bN\bN^\T)^{-2}\bN \right) 
    &= 
    y_k^2 n^{-1}\tr(y_k^2\Id -\bN\bN^\T)^{-2} - n^{-1}\tr(y_k^2 \Id - \bN\bN^\T)^{-1} 
    \asympeq y_k^2 \sbar'(y_k^2) + \sbar(y_k^2) \,.\nonumber
\end{align}
In (\ref{eq:aux0}) we have calculated:
\begin{align*}
    H_k(z)&=\lim_{p\to\infty} p^{-1}\tr \left( \bD_k^\T  \hbSigma^{-1/2}(z\Id-\bN\bN^\T)^{-1}\hbSigma^{-1/2} \bD_k \right) = -\tau_k\res(z) \\
    -H_k'(z)&=\lim_{p\to\infty} p^{-1}\tr \left( \bD_k^\T  \hbSigma^{-1/2}(z\Id-\bN\bN^\T)^{-2}\hbSigma^{-1/2} \bD_k \right) =  \tau_k\res'(z)
    \,.
\end{align*}
Thus, the bottom left entry of (\ref{eq:lem3-mat}) is $\asympeq \tau_k y_k^2 \res'(y_k^2) $. As for the top left entry, observe that $\bN (y_k^2\Id-\N^\T \N)^{-2}\bN^\T = (y_k^2\Id-\bN\bN^\T)^{-2}\N \bN^\T$. Consequently, this entry tends to $\asympeq -y_k^2 H_k'(y_k^2) - H_k(y_k^2) = \tau_k y_k^2\res'(y_k^2) + \tau_k\res(y_k^2) $. 
Recalling that $1/\sigma_k^2=\tau_k \psi(y_k^2)=y_k^2\res(y_k^2)\sbar(y_k^2)$ and plugging the aforementioned limits in (\ref{eq:lem3-mat}), we deduce that $(\balpha_k)_k^2,(\bbeta_k)_k^2$ asymptotically satisfy
\begin{align}\label{eq:lem3-mat-lim}
\begin{bmatrix}
\tau_k\psi(y_k^2) \\ \tau_k\psi(y_k^2)
\end{bmatrix} \asympeq
\begin{bmatrix} &\tau_k y_k^2\res'(y_k^2) + \tau_k \res(y_k^2)
& y_k^2 \sbar'(y_k^2) \\
& \tau_k y_k^2 \res'(y_k^2)
& y_k^2 \sbar'(y_k^2)+ \sbar(y_k^2)
\end{bmatrix}
\begin{bmatrix}
(\bbeta_k)_k^2\\
(\balpha_k)_k^2
\end{bmatrix} \,.
\end{align}
Solving this system yields the claimed expressions. Lastly, to conclude the calculation of $(\balpha_k)_k (\bbeta_k)_k$, it suffices to show that it is non-negative (since we already computed its modulus). Indeed, by definition,
\begin{align*}
    \yval_k = {\bu_k^w}^\T \bY^w \bv_k = {\bu_k^w}^\T (\hbSigma^{-1/2}\bU\bLambda\bV^\T + \bN) \bv_k \le  \balpha_k^\T \bLambda \bbeta_k + \|\bN\| \,. 
\end{align*}
Now, since $\yval_k$ is an outlier, $\|\bN\|\asympeq \sigmaBulk < y_k \asympeq \yval_k$, hence $\balpha_k^\T \bLambda \bbeta_k > 0$. By Lemma~\ref{lem:vecs-off-diagonal}, $\balpha_k^\T \bLambda \bbeta_k\asympeq \sigma_k (\balpha_k)_k(\bbeta_k)_k$, and so we are done.
\end{proof}

Next, we  calculate the correlations between the recolored empirical singular vectors $\hbSigma^{1/2}\hbu_i^w$ and their corresponding population spikes $\bu_k$.
\begin{lemma}
\label{lem:post-corr}
Suppose that $\sqrt{\tau_k}\sigma_k > \sigmaThresh$. Then Eqs. (\ref{eq:cos}), (\ref{eq:lem:post-corr-2}) and (\ref{eq:cos-off}) hold.
\end{lemma}
\begin{proof}
We start with the representation (\ref{eq:outlier-vector-rep}) and take an inner product with $\hbSigma^{1/2}\bu_\ell$:
\begin{align*}
    \bu_\ell^\T \hbSigma^{1/2}\hbu_k^w \asympeq \sigma_k\cdot (\balpha_k)_k\cdot \bu_\ell^\T \hbSigma^{1/2}\left(y_k^2\bI - \bN\bN^\T \right)^{-1} \bN \bv_k +
    \sigma_k y_k\cdot (\bbeta_k)_k\cdot  \bu_\ell^\T \hbSigma^{1/2}\left(y_k^2\bI - \bN\bN^\T \right)^{-1} \hbSigma^{-1/2}\bu_k \,.
\end{align*}
The first term is asymptotically vanishing, and the second term is possibly non-vanishing only when $k=\ell$ (since the $\bu_k$-s are independent), and so (\ref{eq:cos-off}) is established. We have
\begin{align*}
    \bu_k^\T \hbSigma^{1/2}\left(y_k^2\bI - \bN\bN^\T \right)^{-1} \hbSigma^{-1/2}\bu_k
    &= \bu_k^\T \hbSigma\left(y_k^2\hbSigma - \bSigma^{1/2}\bE\bSigma^{1/2} \right)^{-1} \bu_k = \bu_k^\T \bSigma^{1/2}\bS\left(y_k^2\bS - \bE \right)^{-1}\bSigma^{-1/2} \bu_k
    \\
    &\overset{(i)}{\asympeq} 
    p^{-1}\tr\left(\bS\left(y_k^2\bS - \bE \right)^{-1} \right) \cdot \langle \bSigma^{1/2}\bu_k,\bSigma^{-1/2}\bu_k\rangle 
    \overset{(ii)}{\asympeq} -\stiel(y_k^2)\,,
\end{align*}
where (i) follows from the orthogonal invariance of $\bS(y_k^2\bS-\bE)^{-1}$; and (ii) follows from Proposition~\ref{prop:stieltjes}, writing $\tr( \bS(y_k^2\bS-\bE)^{-1})=\tr (y_k^2\bI-\bS^{-1/2}\bE\bS^{-1/2})^{-1}=\tr(y_k^2\bI-\bN\bN^\T)^{-1}$  (recall that the Stieltjes transform does not depend on $\bSigma$) and the normalization $\|\bu_k\|^2\asympeq p^{-1}\|\bD_k\|_F^2\asympeq1$. Recall that $(\bv_k^\T \hbv_k^w)^2=(\bbeta_k)_k^2\asympeq \underline{c}_k^2$ was computed in Lemma~\ref{lem:pre-cosines-post}. Thus,
\begin{align*}
    (\bu_k^\T \hbSigma^{1/2}\hbu_k^w)^2 \asympeq \sigma_k^2 y_k^2 \cdot \underline{c}_k^2 \cdot \stiel(y_k)^2,\quad 
    (\bu_k^\T \hbSigma^{1/2}\hbu_k^w)(\bv_k^\T \hbv_k^w) \asympeq -\sigma_k y_k \cdot \underline{c}_k^2 \cdot \stiel(y_k) \,.
\end{align*}
Finally, use $\sigma_k^2 = 1/\left(\tau_k \psi(y_k^2)\right)$ to get the expressions in Eqs. (\ref{eq:cos}) and (\ref{eq:lem:post-corr-2}).
\end{proof}

It remains to compute $\|\hbSigma^{1/2}\hbu_i^w\|^2$. To this end, multiply (\ref{eq:outlier-vector-rep}) by $\hbSigma^{1/2}$ and take the norm,
\begin{align*}
    \|\hbSigma^{1/2}\hbu_k^w\|^2 
    &\asympeq \left\| \sigma_k (\balpha_k)_k \hbSigma^{1/2}\left(y_k^2\bI - \bN\bN^\T \right)^{-1} \bN \bv_k +
    \sigma_k y_k (\bbeta_k)_k  \hbSigma^{1/2}\left(y_k^2\bI - \bN\bN^\T \right)^{-1} \hbSigma^{-1/2}\bu_k \right\|^2 \\
    &\overset{(i)}{\asympeq} 
    \sigma_k^2 \tilde{c}_k^2 \bv_k^\T \bN^\T \left(y_k^2\bI - \bN\bN^\T \right)^{-1} \hbSigma\left(y_k^2\bI - \bN\bN^\T \right)^{-1} \bN \bv_k \\
    &\quad+ \sigma_k^2 y_k^2\underline{c}_k^2 \bu_k^\T \hbSigma^{-1/2}(y_k^2-\bN\bN^\T)^{-1} \hbSigma (y_k^2-\bN\bN^\T)^{-1}\hbSigma^{-1/2}\bu_k \\
    &\overset{(ii)}{\asympeq} 
    \sigma_k^2 \tilde{c}_k^2\cdot n^{-1}\tr\left(\bN^\T \left(y_k^2\bI - \bN\bN^\T \right)^{-1} \hbSigma\left(y_k^2\bI - \bN\bN^\T \right)^{-1} \bN  \right)\\
    &\quad + \sigma_k^2 y_k^2\underline{c}_k^2 \cdot p^{-1}\tr \left( \bD_k^\T \hbSigma^{-1/2}(y_k^2\Id-\bN\bN^\T)^{-1} \hbSigma (y_k^2\Id-\bN\bN^\T)^{-1}\hbSigma^{-1/2}\bD_k \right)\,,
\end{align*}
where: (i) we discarded the cross term, which is asymptotically vanishing; and (ii) similarly to previous calculations, the quadratic forms concentrate around the traces. Define
\begin{equation}
    \begin{split}
        G_1(z) &= \lim_{p\to\infty} n^{-1}\tr\left(\bN^\T \left(z\bI - \bN\bN^\T \right)^{-1} \hbSigma\left(z\bI - \bN\bN^\T \right)^{-1} \bN  \right),\\
        G_{2,k}(z) &= 
        \lim_{p\to\infty} p^{-1}\tr \left( \bD_k^\T \hbSigma^{-1/2}(z\Id-\bN\bN^\T)^{-1} \hbSigma (z\Id-\bN\bN^\T)^{-1}\hbSigma^{-1/2}\bD_k \right) \,,
    \end{split}
\end{equation}
so that 
\begin{equation}\label{eq:aux:norm}
    \|\hbSigma^{1/2}\hbu_k^w\|^2 \asympeq \sigma_k^2 \tilde{c}_k^2\cdot G_1(y_k^2) + \sigma_k^2 y_k^2\underline{c}_k^2 \cdot G_{2,k}(y_k^2) \,.
\end{equation}
Also define the following mixed traces
\begin{align}
\label{eq:UpsOne-def-trace}
    &\UpsOne(z) = \lim_{p\to\infty} p^{-1} \tr (z\bS-\bE)^{-1}\bS^2,\\
\label{eq:UpsTwo-def-trace}
    &\UpsTwo(z) = \lim_{p\to\infty} p^{-1}\tr (z\bS-\bE)^{-2}\bS^2 \,,
\end{align}
where $z\in \CC\setminus [\sigmaMin^2,\sigmaBulk^2]$ and the limit exists a.s. We show in  Appendix, Section~\ref{sec:New:Ups} that $\UpsOne(\cdot),\UpsTwo(\cdot)$ have the closed-form formulas given in Eqs. (\ref{eq:New:UpsOneFormula}) and (\ref{eq:New:UpsTwoFormula}). 
\begin{lemma}
\label{lem:G1}
    Let $\UpsOne(\cdot)$ be defined in (\ref{eq:UpsOne-def-trace}), and recall the notation $\mu=p^{-1}\tr(\bSigma)$. Then 
    \begin{equation}
        G_1(z) = -\gamma \mu \cdot \left[ z\UpsOne(z) + \UpsOne'(z)   \right] \,,\quad z\in \CC\setminus [\sigmaMin^2,\sigmaBulk^2] .
    \end{equation}
\end{lemma}
\begin{proof}
    Let 
    \begin{equation*}
        \tilde{G}_1(z) = \lim_{n\to\infty} n^{-1}\tr\left( (z\Id - \bN\bN^\T)^{-1}\hbSigma \right),\quad\textrm{so that} \quad G_1(z) = -z\tilde{G}_1'(z) - \tilde{G}_1(z)\,.
    \end{equation*}
    Straightforward algebraic manipulation yields $\tr\left( (z\Id - \bN\bN^\T)^{-1}\hbSigma \right) = \tr\left( \bS(z\bS-\bE)^{-1}\bS\bSigma \right)$. Since $\bS(z\bS-\bE)^{-1}\bS$ is orthogonally invariant, $\{\bS(z\bS-\bE)^{-1}\bS,\bSigma\}$ are asymptotically free (see Section~\ref{sec:free-probability}). Thus,
    \[
    \lim_{p\to\infty} p^{-1} \tr\left( \bS(z\bS-\bE)^{-1}\bS\bSigma \right) = \trlim(\bS(z\bS-\bE)^{-1}\bS)\cdot  \trlim(\bSigma) = \mu\cdot \UpsOne(z) \,,
    \]
    and the lemma follows.
\end{proof}


\begin{lemma}
\label{lem:G2}
    Let $\UpsTwo(\cdot)$ be defined in (\ref{eq:UpsTwo-def-trace}). Then
      \begin{equation}
        G_{2,k}(z) = \mu\tau_k\left[ \UpsTwo(z)-\left(\stiel(z)\right)^2 \right] + \left(\stiel(z)\right)^2\,,\quad z\in \CC\setminus [\sigmaMin^2,\sigmaBulk^2] .
    \end{equation}
\end{lemma}
We provide a proof in Appendix, Section~\ref{sec:proof-lem:G2}. To conclude the computation of $\|\hbSigma^{1/2}\hbu_k^w\|^2$, and thereby the proof of Theorem~\ref{thm:products}, combine (\ref{eq:aux:norm}) with Lemmas~\ref{lem:G1} and \ref{lem:G2}. 
\end{proof}

\section{Numerical results}
\label{sec:numerical}

We report on several numerical experiments that illustrate both the performance of the denoising algorithm relative to other methods, and the agreement between the asymptotic results and finite sample estimates for both Gaussian and non-Gaussian noise. One of the questions addressed in this section is under what parameter settings the optimal Whiten-Shrink- re-Color (WSC) denoiser outperforms OptShrink, which is optimal shrinkage without any pre-transformation \cite{nadakuditi2014optshrink}. It was shown in \cite{leeb2021optimal} that when the signal principal components are uniformly random, optimal shrinkage with \emph{oracle} whitening (corresponding to $\beta = 0$) outperforms OptShrink for heteroscedastic noise; consequently, optimal shrinkage with pseudo-whitening (when $\beta > 0$) will still outperform OptShrink when $\beta$ is sufficiently small, though the precise value of $\beta$ will depend on the other problem parameters, such as $\gamma$ and the heteroscedasticity of the noise. 

Sections \ref{sec:errors} and \ref{sec:critical_beta}  numerically illustrate this behavior, by comparing the errors of the two methods in different parameter regimes and evaluating the $\beta$ at which the two methods have identical error. Section \ref{sec:detection} explores a different but related phenomenon, comparing the minimum detectable signal singular value of the two methods. Section \ref{sec:pc_estimation} examines the two methods' performances in estimating the signal principal components. Section \ref{sec:comparison} compares different shrinkage methods on finite-sample data, again showing that the relative performance of pseudo-whitening over OptShrink increases as the heteroscedasticity of the noise increases.

The experiments in Section \ref{sec:convergence_nongaussian} illustrate two phenomena: first, that the asymptotic results appear to hold for non-Gaussian noise with sufficiently many moments (``universality''), though they break down when the noise becomes too fat-tailed; and second, that the spiked model parameters appear to converge at the rate of approximately $O(p^{-1/2})$ for thin-tailed distributions.

All experiments reported in this section were performed in Matlab. Code may be found online at \texttt{https://github.com/wleeb/FShrink}.

\subsection{Plots of optimal shrinkers}

We plot the  asymptotically optimal singular values $t_1$ as functions of the population singular values and of the observed singular values, for $\gamma = 1/2$, covariance $\bSigma \in \RR^{1000 \times 1000}$ with eigenvalues linearly spaced between $1/500$ and $1$, and $\bD_1 = \bI_{1000}$. The optimal singular values are computed for different values of $\beta$, and plotted in Figure \ref{figure.vals}. Note that $\beta=0$ corresponds to the oracle singular values (where $\bSigma$ is known exactly), used in \cite{leeb2021optimal}. For larger $\beta$, the optimal singular values decrease in value; that is, more shrinkage is needed to account for the increased uncertainty in the estimate of $\bSigma$.

\begin{figure}
\centering
\begin{subfigure}{0.4\textwidth}
\includegraphics[scale=.4]{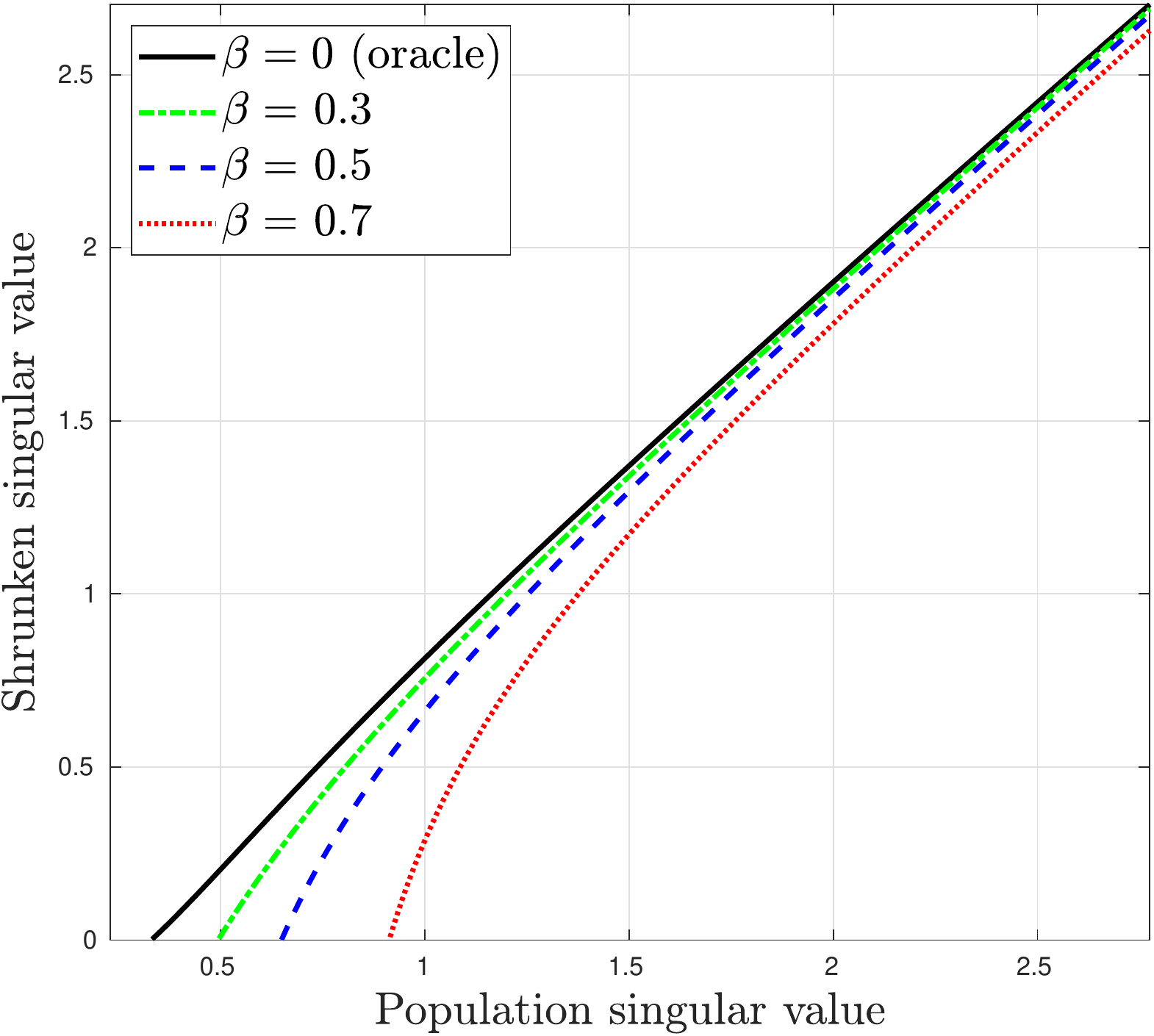}
\end{subfigure}
\hspace{.5in}
\begin{subfigure}{0.4\textwidth}
\includegraphics[scale=.4]{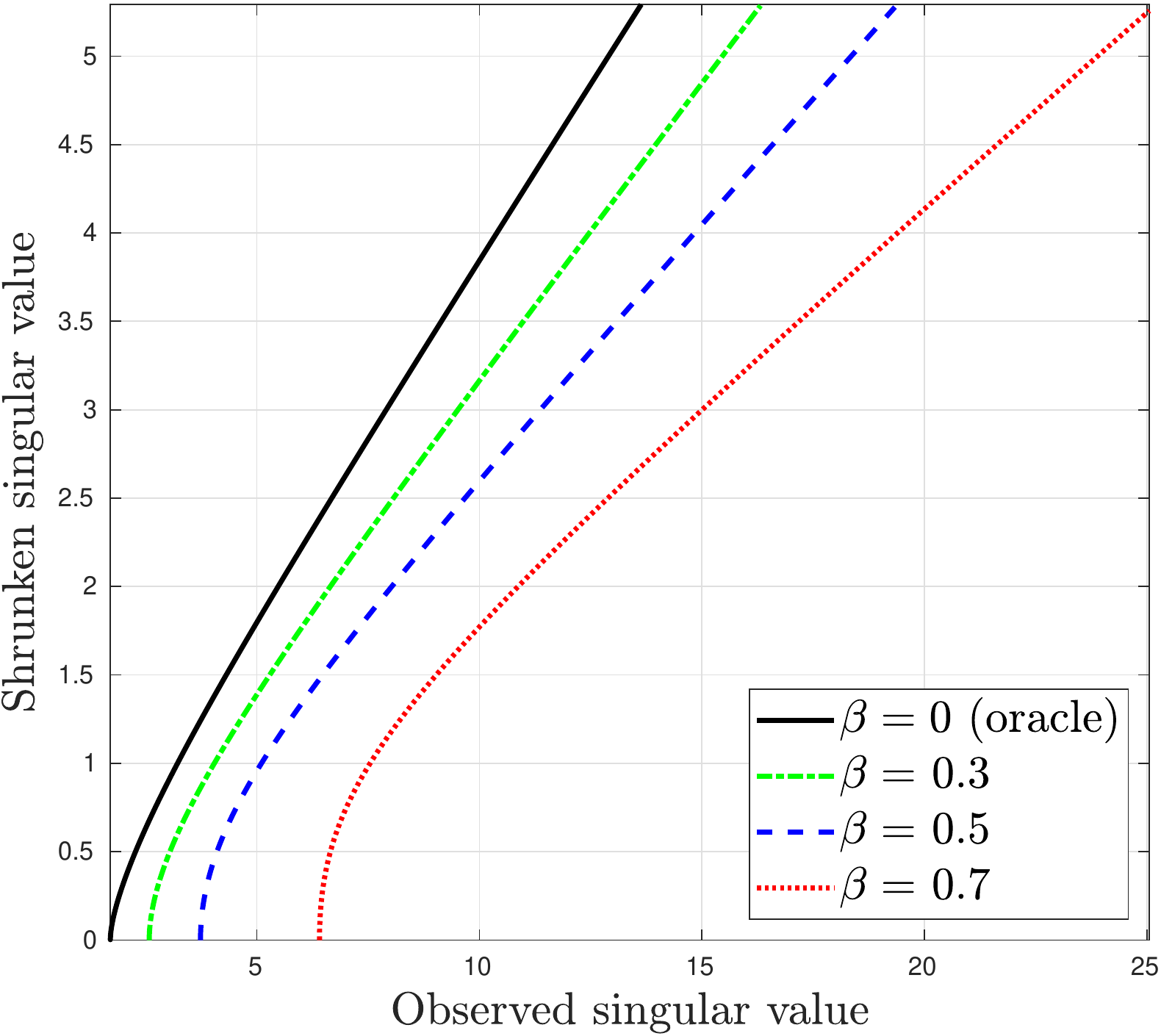}
\end{subfigure}
\caption{Plots of the asymptotically optimal singular values as functions of the population singular value (left) and of the observed singular value (right), for different values of $\beta$.}
\label{figure.vals}
\end{figure}

\subsection{Asymptotic errors}
\label{sec:errors}

The previous work \cite{leeb2021optimal} shows, assuming isotropic random spike directions, that an oracle WSC denoiser (which has access to $\bSigma$; equivalently, $\beta=0$ under our setup) 
outperforms optimal singular value shrinkage (OptShrink) \cite{nadakuditi2014optshrink}. Consequently, pseudo-whitening must also  outperform OptShrink provided that $\beta$ is sufficiently small, that is, one has access to sufficiently many pure-noise samples.  Though the precise value of $\beta$ at which optimal shrinkage with pseudo-whitening has superior performance is not immediately transparent from the error formulas, it can be evaluated numerically. In this set of experiments, we compare the asymptotic errors of optimal shrinkage with pseudo-whitening, comparing them to the asymptotic errors obtained by OptShrink; for context, we also show a comparison with the oracle WSC denoiser from \cite{leeb2021optimal} (which outperforms both methods, but is not a viable procedure in the setting of this paper).

\subsubsection{Linearly-spaced eigenvalue spectrum}
\label{sec:lin_spaced}

For a specified aspect ratio $\gamma > 0$ and condition number $\kappa \ge 1$, we consider a diagonal noise covariance $\bSigma$ with $p=2000$ linearly spaced eigenvalues whose maximum and minimum elements have ratio $\kappa$, and that are normalized so that $\tau = 1$; that is,
\begin{align}
\label{eq:unit_tau}
\frac{1}{p} \sum_{i=1}^{p} \frac{1}{\bSigma_{ii}} = 1.
\end{align}
For $\kappa = 50,000$, we compute $\sigmaBBP$, the asymptotically largest singular value of the noise matrix $\bSigma^{1/2} \bZ / \sqrt{n}$, using the numerical method introduced in \cite{leeb2021rapid} (the term BBP is from the paper \cite{baik2005phase}); this method evaluates $\sigmaBBP$ by numerically solving the equations that implicitly characterize $\sigmaBBP$. We consider a rank $1$ $p$-by-$n$ signal matrix $\bX$ with i.i.d. singular vectors $\bu$ and $\bv$ and singular value
\begin{math}
\sigma = \sigmaBBP + 1.
\end{math}
%
For these parameters, the asymptotic errors for optimal shrinkage with exact whitening are then computed using the formula from \cite{leeb2021optimal}, while the asymptotic error for OptShrink is evaluated numerically using the method from \cite{leeb2021rapid}, which numerically solves the Stieltjes transform of the limiting spectral distribution of the sample covariance of the noise. For optimal shrinkage with pseudo-whitening, we consider a range of aspect ratios $\beta$ between $0$ and $0.95$, and evaluate the AMSE for each value of $\beta$. These values are plotted as functions of $\beta$ in Figure \ref{figure.errs}, alongside the asymptotic errors of the other two methods. As the plots demonstrate, for each value of $\gamma$ and $\kappa$ the error of shrinkage with pseudo-whitening approaches that of shrinkage with exact whitening as $\beta \to 0$. Furthermore, for each value of $\gamma$, as $\kappa$ grows (that is, as the noise becomes more heteroscedastic) the performance of shrinkage with pseudo-whitening improves relative to OptShrink; when $\kappa = 50,000$, shrinkage with pseudo-whitening outperforms OptShrink for the entire considered range of $\beta$.

\begin{figure}
\center
\includegraphics[scale=.5]{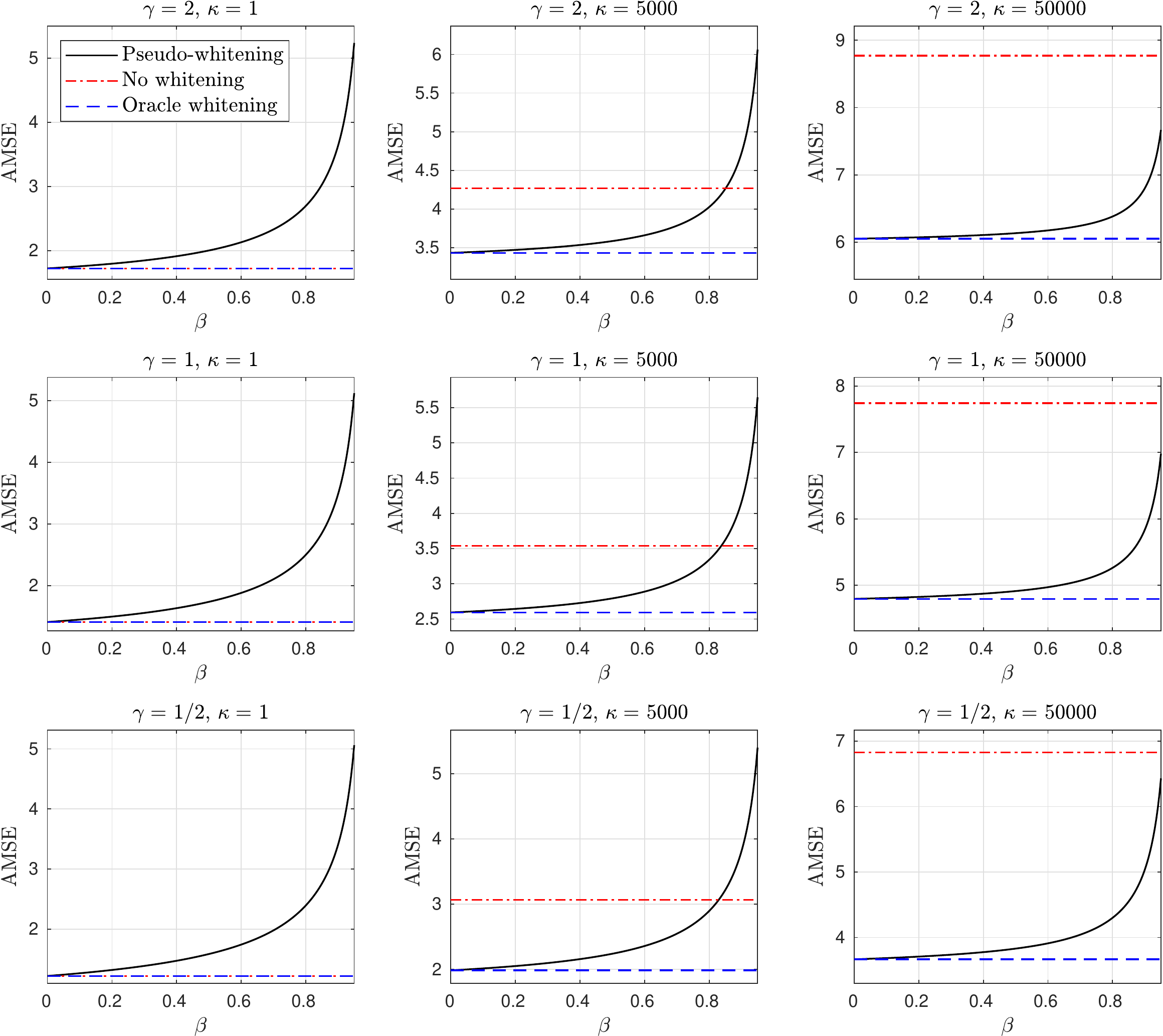}
\caption{Asymptotic errors for the optimal WSC denoiser with pseudo-whitening as a function of $\beta \sim p/m$ for different values of $\gamma \sim p/n$ and parameter $\kappa$, where the noise covariance has linearly spaced eigenvalues and condition number $\kappa$. For reference, we also plot the asymptotic errors for the oracle WSC denoiser
(from \cite{leeb2021optimal}) and optimal shrinkage without whitening (OptShrink, from \cite{nadakuditi2014optshrink}). The errors for OptShrink are numerically evaluated using the method from \cite{leeb2021rapid}. Our numerical results demonstrate that the performance gains offered by whitening become more pronounced as the condition number $\kappa$ increases.
} 
\label{figure.errs}
\end{figure}

\subsubsection{Polynomially-decaying eigenvalue spectrum}
\label{sec:poly_decay}

For a specified aspect ratio $\gamma > 0$ and parameter $\alpha$, we consider a diagonal noise covariance $\bSigma$ with $p=2000$ eigenvalues of the form $C \cdot t^\alpha$, where $t$ are equispaced between $1$ and $3$, and where $C$ is chosen so that $\tau = 1$, i.e.\ \eqref{eq:unit_tau} holds. For $\alpha = 5$, we compute $\sigmaBBP$, the asymptotically largest singular value of the noise matrix $\bSigma^{1/2} \bZ / \sqrt{n}$, using the numerical method introduced in \cite{leeb2021rapid}. We consider a rank $1$ $p$-by-$n$ signal matrix $\bX$ with i.i.d. singular vectors $\bu$ and $\bv$ and singular value
\begin{math}
\sigma = \sigmaBBP + 1.
\end{math}
%
For these parameters, the asymptotic errors for optimal shrinkage with exact whitening are then computed using the formula from \cite{leeb2021optimal}, while the asymptotic error for OptShrink is evaluated numerically using the method from \cite{leeb2021rapid}. For optimal shrinkage with pseudo-whitening, we consider a range of aspect ratios $\beta$ between $0$ and $0.95$, and evaluate the AMSE for each value of $\beta$. These values are plotted as functions of $\beta$ in Figure \ref{figure.errs2}, alongside the asymptotic errors of the other two methods. As the plots demonstrate, for each value of $\gamma$ and $\alpha$ the error of shrinkage with pseudo-whitening approaches that of shrinkage with exact whitening as $\beta \to 0$. Furthermore, for each value of $\gamma$, as $\alpha$ grows (that is, as the noise becomes more heteroscedastic) the performance of shrinkage with pseudo-whitening improves relative to OptShrink; for example, when $\gamma=2$ and $\alpha = 5$, shrinkage with pseudo-whitening outperforms OptShrink for the entire considered range of $\beta$.

\begin{figure}
\center
\includegraphics[scale=.5]{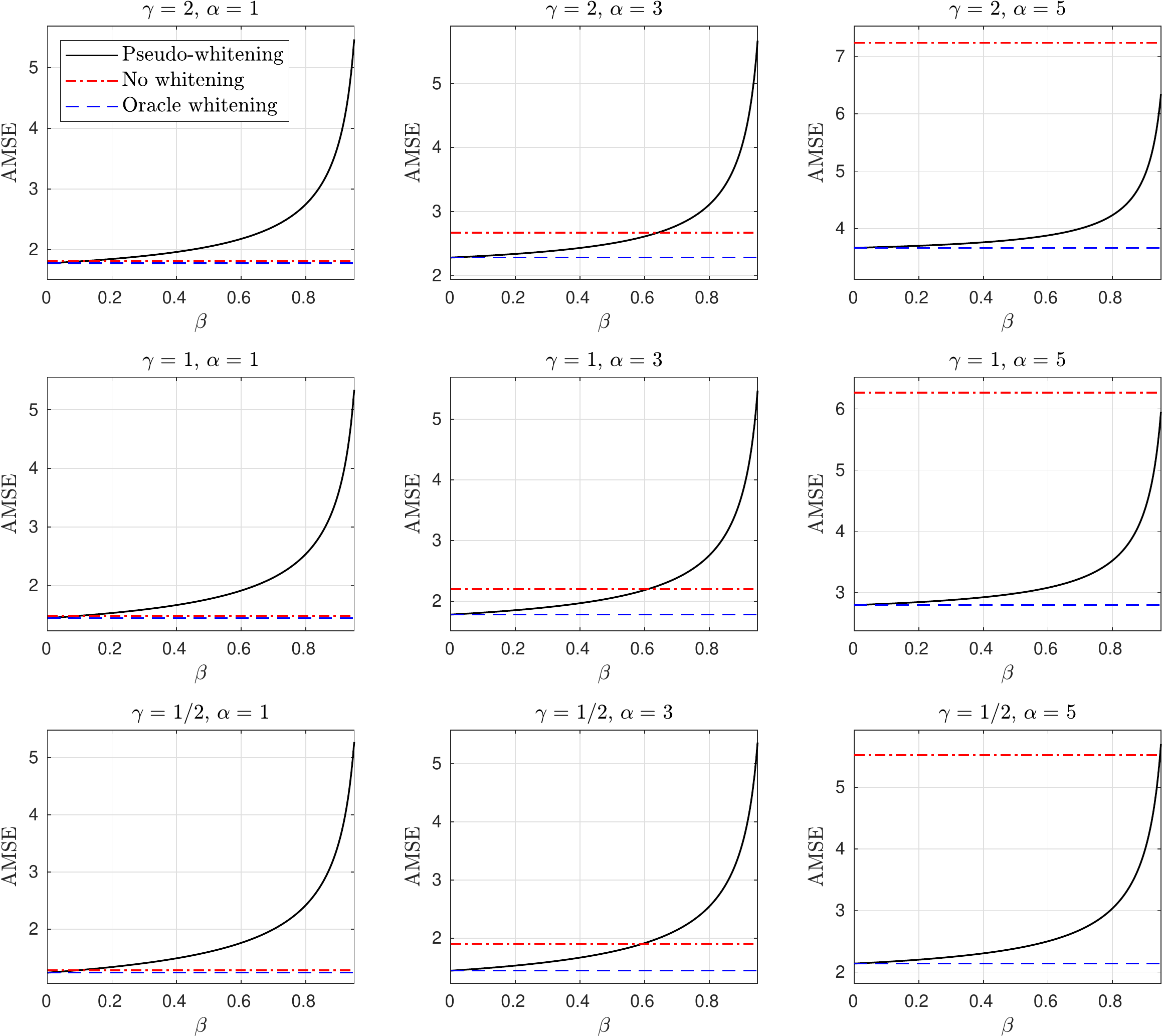}
\caption{Asymptotic errors for optimal shrinkage with pseudo-whitening as a function of $\beta \sim p/m$ for different values of $\gamma \sim p/n$ and parameter $\alpha$, with noise eigenvalues of the form $C \cdot t^{\alpha}$ with $t$ equispaced between $1$ and $3$. For reference, we also plot the asymptotic errors for the oracle WSC denoiser (from \cite{leeb2021optimal}) and optimal shrinkage without whitening (OptShrink, from \cite{nadakuditi2014optshrink}). The errors for OptShrink are numerically evaluated using the method from \cite{leeb2021rapid}. Our numerical results demonstrate that the performance gains offered by whitening become more pronounced as $\alpha$ increases, that is the eigenvalues of the noise covariance deviate further from a constant profile.
}
\label{figure.errs2}
\end{figure}

\subsection{Critical value of $\beta$}
\label{sec:critical_beta}

\begin{figure}
\centering
\begin{subfigure}{0.4\textwidth}
\includegraphics[scale=.3]{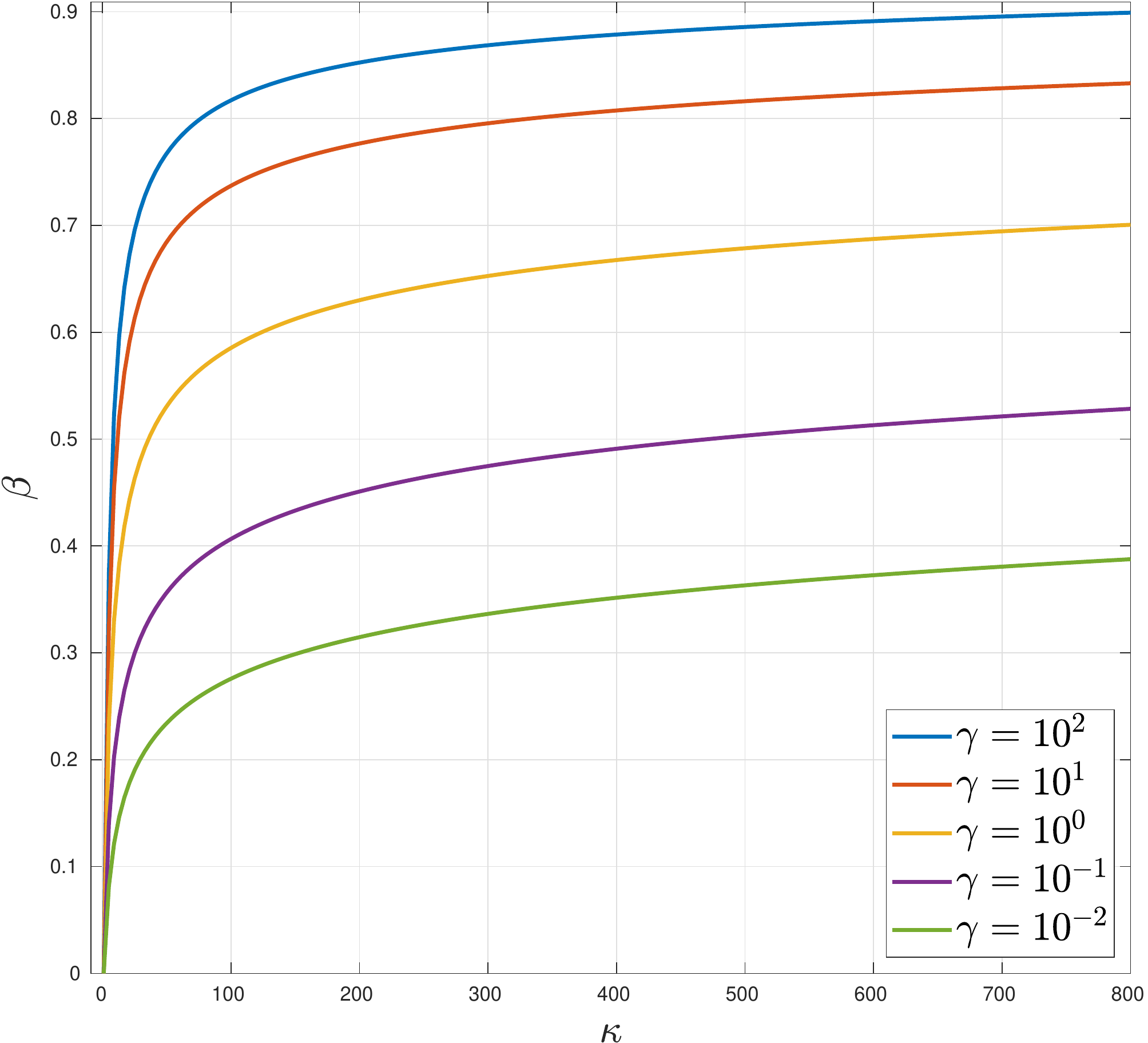}
\end{subfigure}
\hspace{.5in}
\begin{subfigure}{0.4\textwidth}
\includegraphics[scale=.3]{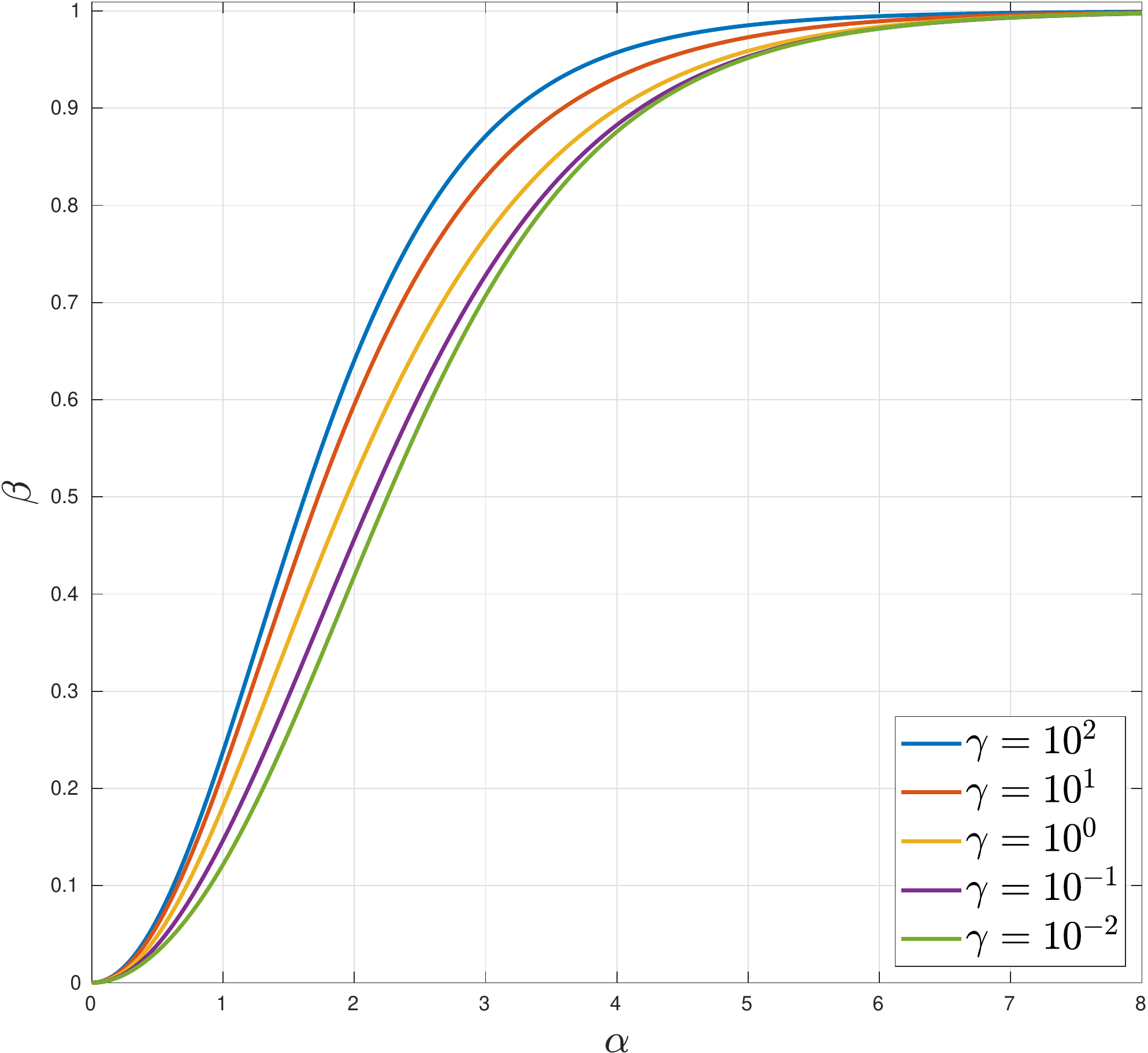}
\end{subfigure}
\caption{Value of $\beta$ for which the AMSE of OptShrink is equal to the AMSE of optimal singular value shrinkage with pseudo-whitening. Left: critical value of $\beta$ plotted as a function of $\kappa$, where the noise covariance has linearly spaced eigenvalues and condition number $\kappa$. Right: critical value of $\beta$ plotted as a function of $\alpha$, with noise eigenvalues of the form $C \cdot t^{\alpha}$ with $t$ equispaced between $1$ and $3$. It is observed that the critical value of $\beta$ grows as the heteroscedasticity parameters ($\kappa$ or $\alpha$) grow; that is, the more heteroscedastic the noise, the fewer noise-only samples are needed for optimal shrinkage with pseudo-whitening to outperform OptShrink.}
\label{figure.betas_errs}
\end{figure}

For another view of the experiments in Section \ref{sec:errors}, we evaluate the value of $\beta$ at which optimal shrinkage with pseudo-whitening outperforms OptShrink. When $\beta=0$ (i.e.\ when oracle whitening is used), pseudo-whitening is guaranteed to outperform OptShrink \cite{leeb2021optimal}; the performance of pseudo-whitening degrades as $\beta$ approaches $1$, and for some parameter values pseudo-whitening will perform worse than OptShrink when $\beta$ is sufficiently close to $1$. To illustrate this phenomenon, we consider the same two families of noise covariances described in Sections \ref{sec:lin_spaced} and \ref{sec:poly_decay}. Holding all other parameter values fixed, we consider the error of optimal shrinkage with pseudo-whitening as a function of $\beta$, and solve for the value of $\beta$ at which the pseudo-whitening error is equal to that of OptShrink when the signal strength is equal to $1.1 \sigmaBBP$. The value of $\sigmaBBP$ and the errors for OptShrink are evaluated using the method from \cite{leeb2021rapid}, as in Section \ref{sec:errors}; and the critical value of $\beta$ is numerically evaluated using the secant method.

The left panel of Figure \ref{figure.betas_errs} plots the critical value of $\beta$ for the Section \ref{sec:lin_spaced} model as a function of the condition number $\kappa$; the right panel plots the critical value of $\beta$ for the Section \ref{sec:poly_decay} model as a function of the decay parameter $\alpha$. These curves are shown for a range of values of $\gamma$. Both plots reveal similar qualititative behavior: the critical value of $\beta$ grows as the heteroscedasticity parameters ($\kappa$ or $\alpha$) grow; that is, the more heteroscedastic the noise, the fewer noise-only samples are needed for optimal shrinkage with pseudo-whitening to outperform OptShrink. Furthermore, the relative performance of optimal shrinkage with pseudo-whitening also increases with $\gamma$.

\subsection{Minimum detectable signal}
\label{sec:detection}

To further illustrate the differences in performance between optimal shrinkage with pseudo-whitening and OptShrink as a function of $\beta$, we consider the smallest singular value of $\bX$ that is detectable under each noise model. We denote by $\thetaBBP$ the smallest detectable singular value of $\bX$; for any specified variance profile, this may be numerically evaluated using the method from \cite{leeb2021rapid}, which solves the equations that implicitly characterize the value. We also consider the minimum detectable signal singular value by pseudo-whitening, given by the formula \eqref{eq:sigma-thresh}; this value obviously grows with $\beta$. Figures \ref{figure.bbp_conds} and \ref{figure.bbp_alphas} plot the minimum detectable values under OptShrink as functions of the heteroscedasticity of the noise, as measured by the parameter $\kappa$ for the model in Section \ref{sec:lin_spaced} and the parameter $\alpha$ for the model in Section \ref{sec:poly_decay}, respectively. These figures also display the values of the largest $\beta$ where the two methods have identical signal detection thresholds, again as functions of the heteroscedasticity of the noise covariances.

\begin{figure}
\centering
\begin{subfigure}{0.4\textwidth}
\includegraphics[scale=.3]{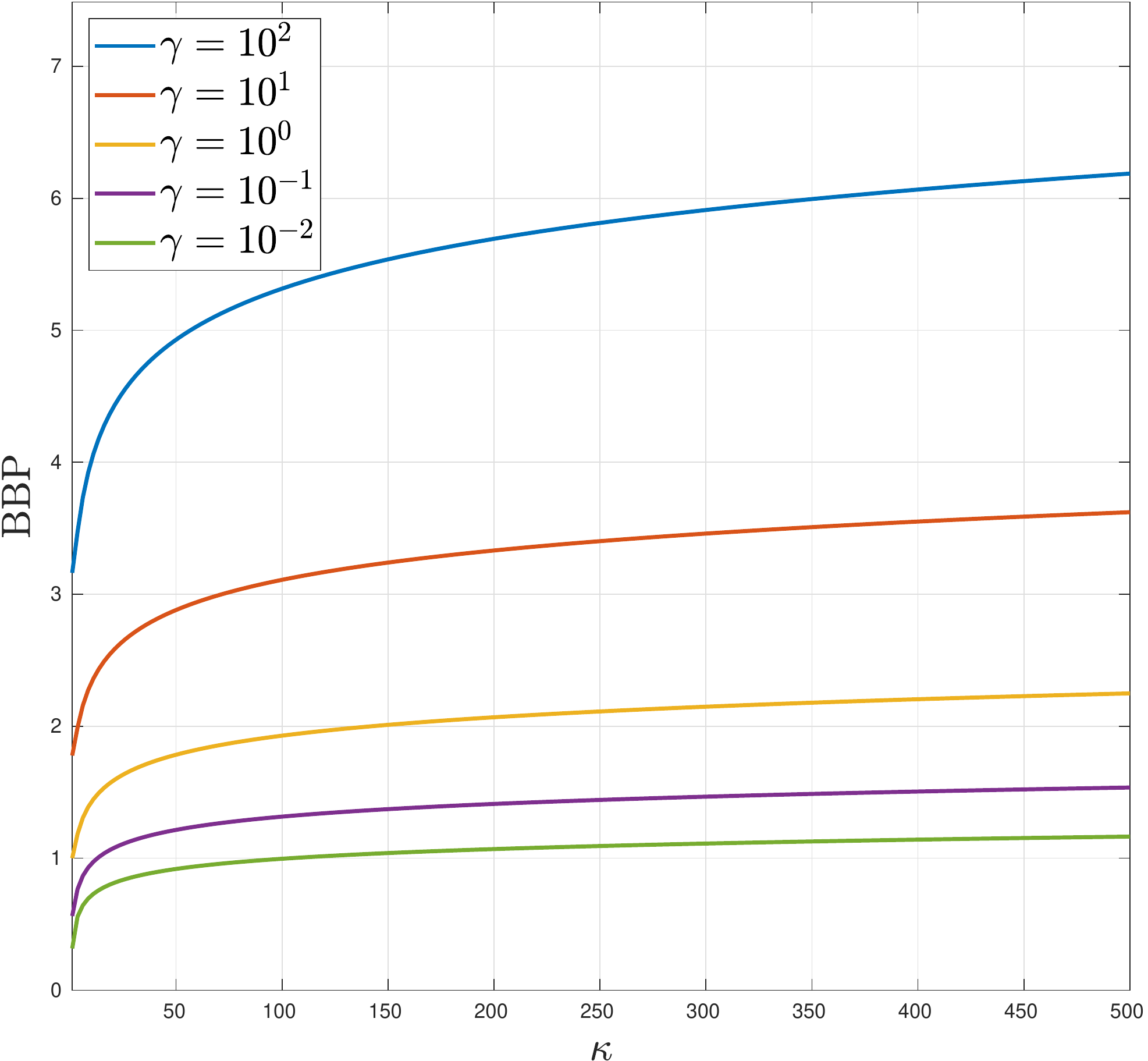}
\end{subfigure}
\hspace{.5in}
\begin{subfigure}{0.4\textwidth}
\includegraphics[scale=.3]{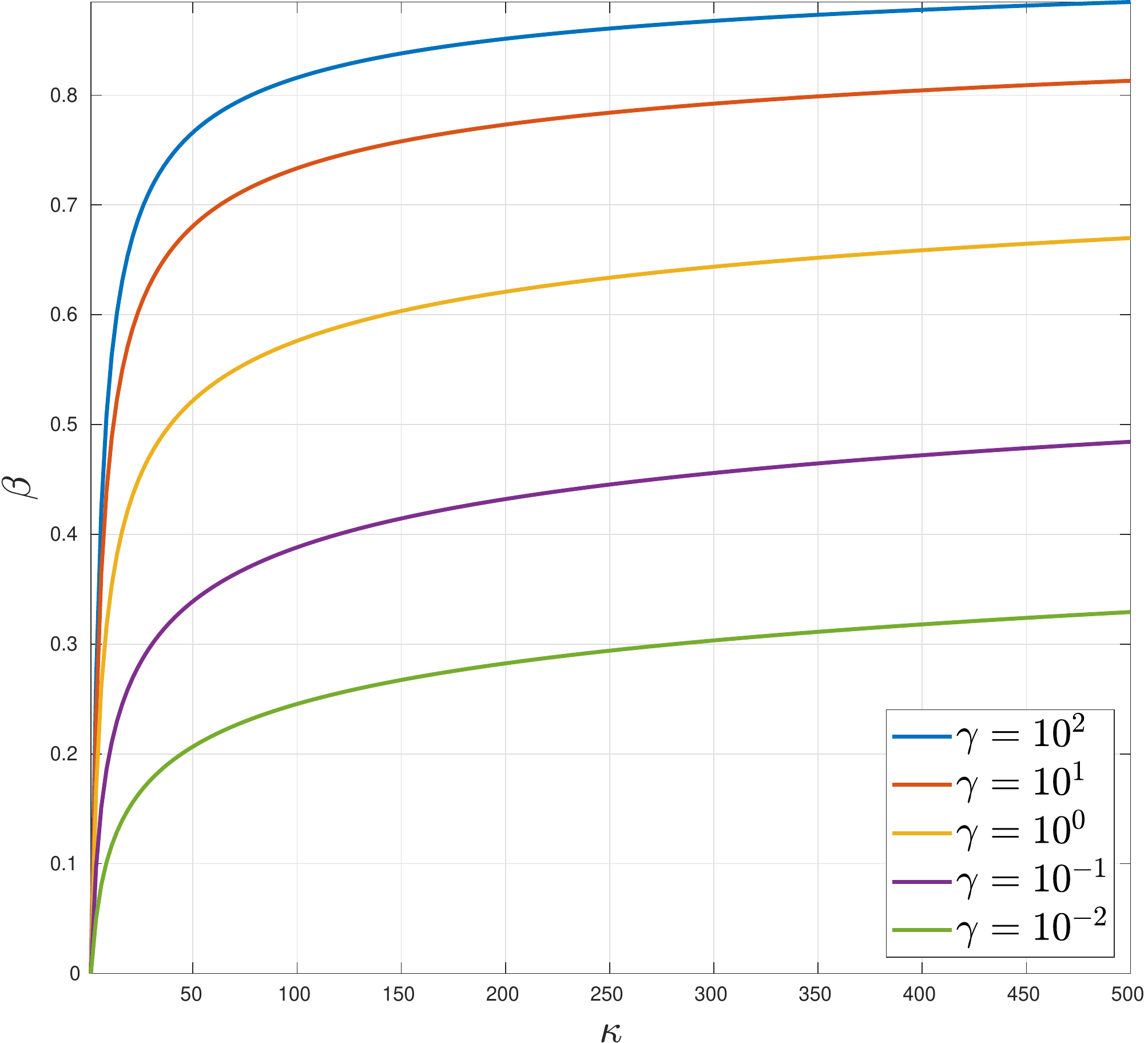}
\end{subfigure}
\caption{Left: the minimum detectable signal singular value by OptShrink. Right: value of $\beta$ for which the minimum detectable signal singular value for the pseudo-whitened matrix equals the minimum detectable value for OptShrink. Values are plotted as functions of $\kappa$, where the noise covariance has linearly spaced eigenvalues and condition number $\kappa$.}
\label{figure.bbp_conds}
\end{figure}

\begin{figure}
\centering
\begin{subfigure}{0.4\textwidth}
\includegraphics[scale=.3]{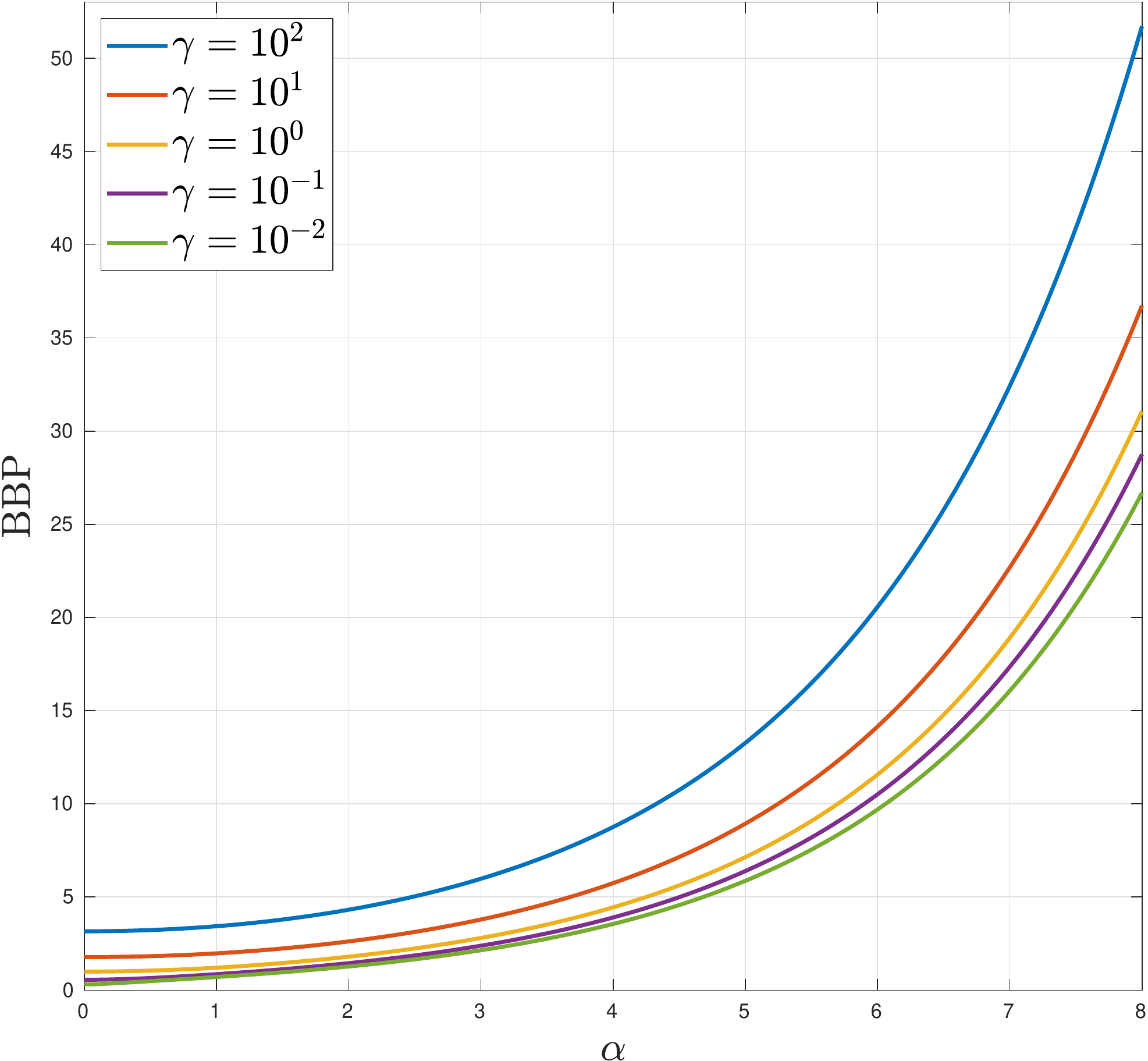}
\end{subfigure}
\hspace{.5in}
\begin{subfigure}{0.4\textwidth}
\includegraphics[scale=.3]{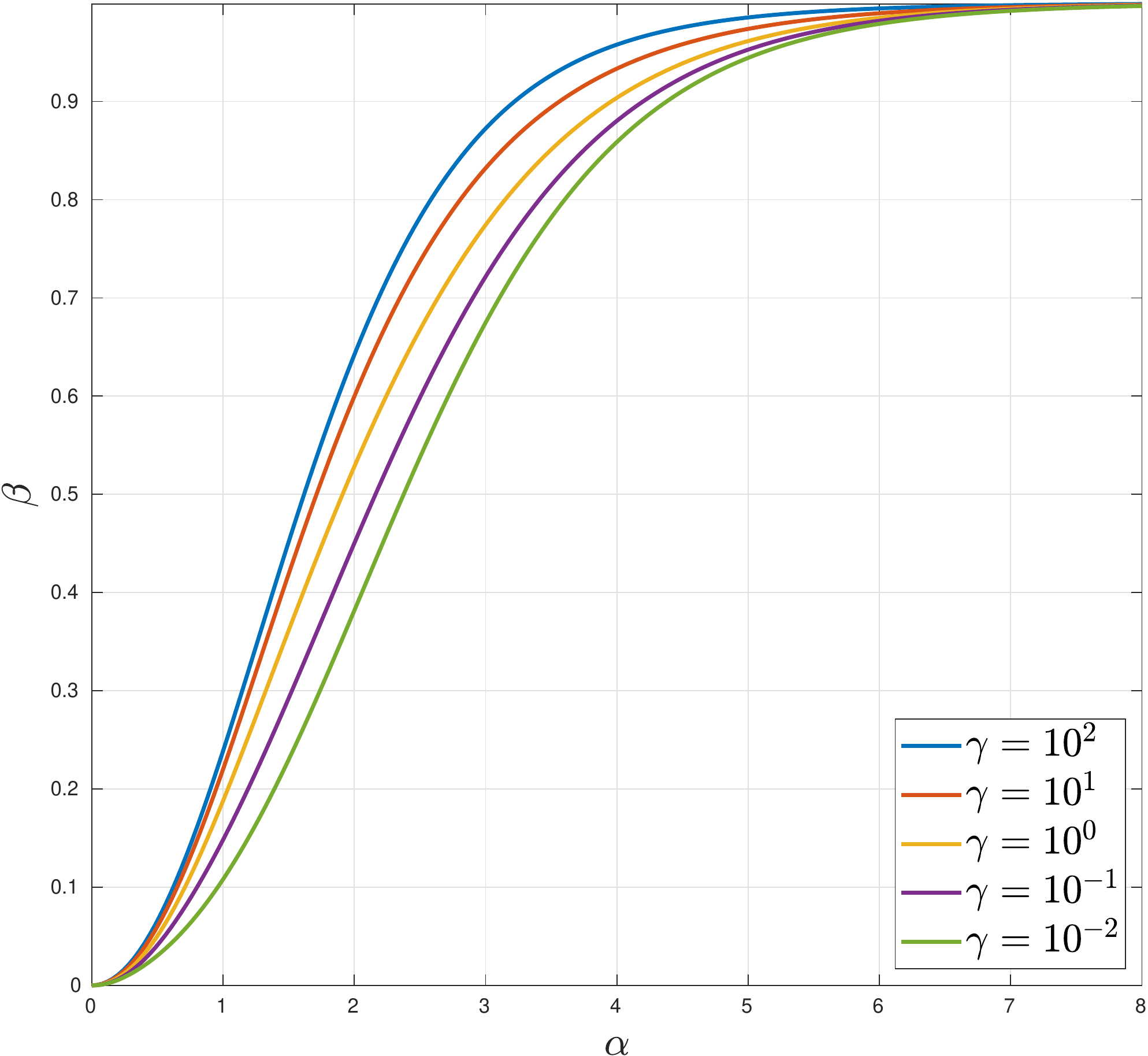}
\end{subfigure}
\caption{Left: the minimum detectable signal singular value by OptShrink. Right: value of $\beta$ for which the minimum detectable signal singular value for the pseudo-whitened matrix equals the minimum detectable value for OptShrink. Values are plotted as functions of $\alpha$, where the noise covariance eigenvalues are of the form $C \cdot t^{\alpha}$ with $t$ equispaced between $1$ and $3$.}
\label{figure.bbp_alphas}
\end{figure}

\subsection{Principal component estimation}
\label{sec:pc_estimation}

Next, we compare the methods of pseudo-whitening to OptShrink in estimating the principal components of the signal vector. We fix the parameter $\gamma=1/2$ and consider the model from Section \ref{sec:lin_spaced}, where the noise covariance has $p=2000$ linearly spaced spectrum and is normalized so that $\tau=1$. Figure \ref{figure.cosines} plots the asymptotic absolute inner products between the estimated PCs under pseudo-whitening, for several values of $\beta$. Also shown are the asymptotic absolute inner products between the true PCs and the estimated PCs used by OptShrink (the top left singular vector of the unnormalized data matrix). The signal singular value is taken to be $\max\{\sigmaThresh,\sigmaBBP\} + 1$, where $\sigmaThresh$ is evaluated for the largest value of $\beta$ considered ($\beta = .5$). Both $\sigmaBBP$ and the cosines for the unwhitened PCs are evaluated using the method from \cite{leeb2021optimal}.

As is apparent from the plot, larger values of $\beta$ result in smaller cosines; that is, when pseudo-whitening is performed with fewer noise-only samples, the resulting estimates are worse. Second, as the condition number grows -- that is, as the noise becomes more heteroscedastic -- the relative performance of pseudo-whitening improves relative to estimation without whitening. At a certain condition number, each pseudo-whitening curve begins to outperform the estimates without any whitening.

\begin{figure}
\centering
\includegraphics[scale=.35]{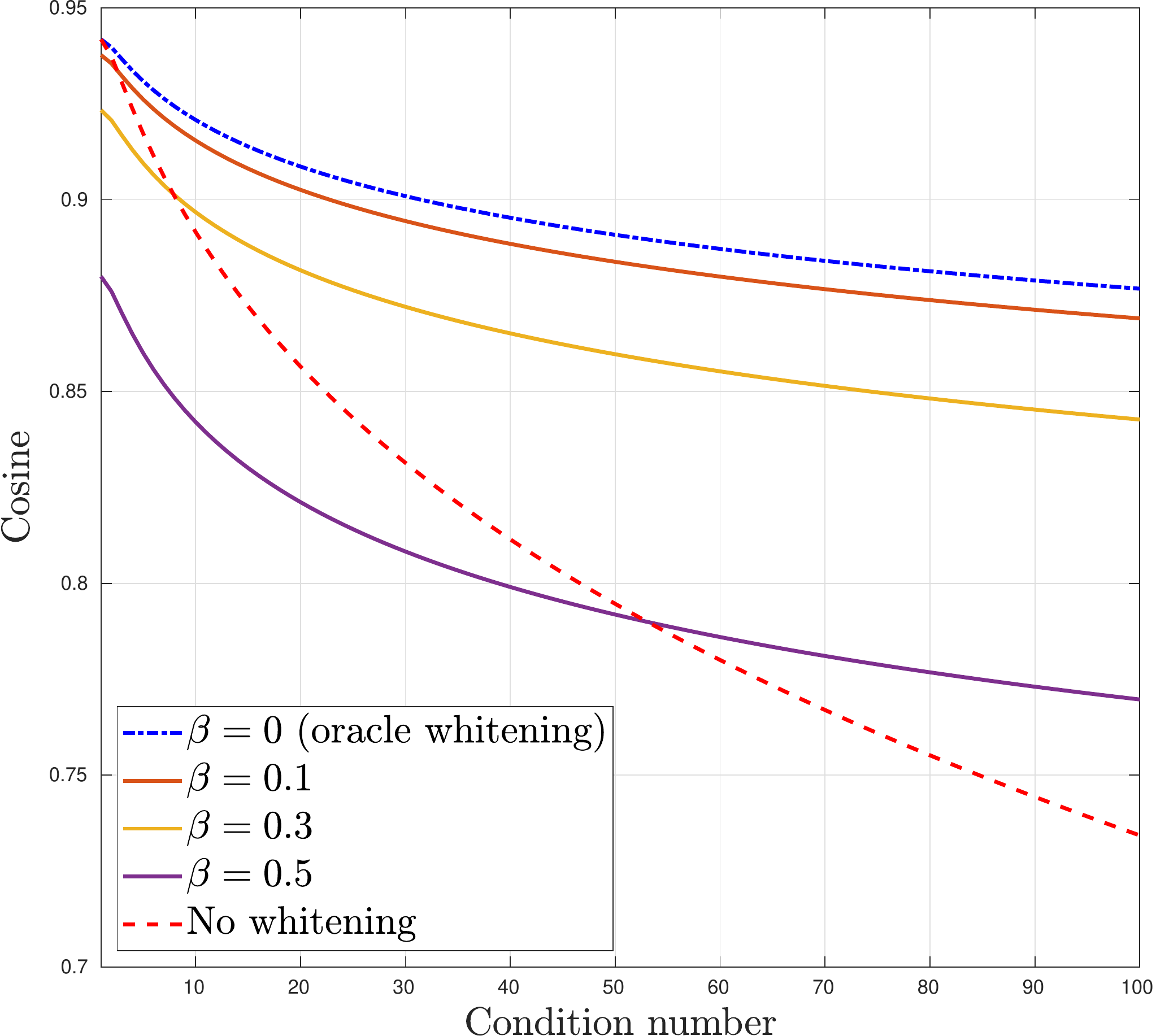}
\caption{The cosines of the angles between the true principal component and the estimated principal component, plotted as functions the condition number $\kappa$ of the noise covariance matrix $\bSigma$ with linearly spaced eigenvalues.}
\label{figure.cosines}
\end{figure}

\subsection{Comparison with other singular value shrinkers}
\label{sec:comparison}

\begin{figure}
\centering
\includegraphics[scale=.35]{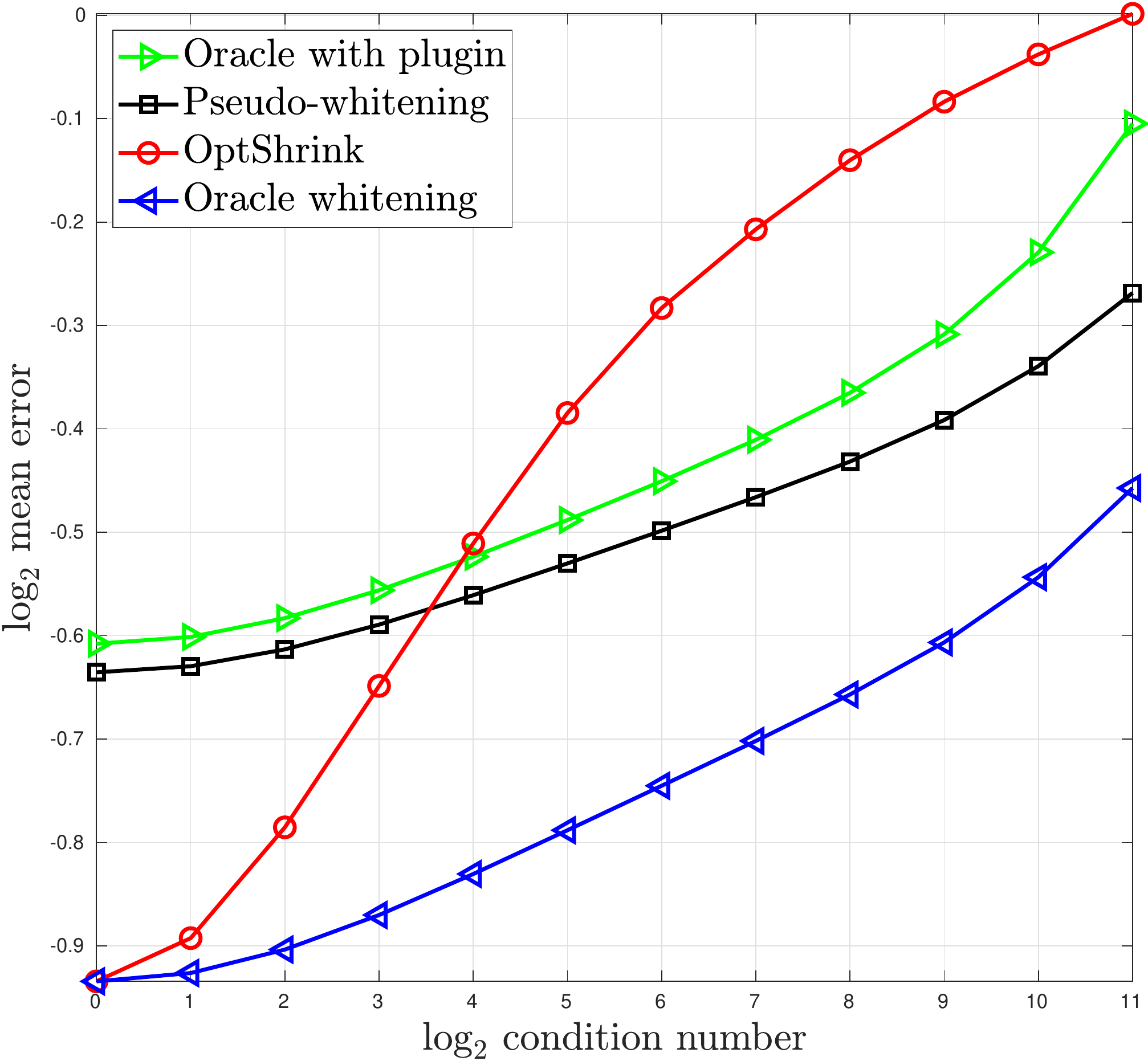}
\caption{
Plots of the mean relative error as a function of the condition number of the noise covariance matrix $\bSigma$ (plotted in log scale). Four shrinkage methods are considered: (i) Oracle with plugin: the whiten-shrink-recolor procedure of \cite{leeb2021matrix}, with a plugin estimate for the noise covariance matrix; (ii) Pseudo-whitening: the method proposed in this paper; (iii) OptShrink: the optimal singular value shrinker of \cite{nadakuditi2014optshrink}; (iv) Oracle whitening: the procedure if \cite{leeb2021matrix} using the true noise covariance matrix -- note that this is an oracle method, and not practically implementable under the setting considered in this paper. The experiment shows that, unsurprisingly, oracle whitening attains the smallest error, and outperforms optimal shrinkage (OptShrink) by an increasing margin as the noise covariance becomes more and more ill-conditioned. Pseudo-whitening using the number of available noise-only measurements ($m=1600$, the data dimensions being $n=1200,p=600$) is worse than OptShrink for small condition number, but attains markedly better performance when it is large. Furthermore pseudo-whitening, which is optimal over all Whiten-Shrink-reColor procedures, outperforms \cite{leeb2021matrix} with a plugin estimate of the noise covariance.
}
\label{figure.compare}
\end{figure}

We compare the shrinker described in Section \ref{sec.estimator} to three other shrinkage methods. The first method is OptShrink \cite{nadakuditi2014optshrink}, the optimal singular value shrinker of the data matrix $\bY$ itself without any transformation. OptShrink uses knowledge of the operator norm of the pure noise matrix $\bSigma^{1/2} \Zin$, which we evaluate numerically using the method from \cite{leeb2021rapid}. We also compare to optimal shrinkage with exact whitening from \cite{leeb2021optimal}. In the present context this is an oracle method, that assumes exact knowledge of the noise covariance $\bSigma$ (equivalently, it may be viewed it as instance of the algorithm from Section \ref{sec.estimator} with $\beta = 0$). Finally, we compare with the shrinker from \cite{leeb2021optimal} with the sample covariance $\what \bSigma$ used as a plug-in estimator; that is, this method uses pseudo-whitening, but treats it as if it were oracle whitening.



The details of the experiment are are follows. We set the problem size parameters $p=600$; $n=1200$; $m= 1800$; and $r=3$. We generate the covariance $\bSigma$ with diagonal entries equally spaced between $1$ and a specified condition number $\kappa \ge 1$. The signal singular values $\sigma_1$, $\sigma_2$, and $\sigma_3$ are defined as follows:
\begin{align}
\sigma_j = \sigmaThresh \cdot \frac{1+(r-j+1)/r}{\sqrt{\tau}}, \quad 1 \le j \le 3,
\end{align}
where $\tau$ is given by
\begin{align}
\tau = \frac{1}{p} \sum_{i=1}^{p} \frac{1}{\bSigma_{ii}}.
\end{align}
These value of $\sigma$ ensures that the signal components are asymptotically detectable, but very close to the edge of detectability. We generate the vectors $\bv_1$, $\bv_2$ and $\bv_3$ in $\RR^n$ and $\bu_1$, $\bu_2$ and $\bu_3$ in $\RR^p$ to be random orthonormal vectors (obtained by applying Gram-Schmidt to random vectors). We then define the $p \times n$ signal matrix $\bX$
\begin{align}
\bX = \sum_{k=1}^{3} \sigma_k \bu_k \bv_k^\T.
\end{align}
We generate the $p \times n$ in-sample noise matrix $\bSigma^{1/2}\Zin$, where $\Zin$ has i.i.d. entries with variance $1$ from a Gaussian distribution. We construct the $p \times n$ data matrix $\bY = \bX + \bSigma^{1/2}\Zin$. Finally, for each value of $\kappa$, we generate the $p \times m$ out-of-sample noise matrix $\bSigma^{1/2}\Zout$, where $\Zout$ has i.i.d. Gaussian entries of variance $1$. Each of the four methods is then applied to this input data, with the oracle rank $r=3$ supplied. For each value of $\kappa$, the data is generated $N = 200$ times, and the relative errors are averaged over all runs. The results of this experiment are shown Figure \ref{figure.compare}, which plots the $\log_2$ mean relative errors of each method against $\log_2(\kappa)$. Optimal shrinkage with oracle whitening outperforms each method, as is expected; the optimal shrinker with pseudo-whitening always outperforms the oracle shrinker with a plug-in covariance, and outperforms OptShrink at high condition numbers.

\subsection{Convergence for Gaussian and non-gaussian noise}
\label{sec:convergence_nongaussian}

\begin{table} 
\centering 
\begin{tabular}{|l| c  | c | c | c| c |}  
\hline  
 &  \multicolumn{5}{c|}{Mean error, inner cosines} \\
\hline  
 $p$ & Gaussian & Rademacher & t, df=10 & t, df=4.5 & t, df=3 \\
\hline  
 550 & 9.857e-03 & 9.920e-03 & 9.979e-03 & 1.991e-02 & 4.841e-01  \\
\hline  
1100 & 6.998e-03 & 7.043e-03 & 7.027e-03 & 1.470e-02 & 6.068e-01  \\
\hline  
2200 & 4.926e-03 & 4.933e-03 & 4.946e-03 & 1.073e-02 & 7.302e-01  \\
\hline  
4400 & 3.508e-03 & 3.491e-03 & 3.440e-03 & 8.147e-03 & 8.476e-01  \\
\hline  
\end{tabular}  
\caption{Mean error of $\hbv^\T \bv$.}
\label{table.inner}

\vspace{2\baselineskip}

\centering 
\begin{tabular}{|l| c  | c | c | c| c |}  
\hline  
 &  \multicolumn{5}{c|}{Mean error, outer cosines, whitened} \\
\hline  
 $p$ & Gaussian & Rademacher & t, df=10 & t, df=4.5 & t, df=3 \\
\hline  
 550 & 1.559e-02 & 1.600e-02 & 1.603e-02 & 2.604e-02 & 4.968e-01  \\
\hline  
1100 & 1.133e-02 & 1.131e-02 & 1.123e-02 & 1.907e-02 & 6.146e-01  \\
\hline  
2200 & 7.974e-03 & 8.020e-03 & 7.896e-03 & 1.366e-02 & 7.346e-01  \\
\hline  
4400 & 5.662e-03 & 5.534e-03 & 5.611e-03 & 1.026e-02 & 8.498e-01  \\
\hline  
\end{tabular}  
\caption{Mean error of $(\hbu^w)^\T \hbSigma^{1/2} \bu$.}
\label{table.outer.whitened}

\vspace{2\baselineskip}

\centering 
\begin{tabular}{|l| c  | c | c | c| c |}  
\hline  
 &  \multicolumn{5}{c|}{Mean error, outer cosines, unwhitened} \\
\hline  
 $p$ & Gaussian & Rademacher & t, df=10 & t, df=4.5 & t, df=3 \\
\hline  
 550 & 2.121e-02 & 2.130e-02 & 2.123e-02 & 3.085e-02 & 4.794e-01  \\
\hline  
1100 & 1.498e-02 & 1.500e-02 & 1.496e-02 & 2.259e-02 & 6.022e-01  \\
\hline  
2200 & 1.058e-02 & 1.071e-02 & 1.060e-02 & 1.614e-02 & 7.266e-01  \\
\hline  
4400 & 7.529e-03 & 7.417e-03 & 7.477e-03 & 1.211e-02 & 8.452e-01  \\
\hline  
\end{tabular}  
\caption{Mean error of $(\hbu^w)^\T \hbSigma^{-1/2} \bu$.}  
\label{table.outer.unwhitened}

\vspace{2\baselineskip}

\centering 
\begin{tabular}{|l| c  | c | c | c| c |}  
\hline  
 &  \multicolumn{5}{c|}{Mean error, singular values} \\
\hline  
 $p$ & Gaussian & Rademacher & t, df=10 & t, df=4.5 & t, df=3 \\
\hline  
 550 & 1.653e-02 & 1.673e-02 & 1.677e-02 & 1.944e-02 & 2.857e-01  \\
\hline  
1100 & 1.169e-02 & 1.188e-02 & 1.164e-02 & 1.352e-02 & 3.609e-01  \\
\hline  
2200 & 8.267e-03 & 8.343e-03 & 8.350e-03 & 9.464e-03 & 4.464e-01  \\
\hline  
4400 & 5.918e-03 & 5.907e-03 & 5.928e-03 & 6.811e-03 & 6.014e-01  \\
\hline  
\end{tabular}  
\caption{Mean error of $\yval$.}  
\label{table.singular.values}
\end{table}

We check the convergence rates of the observed cosines and singular values to their limiting values. The simulation was run as follows. We set the parameters $\gamma = 2/3$ and $\beta = 1/4$. For each fixed value of $p$, we take $n = p / \gamma$ and $m = p / \beta$. For each $p$, we generate the noise covariance $\bSigma$ with diagonal entries linearly spaced between $1$ and $50$. We generate the matrix $\bD = \bD_1$ as
\begin{align}
\bD = \sqrt{p} \cdot \diag(1/p^2,4/p^2,\dots,(p-1)^2/p^2,1).
\end{align}
We then generate the vector $\bu = \bu_1$ as
\begin{align}
\bu = \frac{\bD \bw}{\|\bD \bw\|},
\end{align}
where $\bw$ has entries that are i.i.d. Gaussian. We also set the value $\sigma = \sigma_1$ to be:
\begin{align}
\sigma = \sigmaThresh \cdot \frac{1.8}{\sqrt{\tau}},
\end{align}
where
\begin{align}
\tau = \frac{1}{p} \sum_{i=1}^{p} \frac{\bD_{1i}^2}{\bSigma_{ii}}.
\end{align}

This value of $\sigma$ ensures that the signal is detectable. We also generate the vector $\bv = \bv_1$ with i.i.d. Gaussian entries. With these parameters set, we define the $p \times n$ signal matrix $\bX = \sigma \bu \bv^\T$. We generate the $p \times n$ in-sample noise matrix $\bSigma^{1/2}\Zin$, where $\Zin$ has i.i.d. entries with variance $1$ from a specified distribution, either Gaussian, Rademacher, normalized $t_{10}$, normalized $t_{4.5}$, or normalized $t_3$ (where the $t$ distributions are normalized to have unit variance). We construct the $p \times n$ data matrix $\bY = \bX + \bSigma^{1/2}\Zin$. Finally, we generate the $p \times m$ out-of-sample noise matrix $\bSigma^{1/2}\Zout$, where $\Zout$  has i.i.d. entries from the same distribution as the entries of $\bSigma^{1/2}\Zin$.

With the data generated in this way, we compute the values $\hbv^\T \bv$, $(\hbu^w)^\T \hbSigma^{1/2} \bu$, $(\hbu^w)^\T \hbSigma^{-1/2} \bu$, and $\yval$. For each value of $p$, the experiment is run $N = 10,000$ times, and the errors are averaged as follows:
\begin{align}
\text{Error}_p(\xi) = \frac{1}{N} \sum_{i=1}^{N} \frac{|\widehat{\xi}_i - \xi|}{|\xi|},
\end{align}
where $\xi$ denotes the asymptotic value of the parameter in question and $\widehat\xi_1,\dots,\widehat\xi_N$ denote the $N$ realizations of the parameter. These average errors are presented in Tables \ref{table.inner} through \ref{table.singular.values}. Figure \ref{figure.error_rate} shows the errors are plotted in log scale.

\begin{figure}
\centering
\begin{subfigure}{0.4\textwidth}
\includegraphics[scale=.3]{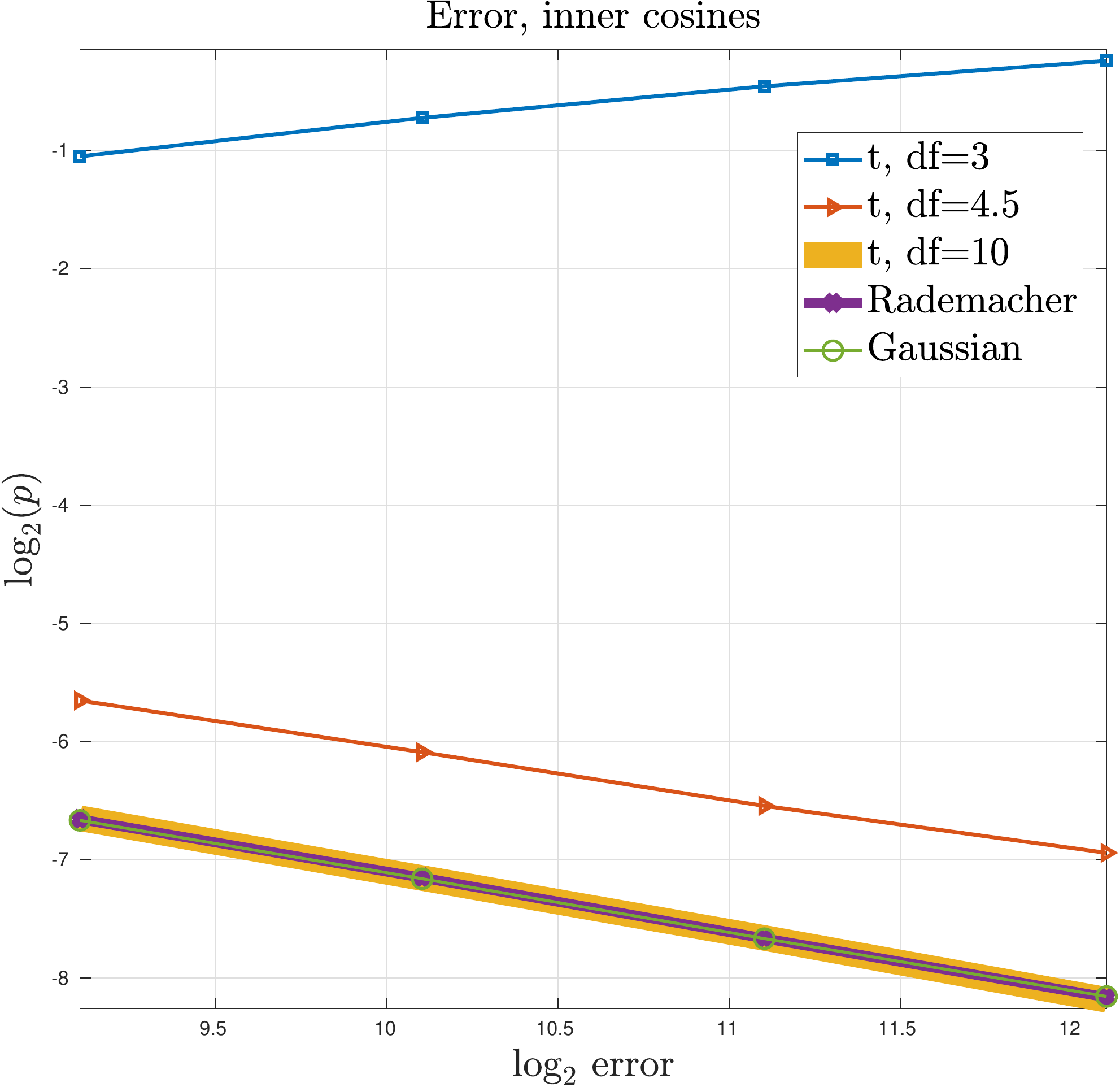}
\end{subfigure}
\hspace{.5in}
\begin{subfigure}{0.4\textwidth}
\includegraphics[scale=.3]{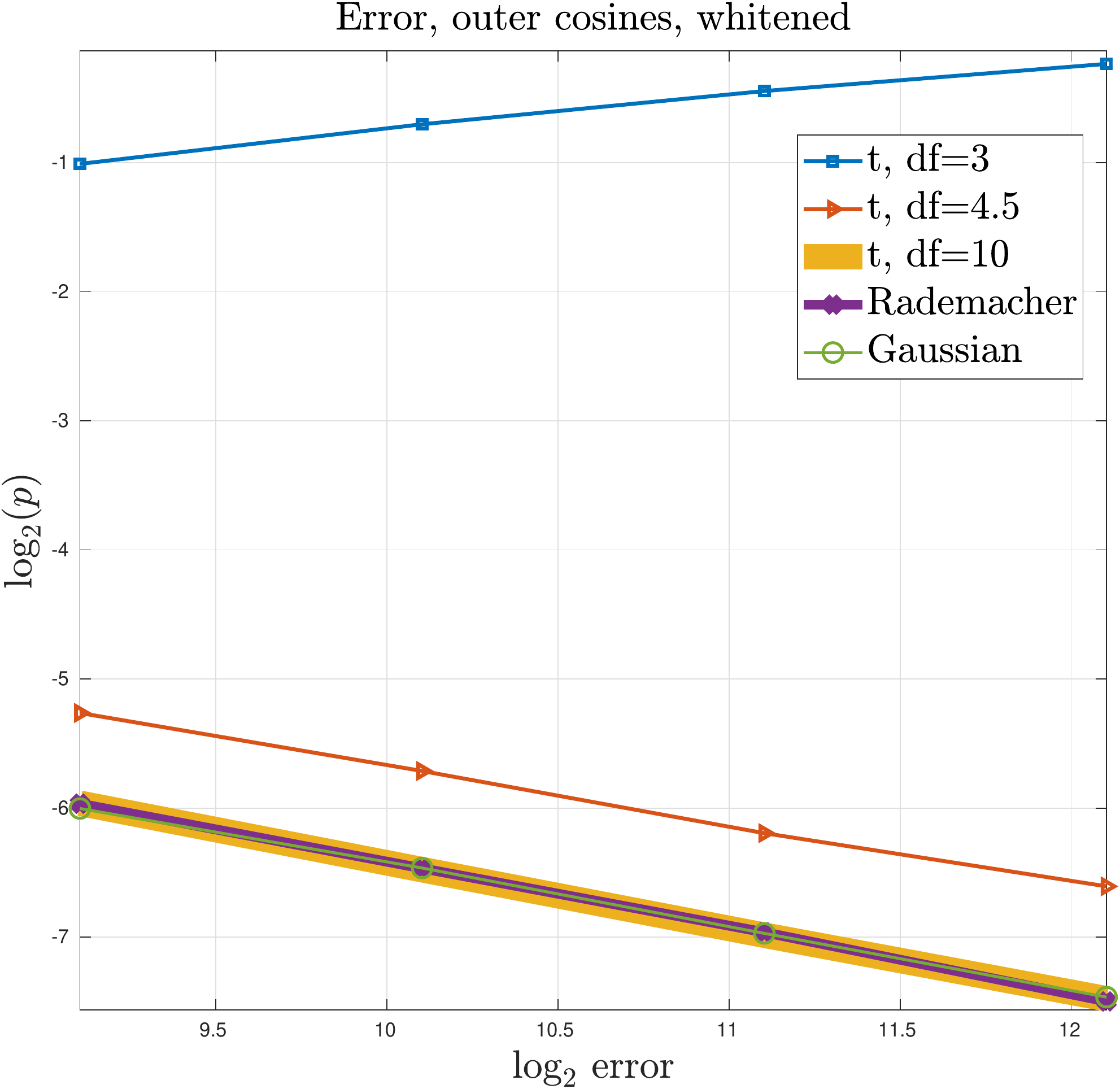}
\end{subfigure}
\hfill
\begin{subfigure}{0.4\textwidth}
\includegraphics[scale=.3]{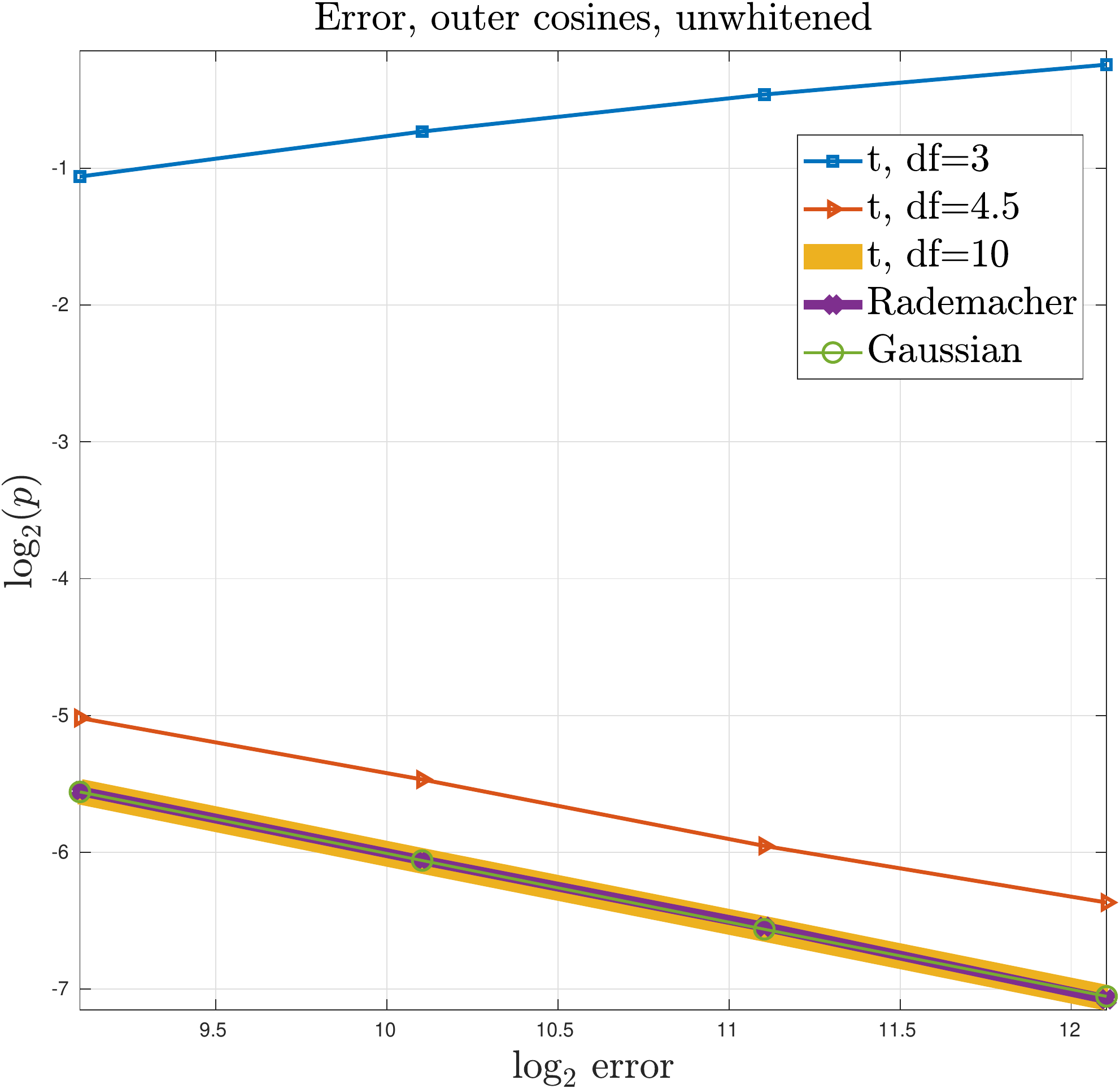}
\end{subfigure}
\hspace{.5in}
\begin{subfigure}{0.4\textwidth}
\includegraphics[scale=.3]{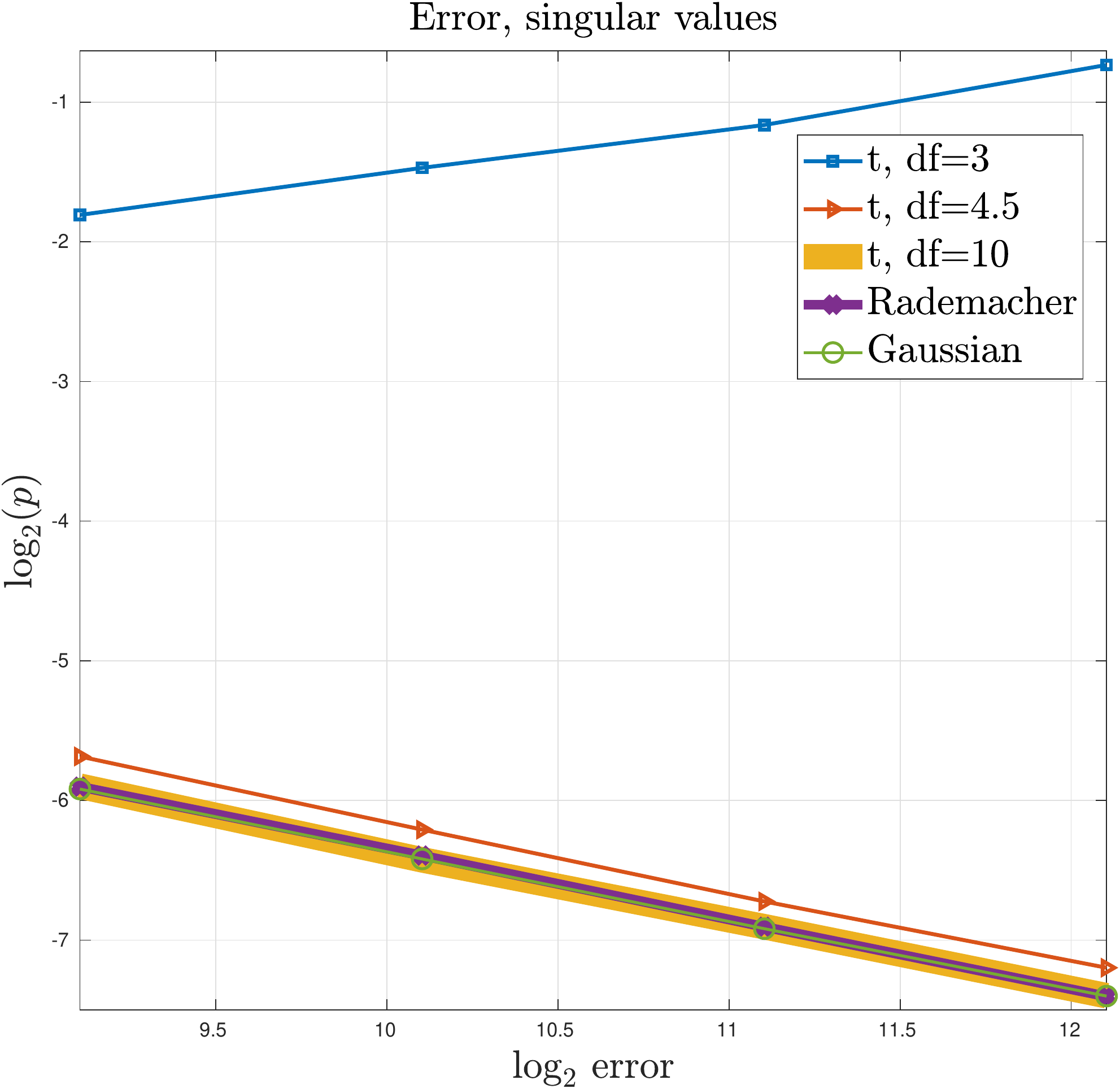}
\end{subfigure}
\caption{The $\log_2$ of the average relative errors between the spiked model parameters and their asymptotic values, plotted against $\log_2(p)$, for five different noise types. The errors for the Gaussian, Rademacher, and $t_{10}$ noise are nearly indistinguishable, and appears to decay at a rate of approximately $O(p^{-1/2})$, since their curves have slopes close to $1/2$; the errors for $t_{4.5}$ noise are larger, but still appear to decay at approximately the same rate. By contrast, the errors for the heavy-tailed $t_3$ noise diverge.}
\label{figure.error_rate}
\end{figure}

For noise with Gaussian, Rademacher, and $t_{10}$ distributions the errors between the observed values and the true values decays at approximately the rate $O(p^{-1/2})$; furthermore, the errors themselves are nearly identical regardless of the noise distribution. By contrast, the errors for the $t_{4.5}$ distribution are larger, though they still shrink at a similar rate; whereas the errors for the fat-tailed $t_3$ distribution grow with $p$.

\section{Discussion and conclusion}
\label{sec:conclusion}

This paper considered the problem of low-rank matrix denoising under additive noise, assuming that the columns of the noise matrix are i.i.d. but have an otherwise arbitrary and unknown inter-row covariance structure. Under our setup, one has access to side information in the form of pure-noise samples, and wishes to make use of this additional information in denoising. Crucially, the number of available pure-noise samples is such that one {cannot} consistently estimate the noise covariance matrix $\bSigma$; thus, one can think of $\bSigma$ as being only partially known.

We proposed to tackle this problem by means of singular value shrinkage, under a Whiten-Shrink- re-Color framework. To wit, one forms the pure-noise sample covariance $\hbSigma$, and 1) uses $\hbSigma$ to  pseudo-whiten the observed signal-plus-noise matrix; 2) applies optimal singular value shrinkage to the resulting pseudo-whitened matrix; and finally 3) performs a re-coloring step. Our main contribution is the derivation of the optimal singular shrinker to be used in this compound procedure. To this end, we proved new results on the spectrum of the spiked F-matrix ensemble, which may be of independent interest. 

As one would expect, we demonstrated that our optimally-tuned WSC denoiser outperforms OptShrink, the optimally-tuned singular value shrinkage without noise whitening \cite{nadakuditi2014optshrink}), provided that the number of pure-noise samples is sufficiently large (see Section~\ref{sec:numerical}).
A particularly appealing quality of the proposed estimator is its simplicity, both conceptually and implementation-wise. Indeed, the idea of noise whitening is classical within signal processing. Furthermore, the method comes with precise performance guarantees, including estimates of the AMSE and the inner products between the signal PCs and their estimates, which can be useful in practice. There is no reason to believe, however, that this approach is optimal among all denoising procedures. Devising new methods to make better use of the available side information (and/or deriving tight lower bounds on the attainable MSE) is an interesting direction for future research.

A natural extension of our method would be to replace the noise sample covariance $\hbSigma$ with a better covariance estimator: note that while the sample covariance has generically an optimal  \emph{estimation rate}, one can often construct estimators that attain a smaller estimation error (e.g., smaller constant prefactors); see for example the shrinkage estimator of \cite{ledoit2012nonlinear}. To implement such modified WSC scheme, one needs to calculate precise analytic formulas for the singular values and singular vector angles of the whitened and recolored data matrix. For more sophisticated covariance estimators (such as \cite{ledoit2012nonlinear}), this calculation appears to be challenging.

    \section*{Acknowledgements}

    ER was partially supported by a Hebrew University of Jerusalem Einstein-Kaye scholarship.
MG was partially supported by ISF grant no. 871/22.
WL acknowledges support from NSF BIGDATA award IIS-1837992, BSF award 2018230, and NSF CAREER award DMS-2238821.

ER wishes to thank David Donoho, Apratim Dey and Michael Feldman for interesting discussions regarding this manuscript. 

    \bibliographystyle{plain}
    \bibliography{ref}

    \appendix

    \section{Auxiliary Technical Results}

\subsection{Elementary concentration lemmas}

The following elementary concentration results are used throughout the paper:

\begin{lemma}
    Let $\bA_p\in \RR^{p\times p}$ be a sequence of bounded matrices, $\ba\in \RR^p$ a bounded vector,  and $\bg\sim \mathcal{N}(\bm{0},p^{-1}\Id_{p\times p})$. Then
    \[
    \|\bg\|^2\asympeq 1,\quad \bg^\T \bA_p \bg \asympeq p^{-1}\tr(\bA_p),\quad \bg^\T \ba \asympeq 0 \,.  
    \]
\end{lemma}

\begin{lemma}
    \label{lem:quadratic-square}
    Let $\bA_{p} \in \RR^{p\times p}$ be a sequence of bounded, orthogonally invariant, random matrices, namely, for any $\bO \in O(p)$, $\bO\bA_p\bO^\T \overset{d}{=}\bA_p$. Let $\ba_p,\bb_p\in\RR^n$ be two sequences of bounded vectors.
    Then
\begin{align*}
    \ba_p^\T \bA_p \bb_p \asympeq \left(\ba_p^\T \bb_p\right) \cdot p^{-1}\trace(\bA_p) \,.
\end{align*}
\end{lemma}

\begin{lemma}
    \label{lem:quadratic-rectangle}
    Let $A_p \in \RR^{p\times n}$ be a sequence of left orthogonally invariant random matrices, namely, for any $\bO\in O(p)$, $\bO\bA_p \overset{d}{=}\bA_p$. Let $\ba_p\in \RR^{p},\bb_p\in\RR^n$ be bounded. Then $\ba_p^\T \bA \bb_p \asympeq 0 $. 
\end{lemma}

\subsection{Free Probability and calculation of mixed moments}
\label{sec:free-probability}


Free probability is a theory of noncommutative random variables, originally introduced by Voiculescu \cite{voiculescu1992free}. We briefly describe several elementary results that are used in our calculations. For more background, see e.g. \cite{mingo2017free,couillet2011random}.

\begin{definition}[Limiting joint law]
    Let $\mathbb{X} = \left\{\bX_{1},\ldots,\bX_{l}\right\}$ be $l$ (sequences of) of $p$-by-$p$  matrices with bounded operator norm. We say $\mathbb{X}$ has an a.s. limiting joint law if for any non-commutative polynomial $\mathcal{P} \in \CC\left\langle x_1,\ldots,x_l,y_1,\ldots,y_l\right\rangle$,\footnote{By that, we mean a polynomial in $2l$ non-commutative variables. For example, $x_1x_2 \ne x_2x_1$. } the a.s. limit 
\begin{align}
    \trlim \,\mathcal{P}\left(\bX_{1},\ldots,\bX_{l},\bX^\T_{1},\ldots,\bX^\T_{l}\right) = 
    \lim_{p\to \infty} \frac1p \trace \mathcal{P}\left(\bX_{1},\ldots,\bX_{l},\bX^\T_{1},\ldots,\bX^\T_{l}\right) 
\end{align}
    exists. The mapping $ \CC\left\langle x_1,\ldots,x_l,y_1,\ldots,y_l\right\rangle \to \RR$ taking a polynomial $\mathcal{P}$ to its corresponding limit is called the joint law of $\m{S}$. 
\end{definition}

\begin{remark}
    Note that the joint law of $\m{S}_p$ is completely determined by its (non-commutative) joint moments:
\begin{align*}
    z_{i_1}\cdots z_{i_t} \,, z\in\left\{x,y\right\},\,\,i_j\in\left\{1,\ldots,l\right\} \,.
\end{align*}
    Moreover, if $\mathbb{X}=\left\{\bX\right\}$ is a single symmetric matrix with an LESD $dF$, then it limiting law is
\begin{align*}
    x^k \mapsto \int x^k dF(x) \,.
\end{align*}
\end{remark}

\begin{definition}[Free independence]
    Suppose that $\mathbb{X}_{1},\ldots,\mathbb{X}_{k}$ each have a joint limiting law. They are (asymptotically) freely independent if the following holds: let $\mathcal{P}_{j}(\mathbb{X}_{i_j})$ be any non-commutative polynomials in $\mathbb{X}_{i_j}$, then
\begin{align}
        \label{eq:asymp-freeness-def}
        \trlim\, \left\{ \left[\mathcal{P}_1(\mathbb{X}_{i_1})-\trlim(\mathcal{P}_1(\mathbb{X}_{i_1}))\Id\right] \cdots \left[\mathcal{P}_t(\mathbb{X}_{i_t})-\trlim(\mathcal{P}_t(\mathbb{X}_{i_t}))\Id\right]\right\} = 0  \,,
\end{align}
    whenever adjacent indices are always different:
    $i_1\ne i_2$, $i_2\ne i_3$ etc. 
\end{definition}


Free independence is a powerful tool for computing traces involving multiple random matrices. A fundamental result connecting free probability with random matrices states that independent unitarily/orthogonally invariant random matrices are asymptotically freely independent. To our knowledge, the complex (unitary) first appeared in the work of Voiculescu \cite{voiculescu1992free}; for the real case, we cite a result of Collins and \'Sniady \cite{collins2006integration} (see also \cite{collins2003moments}). 
The following is essentially \cite[Theorem 5.2]{collins2006integration}:

\begin{theorem}[Asymptotic freedom for orthogonally invariant matrices]
    \label{thm:freeness}
    Let $(\bX_{1}),\ldots,(\bX_{k})$ be $k$ (sequences of) $p$-by-$p$ matrices, so that  each matrix has an a.s. limiting law. Let $\bO_2,\ldots,\bO_k\sim \mathrm{Haar}(O(p))$. 
    Then
\begin{align*}
    (\bX_{1}), (\bO_2\bX_{2}\bO_2^\T), \ldots, (\bO_{k}\bX_{k}\bO_k^\T)
\end{align*}
    are asymptotically freely independent.
\end{theorem}

\section{Proof of Lemma~\ref{lem:G2}}
\label{sec:proof-lem:G2}
The proof consists of repeated applications of (\ref{eq:asymp-freeness-def}). 
It is straightforward to verify that 
\begin{equation}
    G_{2,k}(z)=\trlim\left( \bA_z\bSigma \bA_z^\T \bC_k \right),\quad\textrm{where}
    \quad
    \bA_z:=(z\bS-\bE)^{-1}\bS,\;\bC_k=\bSigma^{-1/2}\bD_k\bD_k^\T \bSigma^{-1/2} \,.
\end{equation}
Since $\bA_z$ is orthogonal invariant and has bounded norm (since $z\notin [\sigmaMin,\sigmaBulk])$, it is freely independent of all deterministic matrices. 
To carry out the computation,  we apply \eqref{eq:asymp-freeness-def} successively.
For brevity, denote $\bm{\Delta}_z=\bA_z-\trlim(\bA_z)\bI$. Then
\begin{align*}
    G_{2,k}(z)
    &=\trlim\left(\bm{\Delta}_z\bSigma \bA_z^\T \bC_k  \right) + \trlim(\bA_z)\trlim\left(\bSigma\bA_z^\T \bC_k\right)\\
    &= \trlim\left(\bm{\Delta}_z\bSigma \bA_z^\T \bC_k  \right) + \left(\trlim(\bA_z)\right)^2
    \trlim(\bC_k\bSigma) 
    = 
    \trlim\left(\bm{\Delta}_z\bSigma \bA_z^\T \bC_k  \right) + (\stiel(z))^2\,,
\end{align*}
where the last equality uses the assumption $\trlim(\bD_k\bD_k^\T)$.
We simplify further the first term on the r.h.s.,
\begin{align*}
    \trlim\left(\bm{\Delta}_z\bSigma \bA_z^\T \bC_k  \right) = \trlim\left(\bm{\Delta}_z\bSigma \bm{\Delta}_z^\T \bC_k  \right) + \trlim(\bA_z^\T) \underbrace{\trlim\left(\bm{\Delta}_z\bSigma \bC_k  \right)}_{=0} =
    \trlim\left(\bm{\Delta}_z\bSigma \bm{\Delta}_z^\T \bC_k  \right)\,.
\end{align*}
Next,
\begin{align*}
    \trlim\left(\bm{\Delta}_z\bSigma \bm{\Delta}_z^\T \bC_k  \right)
    &=
    \trlim\left(\bm{\Delta}_z(\bSigma-\mu\Id) \bm{\Delta}_z^\T \bC_k  \right) + \mu\cdot \trlim\left(\bm{\Delta}_z \bm{\Delta}_z^\T \bC_k  \right) \\
    &= 
    \underbrace{\trlim\left(\bm{\Delta}_z(\bSigma-\mu\Id) \bm{\Delta}_z^\T (\bC_k-\tau_k\Id)  \right)}_{=0}
    +
    \tau_k \underbrace{\trlim\left(\bm{\Delta}_z(\bSigma-\mu\Id) \bm{\Delta}_z^\T   \right)}_{=0}
    + \mu\cdot \trlim\left(\bm{\Delta}_z \bm{\Delta}_z^\T \bC_k  \right) \\
    &= \mu\tau_k\cdot \trlim(\bm{\Delta}_z\bm{\Delta}_z^\T) \,.
\end{align*}
Lastly, $\trlim(\bm{\Delta}_z\bm{\Delta}_z^\T) = \trlim(\bA_z\bA_z^\T)-(\trlim(\bA_z))^2=\UpsTwo(z)-(\stiel(z))^2$. Collecting all the terms above, we obtain the claimed formula for $G_{2,k}(z)$. 
\hfill $\blacksquare$

\section{Closed-form formulas for $\UpsOne,\UpsTwo$}
\label{sec:New:Ups}
In this section we prove the closed-form formulas (\ref{eq:New:UpsOneFormula}) and (\ref{eq:New:UpsTwoFormula}) for the mixed traces
\begin{align*}
    \UpsOne(z) = \lim_{p\to\infty} p^{-1} \tr (z\bS-\bE)^{-1}\bS^2,\quad
    \UpsTwo(z) = \lim_{p\to\infty} p^{-1}\tr (z\bS-\bE)^{-2}\bS^2 \,,
\end{align*}
where $z\in \CC\setminus [\sigmaMin^2,\sigmaBulk^2]$.

First, because the formula (\ref{eq:MPStiel-def}) for the Stieltjes transform is only applicable for arguments smaller than the left edge of the Marchenko-Pastur law, $0<z<(1-\sqrt{\beta})^2$, we begin {with:}
%
\begin{lemma}
\label{lem:range_sbar}
    For any $z\in (\sigmaBulk^2,\infty)$, one has $0<-\sbar(z)<(1-\sqrt{\beta})^2$.
\end{lemma}
\begin{proof}
    Since $-\sbar(z)$ is positive and decreasing for $z>\sigmaBulk^2$, it suffices to show that $-\sbar(\sigmaBulk^2)\le (1-\sqrt{\beta})^2$. A straightforward calculation gives $-\sbar(\sigmaBulk^2)=f(\gamma;\beta)=f_1(\gamma;\beta)+f_2(\gamma;\beta)$ where
\begin{align}
    f_1(\gamma;\beta) = \frac{\beta(\beta -\sqrt{\beta+\gamma-\beta\gamma})}{\beta+\gamma},
\end{align}
and
\begin{align}
f_2(\gamma;\beta) = \frac{1-\sqrt{\beta+\gamma-\beta\gamma}}{1-\gamma} \,;
\end{align}
note that $f_2(\gamma;\beta)$ is well-defined and {differentiable} at $\gamma=1$, where its value is {$f_2(1;\beta)=\frac12(1-\beta)$}.

We claim that the function $\gamma \mapsto f(\gamma;\beta)$, $\gamma\in[0,\infty)$, is maximized at $\gamma=0$. This, in turn, implies 
\begin{math}
-\sbar(\sigmaBulk^2)\le f(0;\beta)=\beta-\sqrt{\beta}+1-\sqrt{\beta}=(1-\sqrt{\beta})^2
\end{math}
as required. To show this, let us compute the derivative of $f(\cdot;\beta)$. One has
\begin{align}
\label{eq:f1_partial}
    \partial_{\gamma} f_1(\gamma;\beta)
= \beta\frac{- \frac{1-\beta}{2\sqrt{\beta+\gamma-\beta\gamma}}(\beta+\gamma) 
    - (\beta-\sqrt{\beta+\gamma-\beta\gamma})}{(\beta+\gamma)^2}
= \frac{\beta \left(  \beta - \sqrt{\beta+\gamma-\beta\gamma} \right)^2}{2(\beta+\gamma)^2
    \sqrt{\beta+\gamma-\beta\gamma}}\,,
\end{align}
and
\begin{align}
\label{eq:f2_partial}
    \partial_{\gamma} f_2(\gamma;\beta) = 
\frac{-\frac{1-\beta}{2\sqrt{\beta+\gamma-\beta\gamma}}(1-\gamma)    
    +(1-\sqrt{\beta+\gamma-\beta\gamma})}{(1-\gamma)^2} 
= - \frac{\left( 1 - \sqrt{\beta+\gamma-\beta\gamma} \right)^2} {2(1-\gamma)^2 \sqrt{\beta+\gamma-\beta\gamma}}\,,
\end{align}
(which may also be {continuously} extended to $\gamma=1$). Since $f(\gamma; \beta) \to 0$ as $\gamma \to \infty$, it is enough to show that the only solution to $\partial_{\gamma} f(\gamma; \beta) = 0$ for $\gamma \in [0, \infty]$ is $\gamma=0$. {First, let us consider $\gamma=1$. One may readily calculate (e.g., using L'H\^{o}pital's rule): $\partial_{\gamma} f_2(1;\beta)=-\frac18(1-\beta)^2$ and $\partial_{\gamma} f_1(1;\beta)=\frac{\beta(1-\beta)^2}{2(1+\beta)^2}$, so $\partial_{\gamma} f(1;\beta)=-\frac{(1-\beta)^4}{8(1+\beta)^2}<0$.} {We next consider values $\gamma\ne 1$. From \eqref{eq:f1_partial} and \eqref{eq:f2_partial}, if $\partial_{\gamma} f(\gamma;\beta)=0$ then either}
\begin{align}
\label{eq:1455}
(1-\gamma)\sqrt{\beta}\left(  \beta - \sqrt{\beta+\gamma-\beta\gamma} \right)
= (\beta+\gamma)\left( 1 - \sqrt{\beta+\gamma-\beta\gamma} \right),
\end{align}
\begin{align}
\label{eq:1456}
(1-\gamma)\sqrt{\beta}\left(  \beta - \sqrt{\beta+\gamma-\beta\gamma} \right)
=- (\beta+\gamma)\left( 1 - \sqrt{\beta+\gamma-\beta\gamma} \right).
\end{align}
%
{Making a change of variables $r = \beta+\gamma-\beta\gamma$, so that $\gamma=\frac{r-\beta}{1-\beta}$,} equations \eqref{eq:1455} and \eqref{eq:1456} may be rewritten in terms of $r$ as follows:
\begin{align}
\label{eq:1457}
\sqrt{\beta}\left(\beta - \sqrt{r} \right) \left( \frac{1-r}{1-\beta}\right)
\pm \left( 1 - \sqrt{r} \right) \left( \frac{r-\beta^2}{1-\beta}\right) = 0\,.
\end{align}
Using $r-\beta^2 = (\sqrt{r} - \beta)(\sqrt{r} + \beta)$ and $1-r = (1- \sqrt{r})(1+\sqrt{r})$, \eqref{eq:1457} is equivalent to:
\begin{align}
\label{eq:1458}
(\beta - \sqrt{r})(1-\sqrt{r}) \left[ \sqrt{\beta}(1+\sqrt{r}) \pm (\sqrt{r}+\beta) \right] = 0.
\end{align}
The roots are $r = \beta^2$, $r = 1$, and $r = \beta$. {Since $\gamma=(r-\beta)/(1-\beta)$, and we only consider $\gamma\in [0,\infty)\setminus\{1\}$ (recall that $\gamma=1$ was treated separately) it follows that the only solution to $\partial_\gamma f(\gamma;\beta) = 0$ with  $\gamma \in [0,\infty)$ is $\gamma= 0$, concluding the proof.}
\end{proof}

Let $\bS=\bW\bLambda\bW^\T$ be an eigen-decomposition, $\bLambda=\diag(\lambda_1,\ldots,\lambda_p)$.
Since $\bE$ is orthogonally invariant and independent of $\bS$, $\bW^\T \bE \bW \overset{d}{=}\bE$ and is independent of $\bLambda$. Consider the resolvent
\begin{equation}
    \bR_z(w) = \left( w\bI + z\bLambda - \bE  \right)^{-1} \,.
\end{equation} 
Clearly, 
\begin{align*}
    \UpsOne(z)&=\lim_{p\to\infty}p^{-1}\tr \left[ \bR_z(0)\bLambda^2 \right] = \lim_{p\to\infty}p^{-1} \sum_{i=1}^p [\bR_z(0)]_{ii}\lambda_i^2,\\ \UpsTwo(z) &= -\lim_{p\to\infty}p^{-1}\tr \left[ \frac{\partial}{\partial w}{\bR_z(0)}\bLambda^2 \right] = -\lim_{p\to\infty}p^{-1} \sum_{i=1}^p \frac{\partial}{\partial w}[\bR_z(0)]_{ii}\lambda_i^2 \,,
\end{align*}
Recall: since $z>\sigmaBulk^2$, assuming small enough $w$, $\bR_z(w)$ is a.s. well-defined and has bounded operator norm.
The main ingredient of the proof consists of deriving formulas for the \emph{individual} diagonal entries of $\bR_z(w)$. 
We remark that the calculation below uses rather standard ideas from random matrix theory, see for example the book \cite{bai2009book}.

For brevity, denote $\bA_{z,w}=w\bI+z\bLambda-\bE$ and so $\bR_{z}(w)=\bA_{z,w}^{-1}$. 
Let $i\in [n]$ be any coordinate. Denote by $\bP_i \in \RR^{(p-1)\times p}$ the projection onto the coordinate set $[p]\setminus\{i\}$. Let $\be_1,\ldots,\be_p\in\RR^p$ be the standard basis vectors. Up to a permutation of the coordinates, $\bA_{z,w}$ has the block form:
\begin{align*}
    \bA_{z,w} = \begin{bmatrix}
    [\bA_{z,w}]_{i,i} & \be_i^\T \bA_{z,w} \bP_i^\T \\
    \bP_i \bA_{z,w}\be_i & \bP_i \bA_{z,w}\bP_i^\T
    \end{bmatrix} = \begin{bmatrix}
    w+z\lambda_i-n^{-1}[\bZ\bZ^\T ]_{ii} & -n^{-1}\be_i^\T \bZ \bZ^\T  \bP_i^\T \\
    -n^{-1} \bP_i \bZ\bZ^\T \be_i & \bP_i \bA_{z,w}\bP_i^\T \,.
    \end{bmatrix}
\end{align*}
Applying the block matrix inversion formula, 
\begin{align*}
    [\bA_{z,w}^{-1}]_{ii} = \left[ w+z\lambda_i-n^{-1}[\bZ\bZ^\T ]_{ii}
    -( n^{-1}\be_i^\T \bZ \bZ^\T  \bP_i^\T )(\bP_i \bA_{z,w}\bP_i^\T)^{-1} (n^{-1} \bP_i \bZ\bZ^\T \be_i) \right]^{-1} \,, 
\end{align*}
equivalently,
\begin{equation}\label{eq:New:1}
    \frac{1}{[\bR_{z}(w)]_{ii}} = w+z\lambda_i-n^{-1}[\bZ\bZ^\T ]_{ii} 
    -( n^{-1}\be_i^\T \bZ \bZ^\T  \bP_i^\T )(\bP_i \bA_{z,w}\bP_i^\T)^{-1} (n^{-1} \bP_i \bZ\bZ^\T \be_i) \,.
\end{equation}
We will show that the r.h.s. of (\ref{eq:New:1}) concentrates around a deterministic quantity.

We start with the following.
\begin{lemma}\label{lem:New:1}
    A.s., 
    \begin{align*}
        &\max_{1\le i \le n} \left| (n^{-1}\be_i^\T \bZ \bZ^\T  \bP_i^\T )(\bP_i \bA_{z,w}\bP_i^\T)^{-1} (n^{-1} \bP_i \bZ\bZ^\T \be_i) - n^{-1} \tr\left[   \bP_i^\T (\bP_i \bA_{z,w}\bP_i^\T)^{-1} \bP_i \bE \right] \right| \longrightarrow 0 \,,\\
        &\max_{1\le i \le n} \left| (n^{-1}\be_i^\T \bZ \bZ^\T  \bP_i^\T ) \frac{\partial}{\partial w}(\bP_i \bA_{z,w}\bP_i^\T)^{-1} (n^{-1} \bP_i \bZ\bZ^\T \be_i) - n^{-1} \tr\left[   \bP_i^\T  \frac{\partial}{\partial w}(\bP_i \bA_{z,w}\bP_i^\T)^{-1} \bP_i \bE \right] \right| \longrightarrow 0\,.
    \end{align*}
\end{lemma}
\begin{proof}
    Observe that $\bZ^\T \be_i\sim \m{N}(\0,\bI_{n\times n})$. Moreover, this random vector is independent of $\bP_i \bZ$, hence also of $\bA_{z,w}$. Furthermore, the random matrices $n^{-1}\bZ^\T  \bP_i^\T (\bP_i \bA_{z,w}\bP_i^\T)^{-1} \bP_i \bZ$, as well as their $w$-derivatives, have operator norm bounded (a.s.) by a constant\footnote{Note that by eigenvalue interlacing, $\lambda_{\min}(\bP_i\bA_{z,w}\bP_i^\T) \ge \lambda_{\min}(\bA_{z,w})$, hence $\|(\bP_i\bA_{z,w}\bP_i^\T)^{-1}\| \le \|\bA_{z,w}^{-1}\|$.}. The result follows by the Hanson-Wright inequality (cf. \cite[Theorem 6.2.1]{vershynin2018high}) and a union bound over $1\le i\le n$.  
\end{proof}

We next consider the trace in Lemma~\ref{lem:New:1}. 
\begin{lemma}
    \label{lem:New:2}
    A.s.,
    \begin{align*}
        &\max_{1\le i \le n} \left| n^{-1} \tr\left[   \bP_i^\T (\bP_i\bA_{z,w}\bP_i^\T)^{-1} \bP_i \bE \right] - n^{-1} \tr\left[ \bA_{z,w}^{-1} \bE \right] \right| \longrightarrow 0\,,\\
        &\max_{1\le i \le n}\left| n^{-1} \tr\left[   \bP_i^\T  \frac{\partial}{\partial w}(\bP_i\bA_{z,w}\bP_i^\T)^{-1} \bP_i \bE \right]
        -
        n^{-1} \tr\left[   \frac{\partial}{\partial w}\bA_{z,w}^{-1} \bE \right]
        \right|\longrightarrow 0\,.
    \end{align*}
\end{lemma}
\begin{proof}
By the block matrix inversion formula,
\begin{align*}
    \bP_i\bA_{z,w}^{-1}\bP_i^\T = \left( \bP_i \bA_{z,w}\bP_i^\T - {[\bA_{z,w}]_{ii}}^{-1} \bP_i\bA_{z,w}\be_i\be_i\bA_{z,w}\bP_i^\T \right)^{-1} \,. 
\end{align*}
Thus, 
\begin{align*}
    \bP_i\bA_{z,w}^{-1}\bP_i^\T - (\bP_i\bA_{z,w}\bP_i^\T)^{-1}  = \bP_i\bA_{z,w}^{-1}\bP_i^\T \left( {[\bA_{z,w}]_{ii}}^{-1} \bP_i\bA_{z,w}\be_i\be_i\bA_{z,w}\bP_i^\T \right) (\bP_i\bA_{z,w}\bP_i^\T)^{-1} 
\end{align*}
is rank $1$, and clearly has bounded operator norm. Write $\bP_i\bA_{z,w}^{-1}\bP_i^\T - (\bP_i\bA_{z,w}\bP_i^\T)^{-1}=\bq\bq^\T$, so
\begin{align*}
    \max_{1\le i \le n}\left| n^{-1} \tr\left[   \bP_i^\T (\bP_i\bA_{z,w}\bP_i^\T)^{-1} \bP_i \bE \right] - n^{-1} \tr\left[   \bP_i^\T (\bP_i(\bA_{z,w})^{-1}\bP_i^\T) \bP_i \bE \right] \right| 
    &= \max_{1\le i \le n}\left| n^{-1}\bq_i^\T (\bP_i\bE\bP_i^\T)\bq_i \right| \\
    &\le  \max_{1\le i \le n} n^{-1} \|\bq_i\|^2\|\bP_i\bE\bP_i^\T\| \longrightarrow 0 \,.
\end{align*}
Moreover, $\bI - \bP_i^\T \bP_i = \be_i\be_i^\T$ is rank $1$, allowing us to deduce 
\[
\max_{1\le i \le n}\left| n^{-1} \tr\left[   \bP_i^\T \bP_i(\bA_{z,w})^{-1}\bP_i^\T \bP_i \bE \right] - n^{-1}\tr\left[(\bA_{z,w})^{-1}\bE\right] \right|\longrightarrow 0
\]
by a similar argument. This establishes the first claim of the Lemma. The second claim (pertaining to the $w$-derivatives) follows by a similar calculation. 
\end{proof}

Now,
\begin{equation}
    \label{eq:New:3}
    \frac1n \tr[\bA_{z,w}^{-1}\bE] = -\frac{p}{n} + \frac{1}{n}\tr\left[ \bA_{z,w}^{-1}(z\bLambda + w\bI) \right] = \gamma \left( -1 + w\frac1p\tr(\bR_z(w)) + z\frac1p\tr(\bR_z(w)\bLambda) \right) \,.
\end{equation}
Setting $w=0$, by
Propositions~\ref{prop:stieltjes} and \ref{prop:res}, 
\begin{equation}\label{eq:New:2}
\begin{split}
    &\frac1p \tr(\bR_z(0))\overset{d}{=}\frac1p \tr(z\bS-\bE)^{-1} \longrightarrow -\res(z)\,,\\
    & \frac1p \tr(\bR_z(0)\bLambda) \overset{d}{=}\frac1p \tr ((z\bS-\bE)^{-1}\bS) \longrightarrow -\stiel(z)\\
    & \left.\frac{\partial }{\partial w} \frac1p \tr(\bR_z(w)\bLambda)\right|_{w=0}\overset{d}{=} -\frac1p\tr(z\bS-\bE)^{-2}\bS) \longrightarrow -\res'(z) \,.
\end{split}
\end{equation}
Combining (\ref{eq:New:1}) with Lemmas~\ref{lem:New:1}, \ref{lem:New:2} and Eqs. (\ref{eq:New:3}), (\ref{eq:New:2}), along with $\max_{1\le i \le n}|[n^{-1}\bZ\bZ^\T]_{ii}-1|\to 0$, yields $\max_{1\le i \le n} |[\bR_{z}(0)]_{ii}-\rho_i(z)|\to 0$, where
\begin{equation}\label{eq:New:R}
     \frac{1}{\rho_i(z)} = z\lambda_i-1 +\gamma(1+z\stiel(z)) = z(\lambda_i + \sbar(z) )\,.
\end{equation}
The second equality in (\ref{eq:New:R}) uses the relation $\sbar(z) = \gamma\stiel(z) - (1-\gamma)\frac1z  $ between the Stieltjes and the associated transforms. 
Furthermore, differentiating (\ref{eq:New:1}) with respect to $w$ yields
\begin{align*}
    -\frac{\frac{\partial}{\partial w}[\bR_{z}(0)]_{ii}}{[\bR_{z}(0)]_{ii}^2} = 1 - \left.\frac{\partial}{\partial w} ( n^{-1}\be_i^\T \bZ \bZ^\T  \bP_i^\T ){(\bP_i \bA_{z,w}\bP_i^\T)^{-1}} (n^{-1} \bP_i \bZ\bZ^\T \be_i)\right|_{w=0}\,.
\end{align*}
Consequently, $\max_{1\le i \le n} |\frac{\partial}{\partial w}[\bR_{z}(0)]_{ii}-\tilde{\rho}_i(z)|\to 0$ where
\begin{equation}\label{eq:New:Rtilde}
    \tilde{\rho}_i(z) = -R_i(z)^2\left( 1 + \gamma [\res(z) + z\res'(z)] \right) = - \frac{1 + \gamma(\res(z) + z\res'(z))}{z^2} \cdot \frac{1}{(\lambda_i+\sbar(z))^2} \,.
\end{equation}

Equipped with Eqs. (\ref{eq:New:R}) and (\ref{eq:New:Rtilde}), we now conclude the calculation. Starting with $\UpsOne$, 
\begin{align*}
    \UpsOne(z) \asympeq \frac1p \sum_{i=1}^p\rho_i(z)\lambda_i^2 = \frac{1}{z}\cdot \frac1p \sum_{i=1}^p \frac{\lambda_i^2}{\lambda_i+\sbar(z)} 
    = 
    \frac1z \cdot \frac1p \sum_{i=1}^p\left( \lambda_i - \sbar(z) + \frac{(\sbar(z))^2}{\lambda_i+\sbar(z)} \right) \,.
\end{align*}
Note that $p^{-1}\sum_{i=1}^p\lambda_i = p^{-1}\tr(\bS)\asympeq 1$, and $p^{-1}\sum_{i=1}^p \frac{1}{\lambda_i+\sbar(z)} \asympeq \MPStiel_\beta(-\sbar(z))$, where $\MPStiel_\beta(\cdot)$ is the Stieltjes transform of a Marchenko-Pastur law with shape $\beta$. Thus, formula (\ref{eq:New:UpsOneFormula}) is obtained.

Next, 
\begin{align*}
    \UpsTwo(z) \asympeq -\frac1p \sum_{i=1}^p \tilde{\rho}_i(z)\lambda_i^2 
    &= 
    \frac{1 + \gamma(\res(z) + z\res'(z))}{z^2} \cdot  \frac1p \sum_{i=1}^p\frac{\lambda_i^2}{(\lambda_i+\sbar(z))^2} \\
    &= \frac{1 + \gamma(\res(z) + z\res'(z))}{z^2} \cdot  \frac1p \sum_{i=1}^p \left( 1 - \frac{2\sbar(z)}{\lambda_i+\sbar(z)} + (\sbar(z))^2\frac{1}{(\lambda_i + \sbar(z))^2}  \right)\,.
\end{align*}
Observing that ${p^{-1}\sum_{i=1}^p \frac{1}{(\lambda_i + \sbar(z))^2}} \asympeq \MPStiel_\beta'(-\sbar(z))$, we deduce (\ref{eq:New:UpsTwoFormula}).
Thus the computation is concluded.

\end{document}